\documentclass[a4paper,10pt,reqno]{amsart}
          \usepackage{amssymb}
	  \usepackage{amsmath}
          \usepackage{amsfonts}
          \usepackage[english]{babel}
          \usepackage[utf8]{inputenc}
          \usepackage{enumitem}

\usepackage[margin=2.5cm]{geometry}

 \usepackage[unicode,colorlinks,plainpages=false,hyperindex=true,bookmarksnumbered=true,bookmarksopen=false,pdfpagelabels]{hyperref}
 \hypersetup{urlcolor=cyan,linkcolor=blue,citecolor=red,colorlinks=true,bookmarksdepth=2}
 \usepackage{xcolor}
 
 \usepackage{tikz}   
\usetikzlibrary{arrows}  

\makeatletter
\renewcommand*{\p@section}{\,}
\renewcommand*{\p@subsection}{\S\,}
\renewcommand*{\p@subsubsection}{\S\,}
\makeatother

\newtheorem{thm}{Theorem}[section]
\newtheorem{cor}[thm]{Corollary}
\newtheorem{lem}[thm]{Lemma}
\newtheorem{prop}[thm]{Proposition}
\newtheorem{exmp}[thm]{Example}
\newtheorem{rem}[thm]{Remark}
\newtheorem{defn}[thm]{Definition}

\numberwithin{equation}{section}

\newcommand{\CC}{\ensuremath{\mathbb{C}}}
\newcommand{\N}{\ensuremath{\mathbb{N}}}

\newcommand{\Z}{\ensuremath{\mathbb{Z}}}

\newcommand{\tr}{\operatorname{tr}}
\newcommand{\diag}{\operatorname{diag}}

\newcommand{\Hom}{\operatorname{Hom}}
\newcommand{\Der}{\operatorname{Der}}
\newcommand{\Mat}{\operatorname{Mat}}
\newcommand{\Id}{\operatorname{Id}}
\newcommand{\id}{\operatorname{id}}

\newcommand{\Gl}{\operatorname{GL}}
\newcommand{\Orm}{\operatorname{O}}
\newcommand{\Sp}{\operatorname{Sp}}
\newcommand{\Rep}{\operatorname{Rep}}

\newcommand{\Jac}{\operatorname{Jac}}
\newcommand{\op}{\operatorname{op}}
\newcommand{\gl}{\ensuremath{\mathfrak{gl}}}
\newcommand{\g}{\ensuremath{\mathfrak{g}}}
\newcommand{\fs}{\ensuremath{\mathfrak{t}}}
\newcommand{\og}{\ensuremath{\mathfrak{o}}}
\newcommand{\spg}{\ensuremath{\mathfrak{sp}}}
\newcommand{\OO}{\ensuremath{\mathcal{O}}}
\newcommand{\X}{\ensuremath{\mathtt{X}}}
\newcommand{\cF}{\ensuremath{\mathcal{F}}}
\newcommand{\cB}{\ensuremath{\mathcal{B}}}
\newcommand{\tA}{\ensuremath{\mathbf{A}}}
\newcommand{\cH}{\ensuremath{\mathcal{H}}}
\newcommand{\typA}{\ensuremath{\mathtt{A}}}
\newcommand{\typB}{\ensuremath{\mathtt{B}}}
\newcommand{\typC}{\ensuremath{\mathtt{C}}}
\newcommand{\typD}{\ensuremath{\mathtt{D}}}
\newcommand{\Ad}{\operatorname{Ad}}

\newcommand{\dd}{\ensuremath{\mathrm{d}}}
\newcommand{\DDer}{\ensuremath{\mathbb{D}\mathrm{er}}}
\newcommand{\Tnc}{\ensuremath{\mathbb{T}^\ast}}
\newcommand{\fus}{\ensuremath{\mathrm{fus}}}

\newcommand{\intern}{\ensuremath{\mathrm{int}}}
\newcommand\br[1]{\{ #1 \}}

\newcommand\dgal[1]{  \left\{\!\!\left\{#1\right\}\!\!\right\} }
\newcommand\ldb{\{\!\!\{}
\newcommand\rdb{\}\!\!\}}

\newcommand{\sgn}{\operatorname{sgn}}
\newcommand{\fsgn}{\ensuremath{\mathfrak{sgn}}}

\begin{document}

\title{Quasi-Poisson varieties from double quasi-Poisson algebras in types $B,C,D$}

\author{Semeon Arthamonov}
 \address[S.~Arthamonov]{Beijing Institute of Mathematical Sciences and Applications, China}
 \email{arthamonov@bimsa.cn}

\author{Maxime Fairon}
 \address[M.~Fairon]{Université Bourgogne Europe, CNRS, IMB UMR 5584, F-21000 Dijon, France}
 \email{Maxime.Fairon@u-bourgogne.fr}

 
 \begin{abstract}
Double (quasi-)Poisson brackets were introduced on associative algebras by Van den Bergh to induce a (quasi-)Poisson structure on their representation spaces naturally equipped with a $\mathrm{GL}$-action (type $\mathtt{A}$).
If there exists a compatible involutive anti-automorphism on the underlying associative algebras, Olshanski and Safonkin proved that this construction can be upgraded to induce a Poisson structure on twisted representation spaces (types $\mathtt{B},\mathtt{C},\mathtt{D}$).
We provide an analogous result for double quasi-Poisson brackets, and over an arbitrary semisimple base.
We also apply our theory to quivers in order to understand the Poisson structure on twisted (localised multiplicative) quiver varieties.
The formalism permits that different vertices are assigned different types.
As a first application, we recover the framework of Massuyeau and Turaev for Hopf algebras with a Fox pairing,
which induces in particular the Poisson structure of character varieties for the orthogonal or symplectic groups.
As a second application, we introduce a modified Kontsevich system.
 \end{abstract}

\maketitle

 \setcounter{tocdepth}{1} 

\tableofcontents


\section{Introduction}  \label{S:Intro}

Given a unital associative algebra $A$ over the complex field $\CC$ and an integer $N\geq 1$,
one can consider the $N$-th representation space $\Rep(A,N)$ which parametrises representations $\rho:A\to \Mat_N(\CC)$ over $\CC^N$.
Its coordinate ring $\CC[\Rep(A,N)]$ admits a simple presentation in terms of generators $a_{ij}$, with $a\in A$ and $1\leq i,j\leq N$, satisfying natural `matrix relations', and it is equipped with an action of the group $\Gl_N$ corresponding to conjugation of matrices.
It was observed by Van den Bergh \cite{VdB1} that by introducing a \emph{double Poisson bracket} on $A$ as a linear operation
\[
 \dgal{-,-}: A\otimes A \longrightarrow A\otimes A
\]
satisfying particular properties (see \ref{ss:Dbr}), one can induce for any $N$ a Poisson bracket on $\Rep(A,N)$ which is $\Gl_N$-invariant.
The Poisson bracket hence obtained is completely determined by its values on generators, on which it takes the elementary form
\begin{equation} \label{Eq:VdB}
  \br{a_{ij} , b_{kl}} = \dgal{a,b}_{kj,il}, \quad a,b\in A,\quad  1\leq i,j,k,l\leq N,
\end{equation}
where $(c\otimes d)_{kj,il}:=c_{kj}d_{il}\in \CC[\Rep(A,N)]$.
This is an instance of the Kontsevich-Rosenberg principle \cite{KoR}, which has paved the way for a new branch of \emph{noncommutative Poisson geometry}.
Let us emphasise  two important ramifications in this theory.
On the one hand, Van den Bergh had already introduced a generalisation of this formalism over an arbitrary ring (suitable for quivers) and for quasi-Poisson brackets \cite{VdB1,VdB2}.
On the other hand, Massuyeau and Turaev \cite{MT18,Tu}, or recently Olshanski and Safonkin \cite{OS}, have studied the case of representation spaces with a symmetry by another group than $\Gl_N$.
The present paper aims at bridging the gap between these two developments by presenting a uniform point of view encompassing many facets of these two directions.
In particular, our approach induces (quasi-)Poisson brackets on quiver representation varieties for orthogonal and symplectic groups.

\medskip

We start by explaining the easiest situation where one induces Poisson brackets on spaces with an action of the orthogonal group $\Orm_N$.
Fix a double Poisson bracket $\dgal{-,-}$ on $A$.
Following Olshanski and Safonkin \cite{OS}, assume that $A$ is equipped with an involutive anti-homomorphism $\phi:A\to A$
(that is $\phi(ab)=\phi(b)\phi(a)$ and $\phi^2(a)=a$ for all $a,b\in A$) such that $\dgal{-,-}$ is \emph{$\phi$-adapted}, i.e., the following compatibility relation holds:
\begin{equation} \label{Eq:BrComp}
 (\phi\otimes \phi)(\dgal{a,b})=\dgal{\phi(a),\phi(b)}^\circ\,, \qquad \forall a,b\in A.
\end{equation}
Here, we use the transposition of factors $(c\otimes d)^\circ=d\otimes c$ extended linearly to $A\otimes A$.
Then, for any fixed $N\geq 1$, one forms the $2$-sided ideal
\[
 \mathcal{I}^\phi := \langle \phi(a)_{ij} - a_{ji} \mid a\in A,\,\, 1\leq i,j\leq N \rangle \subset  \CC[\Rep(A,N)].
\]
The quotient $\CC[\Rep(A,N)]/ \mathcal{I}^\phi$ corresponds to taking representations $\rho:A\to \Mat_N(\CC)$ such that
$\rho(\phi(a))=\rho(a)^T$ for all $a\in A$.
The generators of the form $a_{ij}$ of $\CC[\Rep(A,N)]$ remain generators of the quotient.
\begin{thm}[\cite{OS}]
Given $(A,\dgal{-,-},\phi)$ as above,
the algebra $\CC[\Rep(A,N)]/ \mathcal{I}^\phi$ is naturally equipped with a Poisson bracket uniquely determined by
($a,b\in A$, $1\leq i,j,k,l\leq N$)
 \begin{equation} \label{Eq:Thm1}
\br{a_{ij} , b_{kl}} = \frac12\dgal{a,b}_{kj,il}+\frac12 \dgal{\phi(a),b}_{ki,jl}.
 \end{equation}
Furthermore, this Poisson bracket is $\Orm_N$-invariant for the action induced by the action of $\Gl_N$ on $\Rep(A,N)$.
\end{thm}
Our first aim is to show that, when $A$ is defined over $A_0=\oplus_{s\in I}\CC e_s$ with multiplication  $e_se_t=\delta_{st} e_s$,
this result can be naturally adapted for an action of a product $\prod_s\Orm_{\alpha_s}$ of orthogonal groups. Furthermore, we show that when $A$ is equipped with a (noncommutative) moment map~\cite{VdB1},
that is an element  $\mu\in A$ decomposing as $\mu=\sum_{s\in I} \mu_s$, $\mu_s\in e_s A e_s$, and satisfying
\begin{equation}  \label{mum}
 \dgal{\mu_s,a}=a e_s \otimes e_s - e_s \otimes e_s a\,, \qquad \forall a\in A,\,s\in I,
\end{equation}
then (up to a translation so that $\mu+\phi(\mu)=0$) the mapping $\rho\mapsto \rho(\mu)$ is a moment map valued in the Lie algebra of the symmetry group $\prod_s\Orm_{\alpha_s}$, cf. Theorem~\ref{Thm:BrPO}.
Moreover, an analogue of the aforementioned statement over $\CC$ was derived in \cite{OS} for the symplectic group $\Sp_N$ (for $N$ even), and we also adapt it over the base $A_0$, see Theorem~\ref{Thm:BrSmp}.

The previous discussion should be seen as a teaser, since our interest lies mainly within double \emph{quasi-}Poisson brackets (cf. \ref{ss:Dbr} for the definition). They were again introduced by Van den Bergh so as to induce a $\Gl_N$-invariant antisymmetric biderivation on $\CC[\Rep(A,N)]$ uniquely determined by \eqref{Eq:VdB}, but the operation fails to satisfy Jacobi identity.
This failure is not arbitrary but it is governed by a trivector originating from the action of $\Gl_N$, see \cite{AKSM} for the original (smooth) theory of quasi-Poisson manifolds.
An explicit construction made by Van den Bergh \cite{VdB1} for quivers allows to understand the Poisson geometry of multiplicative quiver varieties \cite{CBS} using such a (double) quasi-Poisson bracket. Similarly, Massuyeau and Turaev \cite{MT14} induced the Poisson structure of $\Gl_N$-character varieties from double quasi-Poisson brackets on fundamental groups of surfaces.
Our second main aim is to upgrade these constructions for the orthogonal and symplectic groups.
In the easiest situation, we get the following new result.
\begin{thm}
Given $A$ a $\CC$-algebra equipped with a double \emph{quasi-}Poisson bracket $\dgal{-,-}$ and an involutive anti-homomorphism $\phi$ satisfying \eqref{Eq:BrComp},
the algebra $\CC[\Rep(A,N)]/ \mathcal{I}^\phi$ is naturally equipped with a \emph{quasi-}Poisson bracket uniquely determined by
\eqref{Eq:Thm1}, which is
associated with the action of $\Orm_N$ on $\CC[\Rep(A,N)]/ \mathcal{I}^\phi$.
\end{thm}
In full generalities over $A_0$, this (orthogonal) statement can be found in Theorem \ref{Thm:BrqPO}, where we also explain what happens to (multiplicative) moment maps. The corresponding symplectic statement is Theorem \ref{Thm:BrqPsp}.
As an application, we revisit the case of quivers and fundamental groups of surfaces. We also explain how our formalism applied to quivers allows to build holomorphic versions of the spaces recently introduced by Maiza \cite{Maiz26} to build a ``quasi-Poisson TQFT''.

\medskip

Returning to the Poisson setting for a moment, let us note that the application of our results to quivers yields, after performing Hamiltonian reduction, analogues of quiver varieties obtained by reduction from an orthogonal or symplectic symmetry group.
Thus, our work suggests an (yet incomplete) approach for understanding symmetric quiver representations \cite{Bo,DW02,Zu05} and symmetric quiver varieties \cite{Li19,Nak25} through noncommutative Poisson geometry.
Interestingly, the case of representations with actions of both orthogonal and symplectic groups are natural in view of quiver gauge theories \cite{GW09}. This led us to upgrade the previous constructions to the setting with \emph{mixed types}, see Theorems  \ref{Thm:PType} and \ref{Thm:qPType}.
The main difference is that the anti-homomorphism $\phi:A\to A$ is no longer involutive, but its square is a multiple of the identity on each block $e_s A e_t$, the multiple being $+1$ if both vertices $s,t$ are given the same type and $-1$ if they have different types.

Our point of view does not currently allow to consider arbitrary reductive groups acting at the different vertices, e.g., we are unable to mix actions of $\Gl_N$ with those of $\Orm_N,\Sp_N$, or work with $\operatorname{SL}_N$.
Nevertheless, we make a first step towards this aim by showing how our construction agrees with that of Massuyeau and Turaev \cite{MT18} in the case of double quasi-Poisson brackets obtained from Fox pairings on involutive Hopf algebras. In that case, $\phi$ is the antipode.
More precisely, we show in Section \ref{Sec:MT} that the quasi-Poisson bracket induced by our theory reproduces their (ungraded) formulas in the orthogonal and symplectic setting.
Since the work \cite{MT18} allows to take representations associated with the action of an arbitrary reductive group,
an important problem consists in merging their approach with ours. This could resolve the issue of mixing the actions of arbitrary reductive groups.

Finally, we revisit the Kontsevich system \cite{EW}, which is a system of differential equations on the free algebra
$\CC\langle u^{\pm 1},v^{\pm 1}\rangle$ on two (Laurent) generators. It was shown in \cite{Art15,Art17} how this system can be made Hamiltonian with respect to a double quasi-Poisson bracket. Then, going to representation spaces, it induces a Hamiltonian system on the quasi-Poisson double of $\Gl_N$.
We are in position to modify the Hamiltonian function to make it $\phi$-invariant for the (antipode) map $u\mapsto u^{-1}$, $v\mapsto v^{-1}$,
so that it induces a system of differential equations on the quasi-Poisson doubles of $\Orm_N$ and $\Sp_N$, see Theorem \ref{Thm:mKonts}.

\medskip

\textbf{Layout.}
In Section \ref{S:Basics}, the geometric and algebraic notions needed in the rest of the manuscript are reviewed.
In Section \ref{Sec:DPoi}, we explain how to go from a double Poisson bracket to a Poisson bracket with orthogonal or symplectic symmetry.
Furthermore, we apply this formalism to the case of noncommutative cotangent spaces and to quivers.
Analogous results for the quasi-Poisson case are presented in Section \ref{Sec:DqPoi} and constitute the core of our work.
In Section \ref{Sec:Mix}, we explain how to mix types to obtain twisted representation spaces with actions of orthogonal and symplectic groups.
Our formalism is compared with~\cite{MT18} in the case of Fox pairings on Hopf algebras in Section \ref{Sec:MT},
then it is applied to a modified Kontsevich system in Section \ref{Sec:Ksys}.
The paper closes with an appendix gathering technical details.

\medskip

\textbf{Acknowledgements.}
We are grateful to Oleg Chalykh, Taro Kimura, Gwénaël Massuyeau, Nikita Safonkin and Travis Schedler for helpful discussions and remarks.
We thank Gwyn Bellamy for bringing~\cite{Los} to our attention.
The work of S.A. was supported by the Beijing Natural Science Foundation grant IS25025.
The work of M.F. was carried out at IMB which receives support from the EIPHI Graduate School (ANR-17-EURE-0002).



\section{Basics}  \label{S:Basics}

 We work over the complex field $\CC$ to ease the presentation. The reader can take it to be any algebraically closed field of characteristic zero.
All algebras are finitely generated associative unital over $\CC$. Unadorned tensor products are over $\CC$.
In many places, we work relatively over the semi-simple algebra $A_0=\oplus_{s\in I} \CC e_s$, where $e_se_t=\delta_{st}e_s$ and  $\sum_{s\in I}e_s=1$. The basic $\CC$-linear case corresponds to $|I|=1$.


\subsection{Quasi-Poisson geometry}

Let $G$ be a complex reductive algebraic subgroup of $\Gl_N(\CC)$, $N\geq 1$.
The Lie algebra  $\g$ of $G$ is viewed as a subspace of $\gl_N(\CC)$, where it inherits the $\Ad$-invariant non-degenerate symmetric trace form, denoted $\langle -,-\rangle_{\g}$.
We fix a basis $(F_a)_{a}$ of $\g$ and take its dual $(\check{F}_a)_a$ so that $\langle F_a,\check{F}_b\rangle_{\g}=\delta_{ab}$.
We define the $\Ad$-invariant Cartan trivector
\begin{equation} \label{Eq:Cartan3}
 \psi^{\g} :=\frac{1}{12} \sum_{a,b,c} \langle \check{F}_a,[\check{F}_b,\check{F}_c]\rangle_{\g} \, F_a\wedge F_b \wedge F_c \,
 \in\,\wedge^3\g \,.
\end{equation}
For any $\xi \in \g$, the left- and right-invariant vector fields $\xi^L$ and $\xi^R$ on $G$ are given by
\begin{equation} \label{Eq:LRinvVF}
 \xi^L(\mathcal{H})(g)=\left.\frac{d}{dt}\right|_{t=0} \mathcal{H}(g \,e^{t \xi}), \qquad
\xi^R(\mathcal{H})(g)=\left.\frac{d}{dt}\right|_{t=0} \mathcal{H}( e^{t \xi}\, g),
\end{equation}
for any function $\mathcal{H}\in \OO_G(U)$ with $U$ an open neighborhood of $g\in G$.

Let $M$ be an affine complex variety with left action of $G$ denoted $(g,x)\mapsto g\cdot x$.
The infinitesimal action of $\g$ on $M$ assigns to each  $\xi\in \g$ the vector field $\xi_M \in \Der(\OO_M)$ such that
\begin{equation} \label{EqinfVectM}
 \xi_M(f)(x)=\left.\frac{d}{dt}\right|_{t=0} f(\exp(-t\xi)\cdot x)\,,
\end{equation}
for any function $f\in\OO_M(V)$ in an open neighborhood $V$ of $x$.
We extend the map $\xi\mapsto \xi_M$ equivariantly to  $\wedge^k\g \to \wedge^k \Der(\OO_M)$ for $k\geq 1$.
Denote by $\br{-,-}:\OO_M\times \OO_M\to \OO_M$ an antisymmetric biderivation on $M$ which is $G$-invariant, i.e.
for any open $V\subset M$,$f_1,f_2\in\OO_M(V)$, and $g\in G$,
$\br{g\cdot f_1,g\cdot f_2}=g\cdot \br{f_1,f_2}$.
Write $\Jac_{\br{-,-}} : \OO_M^{\times 3}\to \OO_M$ for the associated map (\emph{Jacobiator}) defined through
\begin{equation} \label{Eq:Jac}
 \Jac_{\br{-,-}}(f,g,h) = \br{f,\br{g,h}} + \br{g,\br{h,f}} + \br{h,\br{f,g}},
\end{equation}
for functions $f,g,h\in\OO_M(V)$ in an open neighborhood $V\subset M$.

\begin{defn}
The operation $\br{-,-}$ is a \emph{Poisson bracket} if $\Jac_{\br{-,-}}$ vanishes identically.
It is a \emph{quasi-Poisson bracket} if
\begin{equation} \label{Jac-qP}
 \Jac_{\br{-,-}} = \frac12 \psi^{\g}_M\,.
\end{equation}
We then say that $(M,\br{-,-})$ is an affine (quasi-)Poisson $G$-variety, or (quasi-)Poisson variety for short.
\end{defn}

We consider $G$ acting on itself by the adjoint action, and on the dual $\g^\ast$ of $\g$ by the coadjoint action.
We write $\langle -,-\rangle:\g^\ast\times \g\to \CC$ for the canonical pairing.
\begin{defn}
The affine Poisson variety $(M,\br{-,-})$ is \emph{Hamiltonian} if it admits a \emph{moment map}, that is a $G$-equivariant morphism $\mu:M\to \g^\ast$ satisfying for any $\xi\in \g$,
\begin{equation} \label{momap}
  \br{\langle \mu,\xi\rangle,-}=  \xi_M\,.
\end{equation}
The affine quasi-Poisson variety $(M,\br{-,-})$ is \emph{Hamiltonian} if it admits a \emph{moment map}, that is a $G$-equivariant morphism $\Xi:M\to G$ satisfying for any function $\mathcal{H}$ on $G$,
\begin{equation} \label{Gmomap}
  \br{\Xi^\ast \mathcal{H},-}=  \frac12 \sum_a \Xi^\ast\big((F_a^L+F_a^R)(\mathcal{H})\big)\, (\check{F}_a)_M\,.
\end{equation}
\end{defn}

Quasi-Poisson geometry was originally introduced in the smooth setting \cite{AKSM},
and we shall freely use the notions algebraically for coordinate rings of affine $\CC$-schemes.
For later, we exhibit a choice of dual bases and an expression for the Cartan trivector in the orthogonal and symplectic cases.
We always denote by $E_{ij}$ the elementary matrix of $\gl_N$ with a single nonzero entry equal to $+1$ in position $(i,j)$.

\begin{exmp}[Orthogonal case] \label{Ex:Ortho}
 For $n\geq 2$, note that $F_{ij}=E_{ij}-E_{ji}\in \og(n)$ for all $1\leq i,j\leq n$.
We get dual bases  of $\og(k)$ under the trace pairing by considering
\begin{equation*}
 F_{ij}=E_{ij}-E_{ji}, \quad
 \check{F}_{ij}:=-\frac12 F_{ij},\quad 1\leq i<j\leq n\,.
\end{equation*}
One computes the Lie bracket for $k<l$ and $u<v$ as
\begin{equation*}
\begin{aligned}
{  }
[\check{F}_{kl},\check{F}_{uv}]
=&\frac14 \delta_{ul} F_{kv}-\frac14\delta_{kv} F_{ul}
+\frac14 \delta_{uk}(\delta_{(v<l)} F_{vl}-\delta_{(v>l)} F_{lv})
+\frac14 \delta_{vl}(\delta_{(u<k)} F_{uk}-\delta_{(u>k)} F_{ku})\,.
\end{aligned}
\end{equation*}
By duality of the bases, we directly deduce for $i<j$, $k<l$ and $u<v$,
\begin{equation*}
\langle \check{F}_{ij},  [\check{F}_{kl},\check{F}_{uv}]\rangle_{\og(n)}
=\frac14\left(  \delta_{ul} \delta_{ik}\delta_{jv}-\delta_{kv} \delta_{iu}\delta_{jl}
+ \delta_{uk}\delta_{iv}\delta_{jl} - \delta_{uk} \delta_{il}\delta_{jv}
+ \delta_{vl}\delta_{iu}\delta_{jk} -\delta_{vl} \delta_{ik}\delta_{ju}\right).
\end{equation*}
Substituting these calculations inside \eqref{Eq:Cartan3}, we can write
\begin{equation} \label{Eq:CartO}
\begin{aligned}
 \psi^{\og(n)} &=\frac{1}{12} \sum_{i<j}\sum_{k<l}\sum_{u<v}
\langle \check{F}_{ij},  [\check{F}_{kl},\check{F}_{uv}]\rangle_{\og(n)} \, F_{ij}\wedge F_{kl} \wedge F_{uv}  \\
 &=\frac{1}{8} \sum_{u<i<j} \, F_{ij}\wedge F_{uj} \wedge F_{ui}
 =\frac18\sum_{1\leq,i,j,u\leq n} \, F_{ij} \otimes F_{uj} \otimes F_{ui}\,.
\end{aligned}
\end{equation}
(In this last expression, we consider the element $F_{ab}$ for any $1\leq a,b\leq n$.)
\end{exmp}

\begin{exmp}[Symplectic case] \label{Ex:Symp}
For $n\geq 2$, note that for all $1\leq i,j\leq 2n$,
\begin{equation}
 \label{Eq:SpanSP}
F_{ij}=E_{ij} - \sgn_{2n}(i)\,\sgn_{2n}(j)\, E_{j+n,i+n}\in \spg(2n)\,,
\end{equation}
where indices are understood modulo $2n$ and
\begin{equation}
 \sgn_{2n}: \{1,\ldots,2n\} \to \{-1,+1\}, \qquad \sgn_{2n}(k):= \left\{
 \begin{array}{ll}  +1 & 1  \leq k \leq n, \\ -1 & n+1 \leq k \leq 2n.  \end{array}
\right. \label{Eq:sgn}
\end{equation}
We show in \ref{App:CartSP} that the Cartan trivector of $\spg(2n)$ takes the form
\begin{equation} \label{Eq:CartSp}
\begin{aligned}
 \psi^{\spg(2n)} &
 =\frac{1}{16}\sum_{1\leq,i,j,k\leq 2n} \, (F_{ij} \otimes F_{ki} \otimes F_{jk} - F_{ij} \otimes F_{jk} \otimes F_{ki})\,.
\end{aligned}
\end{equation}
\end{exmp}

We easily deduce an explicit form for the Cartan trivectors of products of orthogonal and symplectic groups.
E.g. for $\alpha=(\alpha_s)\in \N^I$, $I=\{1,\ldots,|I|\}$,
we see $\og_\alpha:=\og(\alpha_1)\times \ldots \times \og(\alpha_{|I|})$ as a subgroup of $\gl_N$, $N=\alpha_1+\ldots+\alpha_{|I|}$, and we obtain dual bases by taking (cf. Example \ref{Ex:Ortho})
\begin{equation*}
 F_{ij}^{(s)}=E_{ij}-E_{ji}, \quad
 \check{F}^{(s)}_{ij}:=-\frac12 F^{(s)}_{ij}, \quad
 \alpha_1+\ldots+\alpha_{s-1}+1\leq i<j\leq \alpha_1+\ldots+\alpha_{s-1}+\alpha_s,\quad s=1,\ldots,|I|\,.
\end{equation*}
The Cartan trivector for $\og_\alpha$ is obtained by summing \eqref{Eq:CartO} for each copy of $\og(\alpha_s)$ in $\og_\alpha \subset \gl_N$.

\subsection{(Twisted) Representation spaces} \label{ss:Rep}

Let $A$ be an algebra considered over $A_0=\oplus_{s\in I} \CC e_s$.
Fix a dimension vector $\alpha\in \N^I$.
The ($A_0$-relative) \emph{affine representation space} of $A$ with dimension vector $\alpha\in \N^I$, denoted $\Rep(A,\alpha)$, is the affine scheme parametrising representations $\rho$ of $A$ on $\CC^N$, $N=\sum_s \alpha_s$, such that
the idempotent $e_s$, $s\in \{1,\ldots,|I|\}$, is represented by the $s$-th identity block
\begin{equation} \label{IdemRep1}
\rho(e_s) = \diag(0_{\alpha_1},\ldots,0_{\alpha_{s-1}},\Id_{\alpha_s},0_{\alpha_{s+1}},\ldots, 0_{\alpha_{|I|}})\,.
\end{equation}
For later reference, note that the indices of matrix entries in that block belong to
\begin{equation} \label{Eq:sIndices}
 \mathtt{R}_s = \{\alpha_{1}+\cdots+\alpha_{s-1}+1\,,\,\ldots\,,\, \alpha_{1}+\cdots+\alpha_{s-1}+\alpha_{s}\}\,.
\end{equation}
We can equivalently describe the coordinate ring $\CC[\Rep(A,\alpha)]$ using the generators
\begin{equation*}
 a_{ij}, \qquad \text{with }a\in A, \quad 1\leq i,j\leq N,
\end{equation*}
satisfying linearity relations in $a$ and the matrix rule $(ab)_{ij}=\sum_{r=1}^N a_{ir}b_{rj}$, together with the algebraic counterpart of \eqref{IdemRep1}:
\begin{equation} \label{IdemRep2}
(e_s)_{ij}=\left\{
\begin{array}{ll}
 \delta_{ij}&\text{if }i,j\in \mathtt{R}_s \,\eqref{Eq:sIndices}, \\
 0& \text{else}.
\end{array}
\right.
\end{equation}

For any $a\in A$, consider $\X(a):\Rep(A,\alpha)\to \gl_N$, $\X(a)(\rho)=\rho(a)$, returning the matrix representing $a$.
The group ${\Gl_\alpha}:=\prod_s \Gl_{\alpha_s}(\CC)$ acts by change of basis according to the decomposition
\begin{equation} \label{DecoCNrep}
 \CC^N=\oplus_{s\in I}\CC^{\alpha_s}
\end{equation}
relative to \eqref{IdemRep1}.
In terms of the coordinate ring, the action reads \cite[\S7.1]{VdB1},
\begin{equation}  \label{Infgrp}
 g\cdot \X(a)=g^{-1}\X(a) g, \qquad g\in {\Gl_\alpha}.
\end{equation}
This action differentiates to an action by derivations of ${\gl_\alpha}:=\prod_s \gl_{\alpha_s}(\CC)$ as
\begin{equation} \label{InfAct}
 \xi\cdot \X(a)=[\X(a) ,\xi], \qquad \xi\in {\gl_\alpha}.
\end{equation}
We note that the invariant ring $\CC[\Rep(A,\alpha)]^{\Gl_\alpha}$ is generated \cite{CB2} by the image of
\begin{equation} \label{Eq:trmap}
 \tr : A\to \CC[\Rep(A,\alpha)]\,, \quad \tr(a)=\sum_{1\leq i\leq N} \X(a)_{ii}\,,
\end{equation}


Following \cite{OS}, a pair $(A,\phi)$ is an \emph{involutive algebra} if $\phi:A\to A$ is a $\CC$-linear anti-automorphism such that $\phi\circ\phi = \id_A$.
If $A$ is an algebra over $A_0$, we further assume that $\phi(e_s)=e_s$ for all $s\in I$.
Next, for the twisted case considered in \cite{OS}, we need to fix an involution $\tau$ on ${\gl_N}$.
Moreover, $\tau$ must preserve a (symmetric or antisymmetric) non-degenerate bilinear form on $\CC^N$ which restricts to each subspace $\CC^{\alpha_s}$ in the decomposition \eqref{DecoCNrep}.
Without loss of generality, we are in one of two cases:
\begin{itemize}
 \item[$\bullet$] $\tau$ is the matrix transposition $\xi\mapsto \xi^T$
(orthogonal type $\typD$ on $\gl_{\alpha_s}(\CC)$ if $\alpha_s$ even, and $\typB$ if $\alpha_s$ odd);
 \item[$\bullet$] $\tau$ is given by the linear map $\xi\mapsto \Omega \xi ^T  \Omega^T$ (symplectic type $\typC$ with all $\alpha_s$ even), where
 \begin{equation}\label{Eq:omega}
  \Omega=\diag(\Omega_{\alpha_1},\ldots,\Omega_{\alpha_{|I|}}) \quad \text{ with }
\Omega_{2d}=\left(\begin{array}{cc} 0_{d} & \Id_{d} \\ -\Id_{d}  & 0_{d}  \end{array} \right) \quad (d\geq 1).
 \end{equation}
\end{itemize}
The twisted representation space is then defined by
\begin{equation*}
 \Rep^{\phi,\tau}(A,\alpha) = \{\rho \in \Rep(A,\alpha) \mid \rho \circ \phi = \tau \circ \rho\}\,.
\end{equation*}
In terms of the coordinate ring, one has
$\CC[\Rep^{\phi,\tau}(A,\alpha)]=\CC[\Rep(A,\alpha)]/I^{\phi,\tau}$ for the ideal
\begin{equation} \label{TwIdeal}
I^{\phi,\tau} := \langle(\phi(a)|E_{ij})-(a|\tau(E_{ij}))  \,|\,
 a\in A,\, 1\leq i,j\leq N\rangle \,,
\end{equation}
where we denote $a_{ij}\in \CC[\Rep(A,\alpha)]$ as $(a|E_{ij})$.
The twisted representation space  is naturally equipped with an action of the corresponding group\footnote{In types $\typB$ and $\typD$, we can use factors $\mathrm{SO}(\alpha_s)$ instead of $\Orm(\alpha_s)$ for some $s\in I$. To simplify the exposition, we will work with all copies being $\Orm(\alpha_s)$; changing some factors to $\mathrm{SO}(\alpha_s)$ does not pose any difficulty.\label{ftnote:TypeO}} ($\Orm_\alpha:=\prod_{s\in I} \Orm(\alpha_s)$ or $\Sp_\alpha:=\prod_{s\in I} \Sp(\alpha_s)$) and its Lie algebra
($\og_\alpha:=\prod_{s\in I} \og(\alpha_s)$ or $\spg_\alpha:=\prod_{s\in I} \spg(\alpha_s)$) induced by \eqref{Infgrp} and \eqref{InfAct}, respectively.

\begin{rem}
 On a (twisted) representation space, we use the following notation:
for $d'\otimes d''\in A^{\otimes 2}$,
 $b'\otimes b''\otimes b''' \in A^{\otimes 3}$,
 and $1\leq i,j,k,l,u,v\leq N$,
\begin{equation} \label{Eq:TensIndex}
 (d'\otimes d'')_{ij,kl}:=d'_{ij}\,d''_{kl}, \quad
 (b'\otimes b''\otimes b''')_{ij,kl,uv}:= b'_{ij}\,b''_{kl}\, b'''_{uv}\,.
\end{equation}
\end{rem}

\subsection{Double brackets}

We recall the key objects from \cite{VdB1} and we deduce some elementary results.

\subsubsection{First constructions} \label{ss:Dbr}
Fix an algebra $A$ over $A_0=\oplus_{s\in I} \CC e_s$.
Its multiplication is simply denoted by concatenation.
We view $A\otimes A$ as an algebra in the obvious way. It is equipped with the outer and the inner $A$-bimodule structures, which commute and are given by
\begin{equation*}
  a \,d\, b=a d' \otimes d'' b \,\,(\text{outer}); \qquad
 a \ast d\ast b=d'b \otimes ad'' \,\, (\text{inner}),
\end{equation*}
for $a,b\in A$ and $d\in A^{\otimes 2}$ denoted $d=d'\otimes d''$ using a Sweedler-type notation.
Permutation of tensor factors is given by the $\CC$-linear map
\begin{equation*}
 (-)^\circ : A^{\otimes 2}\to A^{\otimes 2}, \quad
 d^\circ = d''\otimes d'.
\end{equation*}

\begin{defn}
 A \emph{double bracket} on $A$ is a $\CC$-linear map  $\dgal{-,-}:A\otimes A \to A \otimes A$  satisfying
\begin{subequations}
 \begin{align}
 \dgal{a,b}&=-\dgal{b,a}^\circ \quad \forall a,b\in A\,, \quad &&\text{(cyclic antisymmetry)}
\label{Eq:cycanti} \\
 \dgal{a,bc}&=\dgal{a,b}c+b\dgal{a,c} \quad \forall a,b,c\in A\,, \quad &&\text{(right derivation rule)}
 \label{Eq:outder} \\
 \dgal{bc,a}&=\dgal{b,a}\ast c+b\ast\dgal{c,a} \quad \forall a,b,c\in A\,, \quad &&\text{(left derivation rule)}
 \label{Eq:inder}
 \end{align}
\end{subequations}
together with $A_0$-linearity : $\dgal{e,a}=0=\dgal{a,e}$ for all $a\in A$, $e\in A_0$.
\end{defn}

Given a double bracket, consider its left extension
\begin{equation}
 \dgal{-,-}_L: A\otimes A^{\otimes 2} \to A^{\otimes 3}, \qquad
 \dgal{a,d}_L=\dgal{a,d'}\otimes d''\,.
\end{equation}
Then, define an operation $\dgal{-,-,-} : A^{\otimes 3}\to A^{\otimes 3}$ (the \emph{double Jacobiator}) by setting
\begin{equation}
\label{Eq:TripBr}
\begin{aligned}
  \dgal{a,b,c}=&\dgal{a,\dgal{b,c}}_L
  +\tau_{(123)}\dgal{b,\dgal{c,a}}_L
  +\tau_{(123)}^2\dgal{c,\dgal{a,b}}_L \,, \quad \forall a,b,c\in A.
\end{aligned}
\end{equation}
Here, we used $\tau_{(123)}:\,A^{\otimes 3}\to A^{\otimes 3}$,
$\tau_{(123)}(a_1\otimes a_2\otimes a_3)=a_{3}\otimes a_1 \otimes a_{2}$.
The operation \eqref{Eq:TripBr} is a \emph{triple bracket}~\cite[\S2.3]{VdB1}, i.e. a $\CC$-trilinear map satisfying the cyclic symmetry
\begin{equation} \label{Eq:TriAnti}
\tau_{(123)}\circ \dgal{-,-,-}\circ \tau_{(123)}^{-1}
=\dgal{-,-,-}\,,
\end{equation}
and which is a derivation in its last argument for the outer $A$-bimodule structure on $A^{\otimes 3}$.
(It is also a derivation in its two other arguments for commuting $A$-bimodule structures on $A^{\otimes 3}$.)

\begin{defn}
 A double bracket on $A$ is \emph{Poisson} if the double Jacobiator $\dgal{-,-,-}$ given by \eqref{Eq:TripBr} is identically zero.
In that case, we say that $(A,\dgal{-,-})$ is a \emph{double Poisson algebra}.

\noindent A double bracket on $A$ is \emph{quasi-Poisson} if the double Jacobiator $\dgal{-,-,-}$ satisfies for any $a,b,c\in A$,
\begin{equation}
   \begin{aligned} \label{qPabc}
    \dgal{a,b,c}=&\frac14 \sum_{s\in I} \Big(
c e_s a \otimes e_s b \otimes e_s  - c e_s a \otimes e_s \otimes b e_s - c e_s \otimes a e_s b \otimes e_s
+ c e_s \otimes a e_s \otimes b e_s \\
&\qquad \quad - e_s a \otimes e_s b \otimes e_s c + e_s a \otimes e_s \otimes b e_s c + e_s \otimes a e_s b \otimes e_s c - e_s \otimes a e_s \otimes b e_s c \Big)\,.
  \end{aligned}
\end{equation}
In that case, we say that $(A,\dgal{-,-})$ is a \emph{double quasi-Poisson algebra}.
\end{defn}

\begin{rem} \label{Rem:HoP}
The operation $\mathrm{m}\circ \dgal{-,-}:A\otimes A \to A$ obtained by multiplication of the two tensor factors gives rise
to a Lie bracket $\br{-,-}_\sharp$ on the vector space $A/[A,A]$ \cite[\S2.4]{VdB1}.
\end{rem}

\begin{rem} \label{Rem:TripBr}
If we consider the right extension $\dgal{-,-}_R: A\otimes A^{\otimes 2} \to A^{\otimes 3}$,
$\dgal{a,d}_R=d'\otimes \dgal{a,d''}$ of a double bracket, as noticed in \cite{FMC}
the double Jacobiator can be equivalently defined as
\begin{equation}
\label{Eq:TripBr2}
\begin{aligned}
  \dgal{a,b,c}=&-\dgal{b,\dgal{a,c}}_R
  -\tau_{(123)}\dgal{c,\dgal{b,a}}_R
  -\tau_{(123)}^2\dgal{a,\dgal{c,b}}_R \,.
\end{aligned}
\end{equation}
\end{rem}

\begin{defn}
If $(A,\dgal{-,-})$ is a double Poisson algebra, an element $\mu \in A$ is called a \emph{moment map} if it can be decomposed as $\mu=\sum_{s\in I} \mu_s$ with $\mu_s=e_s\mu e_s$, so that \eqref{mum} holds.
We say that $(A,\dgal{-,-},\mu)$ is a \emph{double Hamiltonian algebra}.

\noindent
If $(A,\dgal{-,-})$ is a double quasi-Poisson algebra, an invertible element $\Phi \in A^\times$ is called a \emph{moment map} if it can be decomposed as $\Phi=\sum_{s\in I} \Phi_s$ with $\Phi_s=e_s\Phi e_s$, so that for any $a\in A$ and $s\in I$,
\begin{equation} \label{Phim}
 \dgal{\Phi_s,a}=\frac12 (ae_s\otimes \Phi_s-e_s \otimes \Phi_s a +  a \Phi_s \otimes e_s-\Phi_s \otimes e_s a)\,.
\end{equation}
We say that $(A,\dgal{-,-},\Phi)$ is a \emph{double quasi-Hamiltonian algebra}.
\end{defn}

\begin{rem}
A double bracket is uniquely determined by its evaluation on generators.
Similarly, the double Jacobiator $\dgal{-,-,-}$ \eqref{Eq:TripBr} is uniquely determined by its value on generators.
Thus, the equality $\dgal{a,b,c}=0$ (resp. \eqref{qPabc}) for any $a,b,c\in A$ follows from establishing this equality
when $a,b,c$ are generators.
It also suffices to check the equalities \eqref{mum} and \eqref{Phim} for any generator $a$.
\end{rem}

Write $A^{\op}$ for the opposite algebra of $A$, i.e. with multiplication $a\cdot_{\op}b=ba$ for all $a,b\in A$.
The following two new results will be useful.

\begin{lem} \label{lem:Op-P}
If $(A,\dgal{-,-})$ is a double Poisson algebra, then
\begin{equation} \label{Eq:dgalOp}
 \dgal{-,-}_{\op}: A^{\op}\otimes A^{\op} \to A^{\op}\otimes A^{\op}, \quad
 \dgal{a,b}_{\op}:=\dgal{a,b}^\circ,
\end{equation}
is a double Poisson bracket on $A^{\op}$.
Moreover, if $\mu\in A$ is a moment map, then
$(A^{\op},\dgal{-,-}_{\op},-\mu)$ is a double Hamiltonian algebra.
\end{lem}
\begin{rem}
Assume that $a_1,\ldots,a_\ell$ are generators of $A$, and write $a_1',\ldots,a_\ell'$ for their images in $A^{\op}$.
If $\mu=\sum_{j} a_{i_1}\ldots a_{i_{k_j}} \in A$ for indices $i_k\in \{1,\ldots,\ell\}$,
then writing $-\mu \in A^{\op}$ means that we consider the expression
$-\mu = -\sum_{j} a_{i_{k_j}}' \cdot_{\op} \ldots \cdot_{\op} a_{i_1}'$ in terms of generators of $A^{\op}$.
\end{rem}
\begin{proof}[Proof (of Lemma \ref{lem:Op-P})]
It is clear that \eqref{Eq:dgalOp} is well-defined as a $\CC$-bilinear map, and it is $A_0$-linear.
Cyclic antisymmetry \eqref{Eq:cycanti} for $\dgal{-,-}_{\op}$ is direct since $\dgal{-,-}$ has this property.
The derivation rules \eqref{Eq:outder}--\eqref{Eq:inder} for $\dgal{-,-}_{\op}$
are equivalent to those rules for $\dgal{-,-}$. E.g.
\begin{align*}
\dgal{a,b\cdot_{\op} c}_{\op}&= \dgal{a,cb}^\circ = (c\dgal{a,b}+\dgal{a,c}b)^\circ \\
&=c\ast \dgal{a,b}^\circ + \dgal{a,c}^\circ \ast b
= \dgal{a,b}_{\op} \cdot_{\op} c + b \cdot_{\op} \dgal{a,c}_{\op}\,.
\end{align*}
Next, noticing the equality
\begin{align*}
 \ldb a,\dgal{b,c}_{\op}\rdb_{\op,L}
 =\tau_{(12)}\ldb  a,\dgal{b,c}^\circ \rdb_{L}
 =\tau_{(13)}\dgal{a,\dgal{b,c}}_{R} \,,
\end{align*}
the vanishing of the double Jacobiator \eqref{Eq:TripBr} on $A^{\op}$ follows from
\begin{equation}
 \begin{aligned} \label{Eq:TripOp}
  \dgal{a,b,c}_{\op}=&\ldb  a,\dgal{b,c}_{\op}\rdb_{\op,L}
  +\tau_{(123)}\ldb b,\dgal{c,a}_{\op}\rdb_{\op,L}
  +\tau_{(123)}^2 \ldb c,\dgal{a,b}_{\op}\rdb_{\op,L} \\
  =&\tau_{(13)}\left( \dgal{a,\dgal{b,c}}_{R} +\tau_{(123)}^2 \dgal{b,\dgal{c,a}}_{R}
+\tau_{(123)} \dgal{c,\dgal{a,b}}_{R} \right) \\
=&-\tau_{(13)} \dgal{b,a,c}\,,
 \end{aligned}
\end{equation}
where we used \eqref{Eq:TripBr2} in the last equality.
For the moment map property, one gets from  \eqref{mum}
\begin{equation*}
 -\dgal{\mu_s,a}_{\op}=-(a e_s \otimes e_s - e_s \otimes e_s a)^\circ
 = a \cdot_{\op}e_s \otimes e_s - e_s \otimes e_s \cdot_{\op}a\,,
\end{equation*}
which holds for any $a\in A^{\op}$, as desired.
\end{proof}

\begin{lem} \label{lem:Op-qP}
If $(A,\dgal{-,-})$ is a double quasi-Poisson algebra, then $\dgal{-,-}_{\op}$ defined as in \eqref{Eq:dgalOp}
is a double quasi-Poisson bracket on $A^{\op}$.
Moreover, if $\Phi\in A^\times$ is a moment map, then
$(A^{\op},\dgal{-,-}_{\op},\Phi^{-1})$ is a double quasi-Hamiltonian algebra.
\end{lem}
\begin{proof}
From the proof of Lemma \ref{lem:Op-P}, we already know that $\dgal{-,-}_{\op}$ is a double bracket, hence we have to check \eqref{qPabc}.
In view of \eqref{Eq:TripOp} and the property \eqref{qPabc} for $\dgal{-,-}$, one has
\begin{equation*}
   \begin{aligned}
    &\dgal{a,b,c}_{\op}=-\tau_{(13)} \dgal{b,a,c} \\
=&
    \frac14 \sum_{s\in I} \tau_{(13)} \Big(
- c e_s b \otimes e_s a \otimes e_s  + c e_s b \otimes e_s \otimes a e_s + c e_s \otimes b e_s a \otimes e_s
- c e_s \otimes b e_s \otimes a e_s \\
&\qquad \quad + e_s b \otimes e_s a \otimes e_s c - e_s b \otimes e_s \otimes a e_s c - e_s \otimes b e_s a \otimes e_s c
+ e_s \otimes b e_s \otimes a e_s c \Big)\\
=&
\frac14 \sum_{s\in I} \Big(
- e_s \otimes a \cdot_{\op} e_s  \otimes b \cdot_{\op} e_s \cdot_{\op} c
+ e_s \cdot_{\op} a \otimes e_s \otimes  b \cdot_{\op} e_s \cdot_{\op} c \\
& \qquad + e_s \otimes a \cdot_{\op} e_s \cdot_{\op} b \otimes e_s \cdot_{\op} c
- e_s \cdot_{\op} a \otimes e_s \cdot_{\op} b \otimes e_s \cdot_{\op} c
 + c \cdot_{\op}  e_s  \otimes a\cdot_{\op} e_s \otimes b \cdot_{\op}  e_s  \\
 &\qquad -  c \cdot_{\op} e_s \cdot_{\op} a \otimes e_s \otimes b \cdot_{\op}  e_s
 - c \cdot_{\op}  e_s \otimes a\cdot_{\op}  e_s \cdot_{\op}  b \otimes e_s
+c\cdot_{\op} e_s \cdot_{\op} a \otimes e_s\cdot_{\op} b \otimes e_s \Big)\,.
  \end{aligned}
\end{equation*}
This is exactly the right-hand side (RHS) of \eqref{qPabc} written in $A^{\op}$.

For the moment map property, we get from \eqref{Eq:inder} and \eqref{Phim} in $A$
\begin{equation} \label{Phim-inv}
 \dgal{\Phi_s^{-1},a}
 =-\Phi_s^{-1}\ast \dgal{\Phi_s,a} \ast \Phi_s^{-1}
 =-\frac12 (a\Phi_s^{-1}\otimes e_s-\Phi_s^{-1} \otimes e_s a +  a e_s \otimes \Phi_s^{-1} -e_s \otimes \Phi_s^{-1} a)\,.
\end{equation}
This equality entails in $A^{\op}$,
\begin{align*}
\dgal{\Phi_s^{-1},a}_{\op} &= \dgal{\Phi_s^{-1},a}^\circ
=\frac12 (-e_s \otimes \Phi_s^{-1}\cdot_{\op}a+ a\cdot_{\op}e_s  \otimes \Phi_s^{-1}
- \Phi_s^{-1} \otimes e_s\cdot_{\op}a + a\cdot_{\op}\Phi_s^{-1} \otimes e_s)
\end{align*}
which is precisely the defining condition \eqref{Phim}.
\end{proof}

\subsubsection{Fusion of algebras} \label{ss:Fus}

We recall the notion of fusion following \cite[\S 2.5]{VdB1}, see also \cite[\S2]{F2}.
Fix an algebra $A$ over $A_0=\oplus_{s\in I} \CC e_s$  with $I=\{1,2,\ldots, |I|\}$ ($|I|\geq 2$).
The $\CC$-algebra $\Mat_2(\CC)$ is generated by $e_1=e_{11},e_{12},e_{21},e_2=e_{22}$ with $e_{st}e_{uv}=\delta_{tu}e_{sv}$.
Let $\mu=1-e_1-e_2$ and $\epsilon=1-e_2$, then introduce
\begin{align*}
 A^{(1,2)}:=&\,A \ast_{\CC e_1 \oplus \CC e_2 \oplus \CC \mu} (\Mat_2(\CC)\oplus \CC \mu) = A \ast_{A_0} A_0^{(1,2)}\,; \\
 A^f:=&\, \epsilon A^{(1,2)} \epsilon\,.
\end{align*}
The algebra $A^{(1,2)}$ is the extension algebra of $A$ along the pair $(e_1,e_2)$.
The fusion algebra $A^f$ of $A$ (obtained by fusing $e_2$ onto $e_1$) is the algebra
obtained from $A^{(1,2)}$ by discarding elements of $e_2 A^{(1,2)} + A^{(1,2)} e_2$.
The algebra $A^f$ is an algebra over $A_0^f$, with  $A_0^f=\epsilon A^{(1,2)}_0 \epsilon\simeq \oplus_{s\in I\setminus \{2\}} \CC e_s$.
Any element of $A^f$ can be written in terms of generators of the form
\begin{subequations}
  \begin{align}
   (\text{first kind})&\qquad\qquad a=t\,, \quad \text{with } t \in \epsilon A \epsilon\,, \label{type1}\\
   (\text{second kind})&\qquad\qquad a=e_{12}u\,, \quad \text{with } u \in e_2 A \epsilon\,,\label{type2} \\
   (\text{third kind})&\qquad\qquad a=v e_{21}\,, \quad \text{with } v \in \epsilon A e_2\,, \label{type3} \\
   (\text{fourth kind})&\qquad\qquad a=e_{12} w e_{21}\,, \quad \text{with } w \in e_2 A e_2\,. \label{type4}
  \end{align}
\end{subequations}
Now, assume that $A$ is equipped with a double bracket $\dgal{-,-}$.
It naturally extends to a double bracket on $A^{(1,2)}$, and then it restricts to $A^f$ where it is $A_0^f$-linear.
We still denote it as $\dgal{-,-}$ on $A^f$.
\begin{prop}[\cite{F2,VdB1}]
 \label{PropIndbr}
If $\dgal{-,-}$ is Poisson on $A$, then the induced double bracket on $A^f$ is Poisson.
Furthermore, if $\mu=\sum_{s\in I} \mu_s\in A$ is a moment map, then $(A^f,\dgal{-,-})$ has the moment map
\[
 \mu^f= \mu_1+e_{12}\mu_2 e_{21} +\sum_{s\neq 1,2} \mu_s \in A^f \,.
\]
If $\dgal{-,-}$ is quasi-Poisson on $A$, then $A^f$ has an $A_0^f$-linear double quasi-Poisson bracket given by
\begin{equation} \label{dgalf}
  \dgal{-,-}^f:= \dgal{-,-} + \dgal{-,-}_{\fus}\,,
\end{equation}
where  the first double bracket is the one induced in $A^f$,
and the second double bracket $\dgal{-,-}_{\fus}$ is the unique double bracket satisfying the relations in
\cite[Lem.~2.19]{F2}.
Furthermore, if $\Phi=\sum_{s\in I} \Phi_s\in A^\times$ is a moment map, then $(A^f,\dgal{-,-}^f)$ has the moment map
\[
 \Phi^f= \Phi_1\,e_{12}\Phi_2 e_{21} +\sum_{s\neq 1,2} \Phi_s \,\, \in (A^f)^\times\,.
\]
\end{prop}

In full generalities, we say that an algebra is obtained by fusion if it can be constructed by a sequence of fusions of 2 idempotents as above.
Given $A,A'$ algebras over $A_0,A_0'$, respectively, $A\oplus A'$ is an algebra over $A_0\oplus A_0'$.
Then, we can perform fusion of idempotents of $A_0$ with those in $A_0'$.

\begin{exmp} \label{Ex:Fus}
 Let $A$ be an algebra over $A_0$ endowed with a double bracket $\dgal{-,-}$.
Then, one gets a double bracket $\dgal{-,-}_{\op}$ on $A_0^{\op}$ by considering \eqref{Eq:dgalOp}
 (this is the first part of Lemma \ref{lem:Op-P}).
We can form the double $\overline{A}:= A \ast_{A_0} A^{\op}$ of $A$ by recursively performing fusion in $A\oplus A^{\op}$ of the idempotents
$(e_j,0),(0,e_j)\in A_0\oplus A_0$, $j\in I$.
By construction, we get a double bracket $\dgal{-,-}^{-}$ on $\overline{A}$ which is the unique $A_0$-linear double bracket such that
\[
 \dgal{a,b}^{-}=\dgal{a,b}, \quad \dgal{a,a'}^-=0, \quad \dgal{a',b'}^- = \dgal{a',b'}_{\op},
\]
for $a,b\in A$ and $a',b'\in A^{\op}$.
Note that if $Q$ is a quiver, cf. \ref{sss:P-exmp}, then the fusion algebra $\overline{\CC Q}$ of the path algebra $\CC Q$ of $Q$ is just the path algebra $\CC \overline{Q}$ of the double quiver $\overline{Q}$.
\end{exmp}

Next, note that any \emph{anti}-morphism $\phi:A\to A$ preserving $A_0$ can be extended to
$\widetilde{\phi}:A^{(1,2)}\to A^{(1,2)}$
by acting through $e_{uv}\mapsto e_{vu}$ on $\Mat_2(\CC)$.
By restricting to $A^f$, we get an anti-morphism $\phi_f$ on $A^f$.
By construction, $\phi_f$ acts on the $4$ kinds of generators as
\begin{equation} \label{Eq:phiExt}
  \phi_f(t)=\phi(t), \quad \phi_f(e_{12}u)=\phi(u)e_{21},
 \quad \phi_f(ve_{21})=e_{12}\phi(v),
 \quad \phi_f(e_{12}we_{21})=e_{12}\phi(w)e_{21},
\end{equation}
where $t\in \epsilon A \epsilon$, $u\in e_2 A \epsilon$, $v\in \epsilon A e_2$ and $w\in e_2Ae_2$.
For later purpose, remark that if $\phi$ is involutive, i.e. $\phi^2=\id_A$, then $\phi_f$ is also involutive.

\begin{lem} \label{Lem:Tw-Rep}
Assume that $\alpha_1=\alpha_2$ and let $\alpha^f=(\alpha_1,\alpha_3,\ldots,\alpha_{|I|})$.
Define the endomorphism $\tau_f$ on $\gl_{N-\alpha_1}$, $\xi \mapsto \xi^T$.
Consider $\Orm_{\alpha^f}\hookrightarrow \Orm_{\alpha}$ where the embedding on the first factor is diagonal
and on the other factors is the identity.
Then, there is a canonical isomorphism between $\Rep^{\phi,\tau}(A,\alpha)$ and $\Rep^{\phi_f,\tau_f}(A^f,\alpha^f)$ such that
the induced $\Orm_{\alpha^f}$ action on $\Rep^{\phi,\tau}(A,\alpha)$ is obtained by restriction from the $\Orm_{\alpha}$-action.

\noindent Furthermore, the same statement holds in the symplectic case using the embedding $\Sp_{\alpha^f}\hookrightarrow \Sp_{\alpha}$ and the endomorphism $\tau_f:\xi \mapsto \Omega^f \xi^T (\Omega^f)^T$,
with $\Omega^f=\diag(\Omega_{\alpha_2},\ldots,\Omega_{\alpha_{|I|}})$.
\end{lem}
\begin{proof}
This is the twisted analogue of \cite[\S7.10]{VdB1}.
 It suffices to identify the coordinate rings.
\end{proof}


\section{From double Poisson to Poisson in types \texorpdfstring{$\typB,\typC,\typD$}{B,C,D}}  \label{Sec:DPoi}

We shall generalise the construction of Olshanski and Safonkin \cite[\S4.2]{OS}
in the situation where the algebra $A$ is taken over the semi-simple algebra $A_0=\oplus_{s\in I} \CC e_s$ and the double bracket is $A_0$-linear.

 Given $(A,\phi)$ an involutive algebra (cf. Sect.~\ref{S:Basics}) endowed with a double bracket $\dgal{-,-}$, we say that $\dgal{-,-}$ is $\phi$-adapted, or that $\phi$ is bracket-compatible, if \eqref{Eq:BrComp} holds.
Note that condition \eqref{Eq:BrComp} only needs to be checked on generators of $A$, cf. \cite[Lem.~4.1]{OS}.
Examples of pairs $(\dgal{-,-},\phi)$ satisfying \eqref{Eq:BrComp} are presented in \cite[Sect.~4]{OS}.
New cases can be obtained from a general `cotangent lift' construction given in \ref{ss:NCcot}, which will be applied to quivers.
Before diving into this theory, we give an example of bracket-compatibility and state a preliminary result.

\begin{exmp}
 Fix $g,r\geq 0$ not both zero.
The free algebra
$$A_{g,r}=\CC\langle x_1,y_1,\ldots,x_g,y_g,z_1,\ldots,z_r\rangle$$
is equipped with a double Poisson bracket uniquely determined by
\begin{equation} \label{Eq:dbr-Lin}
 \dgal{x_i,y_j}=- \dgal{y_j,x_i}=\delta_{ij} \,1\otimes 1, \quad
 \dgal{z_k,z_\ell}= \delta_{k\ell}\,(z_k\otimes 1 - 1 \otimes z_k),
\end{equation}
and being zero on any other pair of generators. This operation appeared in connection to higher genera Kashiwara-Vergne problems \cite{AKKN}.
The involutive anti-automorphism $\phi:A_{g,r}\to A_{g,r}$ defined through
\begin{equation*}
 \phi(x_i)=-x_i, \quad \phi(y_i)=-y_i,\quad \phi(z_k)=-z_k,
\end{equation*}
is bracket-compatible for \eqref{Eq:dbr-Lin}, as is easily checked.
Compatibility for $g=0$ was given in \cite{OS}.
\end{exmp}

\begin{prop} \label{Pr:muselect}
 Let $(A,\dgal{-,-},\mu)$ be a double Hamiltonian algebra over $A_0=\oplus_{s\in I} \CC e_s$
 with bracket-compatible anti-involution $\phi$.
Then, up to a shift $\mu \mapsto \mu+c$ with $c\in A_0$,
 we can assume
 \begin{equation}\label{mum-select}
 \mu+\phi(\mu)=0 \,.
\end{equation}
\end{prop}
\begin{proof}
First, note the following consequence of bracket-compatibility \eqref{Eq:BrComp} and \eqref{mum}
\begin{equation} \label{mum-b}
 \dgal{\phi(\mu_s),b}=\phi^{\otimes 2}(\dgal{\mu_s,\phi(b)})^\circ
 =e_s\otimes e_sb - be_s\otimes e_s = -\dgal{\mu_s,b}\,.
\end{equation}
In other words, $-\phi(\mu)=-\sum_{s\in I}\phi(\mu_s)$ is also a moment map.
It follows from \cite[Lem.~3.3]{F3} that $-\phi(\mu)=\mu+c$ for some $c\in A_0$.
The same result also gives that $\mu':=\mu+\frac12c$ is a moment map, and one has
$ \mu'+\phi(\mu')=0$.
Thus, up to a shift by $A_0$, we can pick the moment map so that \eqref{mum-select} holds.
\end{proof}

\subsection{The orthogonal case (types \texorpdfstring{$\typB,\typD$}{B,D})} \label{ss:P-orth}

Recall the notation from \ref{ss:Rep}, where the anti-automorphism in the orthogonal case is given by transposition: $\tau(E_{ij})=E_{ji}$.
The defining ideal \eqref{TwIdeal} of the twisted representation space is generated by the elements
\begin{equation} \label{Eq:DefRelO}
 \phi(a)_{ij}-a_{ji}, \qquad a\in A, \quad 1\leq i,j\leq N.
\end{equation}

\begin{thm}[Orthogonal Poisson case] \label{Thm:BrPO}
 Let $(A,\dgal{-,-})$ be a double Poisson algebra over the semi-simple algebra $A_0=\oplus_{s\in I} \CC e_s$.
Assume that $\phi:A\to A$ is an involutive anti-automorphism compatible with $\dgal{-,-}$ in the sense of \eqref{Eq:BrComp}.
For any dimension vector $\alpha\in \N^I$, $N=\sum_{s\in I}\alpha_s$, considering the involution $\tau$ on $\gl_N$
given by transposition, the antisymmetric biderivation $\br{-,-}_{\phi,\Orm}$ on $\Rep^{\phi,\tau}(A,\alpha)$ uniquely determined by\footnote{The factor $\frac12$ is taken for convenience for the moment map property and does not appear in \cite[(1.8)]{OS}.}
\begin{equation} \label{Eq:BrPO}
\br{a_{ij},b_{kl}}_{\phi,\Orm}=\frac12 \dgal{a,b}_{kj,il} + \frac12 \dgal{\phi(a),b}_{ki,jl} \,, \quad a,b\in A,
\end{equation}
is a Poisson bracket for which the action of $\Orm_\alpha :=\prod_{s\in I} \Orm(\alpha_s)$ is by Poisson automorphisms.

Furthermore, if $\mu \in A$ is a moment map chosen to satisfy \eqref{mum-select}, then
$\X(\mu):\Rep^{\phi,\tau}(A,\alpha) \to \og_\alpha$ is a moment map
for $\br{-,-}_{\phi,\Orm}$ after identifying $\og_\alpha=\prod_{s\in I} \og(\alpha_s)$ and its dual under the trace pairing.
\end{thm}
\begin{proof}
The first part is a straightforward adaptation of Theorem 3.9 and Proposition 3.10 of \cite{OS} to the orthogonal case over a semi-simple algebra.
The second part was proved over $A_0=\CC$ by Safonkin \cite{Saf}, but we give a proof for an arbitrary semi-simple base $A_0$.

We are assuming that $\mu$ satisfies $\mu+\phi(\mu)=0$.
In terms of the matrix $\X(\mu)$ on  $\Rep^{\phi,\tau}(A,\alpha)$,
this gives by \eqref{Eq:DefRelO} the equality $\X(\mu)+\X(\mu)^T=0_N$.
In particular, $\X(\mu)$ is $\og_\alpha$-valued.

Next, using the basis from Example \ref{Ex:Ortho}, we need to check
\begin{equation} \label{Eq:MuFO}
 \br{\langle\mu,F^{(s)}_{ij}\rangle_\og,-}_{\phi,\Orm} = (F^{(s)}_{ij})_M
\end{equation}
for all $s\in I$ and
$i,j\in\mathtt{R}_s$ \eqref{Eq:sIndices} subject to $i<j$,
where on the RHS we consider the infinitesimal vector field on $M=\Rep^{\phi,\tau}(A,\alpha)$.
The left-hand side (LHS) of \eqref{Eq:MuFO} evaluated on $b_{kl}$
reads thanks to \eqref{Eq:BrPO}, \eqref{mum} and \eqref{mum-b} as
\begin{equation} \label{Eq:pfPO1}
 \begin{aligned}
 \br{\mu_{ji}-\mu_{ij},b_{kl}}_{\phi,\Orm}
 &=\frac12 \dgal{\mu_s,b}_{ki,jl} + \frac12 \dgal{\phi(\mu_s),b}_{kj,il}
 -\frac12 \dgal{\mu_s,b}_{kj,il} - \frac12 \dgal{\phi(\mu_s),b}_{ki,jl} \\
&=(be_s\otimes e_s-e_s \otimes e_s b)_{ki,jl}
-(be_s\otimes e_s-e_s \otimes e_s b)_{kj,il}  \,.
 \end{aligned}
\end{equation}
Meanwhile, we obtain from \eqref{InfAct},
\begin{equation}
 \begin{aligned} \label{Eq:pfPO2}
 (F^{(s)}_{ij})_M(b_{kl})
&= ([\X(b),E^{(s)}_{ij}-E^{(s)}_{ji}])_{kl}
= b_{ki} \delta_{jl} - \delta_{ki} b_{jl}
 - b_{kj} \delta_{il} + \delta_{kj} b_{il} \,.
 \end{aligned}
\end{equation}
Due to the choice of indices $i,j$, we have $(be_s)_{ki}=b_{ki}$, $\delta_{ki}=(e_s)_{ki}$, and similar relations for the other indices that allow to conclude that \eqref{Eq:pfPO1} and \eqref{Eq:pfPO2} are equal, i.e. \eqref{Eq:MuFO} holds.
\end{proof}

As $\Orm_\alpha$ acts by Poisson automorphisms, the Poisson bracket of Theorem \ref{Thm:BrPO} restricts to $\CC[\Rep^{\phi,\tau}(A,\alpha)]^{\Orm_\alpha}$.
Explicit formulas can be given on the subalgebra $\CC[\Rep^{\phi,\tau}(A,\alpha)]^{\tr}$ generated by trace functions, cf.  \eqref{Eq:trmap}.

\begin{prop} \label{Pr:Tr-O}
The algebra $\CC[\Rep^{\phi,\tau}(A,\alpha)]^{\tr}$ is a Poisson algebra for the bracket induced by \eqref{Eq:BrPO}.
Explicitly, for $a,b\in A$, one has
\begin{equation} \label{Eq:BrPO-tr}
 \br{\tr(a),\tr(b)}_{\phi,\Orm}=\frac12 \tr\big(\mathrm{m} \circ (\dgal{a,b}+\dgal{\phi(a),b})\big)\,.
\end{equation}
\end{prop}
\begin{proof}
 We directly get by summing over $i=j$ in \eqref{Eq:BrPO},
\begin{align} \label{Eq:BrPO-trLeft}
\br{\tr(a),b_{kl}}_{\phi,\Orm}=\frac12 (\mathrm{m} \circ (\dgal{a,b}+\dgal{\phi(a),b}))_{kl} .
\end{align}
Summing over $k=l$ yields \eqref{Eq:BrPO-tr}. Thus the Poisson bracket closes on $\CC[\Rep^{\phi,\tau}(A,\alpha)]^{\tr}$.
\end{proof}

\subsection{The symplectic case (type \texorpdfstring{$\typC$}{C})}   \label{ss:P-symp}

\subsubsection{Notation for the symplectic case} \label{sss:Not-symp}

Recall the notation from \ref{ss:Rep}, where for fixed $\alpha\in 2\N^I$ the anti-automorphism in the orthogonal case is given by
$\tau : \xi \mapsto \Omega \xi^T \Omega^T$ with $\Omega$ as in \eqref{Eq:omega}.
Let us express $\tau$ on an elementary matrix $E_{ij}$.  To do so, put $N=\sum_{s\in I}\alpha_s$ and introduce
\begin{subequations} \label{Eq:IotaAlph}
 \begin{align}
 &\iota: \{1,\ldots,N\} \to I, \qquad \iota(k):=s \quad \text{if } \sum_{1\leq t< s}\alpha_t< k \leq \sum_{1\leq t\leq s}\alpha_t,  \\
 &\hat{\alpha} : \{1,\ldots,N\} \to \N, \qquad \hat{\alpha}(k):=\alpha_{\iota(k)}.
\end{align}
\end{subequations}

In the ordered decomposition $N=\alpha_1+\ldots+\alpha_{|I|}$, $\iota$ assigns
the $\iota(k)$-th ``block'' to which $k$ belongs
and $\hat{\alpha}$ gives the corresponding dimension.
We also introduce $\sgn_{\alpha}: \{1,\ldots,N\} \to \{-1,+1\}$ by
\begin{equation}
 \begin{aligned}
 &\sgn_{\alpha}(k):= \sgn_{\hat{\alpha}(k)}\Big(k- \sum_{1\leq t< \iota(k)}\alpha_t\Big)= \left\{
 \begin{array}{ll}
 +1 &\text{if } \sum\limits_{1\leq t< \iota(k)}\alpha_t  < k \leq \sum\limits_{1\leq t< \iota(k)}\alpha_t+\frac12 \hat{\alpha}(k) \\
 & \\
 -1 &\text{if } \sum\limits_{1\leq t< \iota(k)}\alpha_t+\frac12 \hat{\alpha}(k) < k \leq \sum\limits_{1\leq t\leq \iota(k)}\alpha_t
 \end{array} \right. \label{Eq:sgnAlp}
\end{aligned}
\end{equation}
where we used \eqref{Eq:sgn} for the explicit expression.
We can finally define $\tau$ at the level of indices as
\begin{equation} \label{Eq:TauSp}
 \tau: \{1,\ldots,N\} \to \{1,\ldots,N\}, \qquad \tau(k)=k+\frac12 \sgn_{\alpha}(k)\,  \hat{\alpha}(k)\,,
\end{equation}
where it satisfies $\tau^2=\id$.
One should understand $\tau$ as the operation ``$k\mapsto k+\frac12\alpha_s$ modulo $\alpha_s$'' inside each block of size $\alpha_s$, as one can see in the next example.

\begin{exmp}
 If $|I|=1$ and $\alpha=2n$, $\tau(k)=k+n$ if $1\leq k\leq n$ while $\tau(k)=k-n$ if $k+1\leq k\leq 2n$.
\end{exmp}

\begin{lem} \label{lem:TauSp}
 On elementary matrices, one has $\tau(E_{ij})=\sgn_\alpha(i)\, \sgn_\alpha(j)\, E_{\tau(j),\tau(i)}$.
\end{lem}
\begin{proof}
 By definition, $\tau(E_{ij})=\sum_{p,q=1}^N (\Omega E_{ji} \Omega^T)_{pq} E_{pq}$.
 Using $\Omega$ \eqref{Eq:omega}, we get for $1\leq u,v\leq N$
 \begin{align*}
\Omega_{uv} &= \delta_{\iota(u),\iota(v)}\,  (\Omega_{\hat{\alpha}(u)})_{u-\sum_{s< \iota(u)}\alpha_s , v-\sum_{t< \iota(v)}\alpha_t} \\
&= \delta_{\iota(u),\iota(v)}\, \sgn_{\hat{\alpha}(u)}\Big(u- \sum_{s< \iota(u)}\alpha_s\Big)\, \delta_{u,\tau(v)}
= \sgn_\alpha(u)\, \delta_{u,\tau(v)}.
 \end{align*}
Noticing from \eqref{Eq:sgnAlp} and \eqref{Eq:TauSp} the equality
\begin{equation} \label{Eq:sgnTau}
 \sgn_\alpha(\tau(u)) = - \sgn_\alpha(u)\,,
\end{equation}
we deduce the claimed coefficients $(\Omega E_{ji} \Omega^T)_{pq}=\Omega_{pj}\Omega_{qi}$.
\end{proof}

We are led to define a $2$-parameter sign function using \eqref{Eq:sgnAlp} as follows
\begin{equation} \label{Eq:sgnDble}
 \sgn_\alpha:\{1,\ldots,N\}\times\{1,\ldots,N\}\to \{-1,+1\}, \quad
 \sgn_\alpha(i,j)=\sgn_\alpha(i) \,\sgn_\alpha(j)\,.
 \end{equation}
It satisfies for any $1\leq i,j,k\leq N$,
\begin{equation} \label{Eq:sgnProp}
 \sgn_\alpha(i,j)=\sgn_\alpha(j,i), \quad \sgn_\alpha(i,j)\sgn_\alpha(j,k)=\sgn_\alpha(i,k), \quad
 \sgn_\alpha(i,\tau(j))=-\sgn_\alpha(i,j).
\end{equation}
As a corollary of Lemma \ref{lem:TauSp} and this notation, the defining ideal \eqref{TwIdeal} of the twisted representation space is generated by
\begin{equation} \label{Eq:DefRelSp}
 \phi(a)_{ij}-\sgn_\alpha(i,j)\, a_{\tau(j),\tau(i)}, \qquad a\in A, \quad 1\leq i,j\leq N.
\end{equation}

\subsubsection{The theorem}

\begin{thm}[Symplectic Poisson case] \label{Thm:BrSmp}
 Let $(A,\dgal{-,-})$ be a double Poisson algebra over the semi-simple algebra $A_0=\oplus_{s\in I} \CC e_s$.
Assume that $\phi:A\to A$ is an involutive anti-automorphism compatible with $\dgal{-,-}$ in the sense of \eqref{Eq:BrComp}.
For any dimension vector $\alpha\in 2\N^I$, $N=\sum_{s\in I}\alpha_s$, considering the involution $\tau$ on $\gl_N$
given by $\xi \mapsto \Omega \xi^T \Omega^T$ (for $\Omega$ in \eqref{Eq:omega}), the antisymmetric biderivation $\br{-,-}_{\phi,\Sp}$ on $\Rep^{\phi,\tau}(A,\alpha)$ uniquely determined by (using $\sgn_\alpha(-,-)$ \eqref{Eq:sgnDble})
\begin{equation} \label{Eq:BrPsymp}
\br{a_{ij},b_{kl}}_{\phi,\Sp}=\frac12 \dgal{a,b}_{kj,il}
+\frac12 \sgn_\alpha(i,j)\, \dgal{\phi(a),b}_{k\tau(i),\tau(j)l} \,, \quad a,b\in A,
\end{equation}
is a Poisson bracket for which the action of $\Sp_\alpha :=\prod_{s\in I} \Sp(\alpha_s)$ is by Poisson automorphisms.

Furthermore, if $\mu \in A$ is a moment map chosen to satisfy \eqref{mum-select}, then
$\X(\mu):\Rep^{\phi,\tau}(A,\alpha) \to \spg_\alpha$ is a moment map for $\br{-,-}_{\phi,\Sp}$ after identifying $\spg_\alpha=\prod_{s\in I} \spg(\alpha_s)$ and its dual under the trace pairing.
\end{thm}
\begin{rem} \label{Rem:Sympl}
  If $|I|=1$ and $\alpha=2n$, \eqref{Eq:BrPsymp} reads
 \begin{align*}
\br{a_{ij},b_{kl}}_{\phi,\Sp}&=\frac12 \dgal{a,b}_{kj,il}
+\frac12\, \dgal{\phi(a),b}_{k (i+n),(j+n)l}, \quad \text{for } 1\leq i,j\leq n, \\
&=\frac12 \dgal{a,b}_{kj,il}  +\frac12\, \dgal{\phi(a),b}_{k (i-n),(j-n)l}, \quad \text{for } n+1\leq i,j\leq 2n, \\
&=\frac12 \dgal{a,b}_{kj,il}  -\frac12\, \dgal{\phi(a),b}_{k (i+n),(j-n)l}, \quad \text{for } n+1\leq i\leq 2n,\,\, 1\leq j\leq n, \\
&=\frac12 \dgal{a,b}_{kj,il}  -\frac12\, \dgal{\phi(a),b}_{k (i-n),(j+n)l}, \quad \text{for } 1\leq i\leq n,\,\, n+1\leq j\leq 2n.
 \end{align*}
\end{rem}

\begin{proof}[Proof of Theorem \ref{Thm:BrSmp}]
The first part is again an adaptation of \cite{OS}.
For the second part,
recall that $\mu$ satisfies $\mu+\phi(\mu)=0$.
Combined with Lemma \ref{lem:TauSp} and \eqref{Eq:DefRelSp}, this entails
$\X(\mu)+\Omega \X(\mu)^T \Omega^T=0_N$, i.e. $\X(\mu)$ is $\spg_\alpha$-valued.
Next, consider the spanning family of $\spg_\alpha$ given by
\begin{equation}  \label{Eq:BasSp}
 F_{ij}^{(s)}=E_{ij}- \sgn_\alpha(i,j)\, E_{\tau(j),\tau(i)}\,,
\end{equation}
for all $s\in I$ and $i,j\in\mathtt{R}_s$ \eqref{Eq:sIndices}.
It suffices to prove
\begin{equation} \label{Eq:MuFSp}
 \br{\langle\mu,F^{(s)}_{ij}\rangle_\spg,-}_{\phi,\Sp} = (F^{(s)}_{ij})_M
\end{equation}
where on the RHS we consider the infinitesimal vector field on $M=\Rep^{\phi,\tau}(A,\alpha)$.
For the LHS, we get
\begin{equation} \label{Eq:pfPS1}
 \begin{aligned}
 \br{\langle\mu,F^{(s)}_{ij}\rangle_\spg , b_{kl}}_{\phi,\Sp}
 &=
 \br{ \mu_{ji} , b_{kl}}_{\phi,\Sp}
 - \sgn_\alpha(i,j)  \br{ \mu_{\tau(i),\tau(j)}  , b_{kl}}_{\phi,\Sp} \\
 &=\frac12 \dgal{\mu_s,b}_{ki,jl}
 +\frac12 \sgn_\alpha(j,i)\, \dgal{\phi(\mu(s)),b}_{k\tau(j),\tau(i)l} \\
&\quad  -\frac12 \sgn_\alpha(i,j) \dgal{\mu_s,b}_{k\tau(j),\tau(i)l}
- \frac12 \dgal{\phi(\mu_s),b}_{ki,jl} \\
&=(be_s\otimes e_s-e_s \otimes e_s b)_{ki,jl}
-\sgn_\alpha(i,j)\, (be_s\otimes e_s-e_s \otimes e_s b)_{k\tau(j),\tau(i)l}  \,.
 \end{aligned}
\end{equation}
We used \eqref{Eq:BrPsymp}, the involutivity of $\tau$ and \eqref{Eq:sgnTau} for the second equality, then \eqref{mum} with \eqref{mum-b}.

Meanwhile, we obtain from \eqref{InfAct},
\begin{equation}
 \begin{aligned} \label{Eq:pfPS2}
 (F^{(s)}_{ij})_M(b_{kl})
&= ([\X(b),E_{ij}])_{kl}-\sgn_\alpha(i,j)\,([\X(b),E_{\tau(j),\tau(i)}])_{kl} \\
&= b_{ki} \delta_{jl} - \delta_{ki} b_{jl}
 -\sgn_\alpha(i,j)( b_{k\tau(j)} \delta_{\tau(i)l} - \delta_{k\tau(j)} b_{\tau(i)l}) \,.
 \end{aligned}
\end{equation}
As in the proof of Theorem \ref{Thm:BrPO}, one concludes that \eqref{Eq:pfPS1} and \eqref{Eq:pfPS2} are equal.
\end{proof}

As in the orthogonal case,
the invariant ring $\CC[\Rep^{\phi,\tau}(A,\alpha)]^{\Sp_\alpha}\subset \CC[\Rep^{\phi,\tau}(A,\alpha)]$ is a Poisson subalgebra.
Denote by $\CC[\Rep^{\phi,\tau}(A,\alpha)]^{\tr}$ the subalgebra generated by trace elements.
The following result is proved similarly to Proposition \ref{Pr:Tr-O}.
\begin{prop} \label{Pr:Tr-Smp}
The algebra $\CC[\Rep^{\phi,\tau}(A,\alpha)]^{\tr}$ is a Poisson algebra for the bracket induced by \eqref{Eq:BrPsymp}.
Explicitly, for $a,b\in A$, one has
\begin{equation} \label{Eq:BrPSymp-tr}
\br{\tr(a),\tr(b)}_{\phi,\Sp}=\frac12 \tr\big(\mathrm{m} \circ (\dgal{a,b}+\dgal{\phi(a),b})\big)\,.
\end{equation}
\end{prop}

\subsection{Applications}

\subsubsection{Bracket compatibility on NC cotangent spaces} \label{ss:NCcot}

Let $A$ be an algebra over $A_0=\oplus_{s\in I} \CC e_s$.
Assume that $A$ is equipped with an involutive anti-automorphism $\phi_0:A \to A$ satisfying
$\phi_0(e_s)=e_s$ for all $s\in I$.
Define the $A$-bimodule of double derivations (relative to $A_0$) as
\begin{equation*}
 \DDer(A)= \{\delta \in \Hom_{\CC}(A,A\otimes A) \mid \delta(A_0)=0; \,\,
\delta(ab)=\delta(a)b+a\delta(b)\quad \forall a,b\in A\}\,.
\end{equation*}
The bimodule structure is obtained from the inner one on $A^{\otimes 2}$ as
$(a\delta b)(c) := a\ast \delta(c)\ast b$ for $a,b,c \in A$ and $\delta\in \DDer(A)$.

\begin{lem}
 The involutive anti-automorphism $\phi_0:A \to A$ induces a linear map
 \begin{equation} \label{Def-phi1}
\phi_1 :  \DDer(A) \to  \Hom_{\CC}(A,A\otimes A), \qquad
\phi_1(\delta) (a):= \phi_0^{\otimes 2} ( [\delta(\phi_0(a))]^\circ)
 \end{equation}
which is valued in $\DDer(A)$ and is involutive. Furthermore, the $A$-bimodule structure transforms as
\begin{equation} \label{phi1Bim}
 \phi_1(a \delta b) = \phi_0(b) \,\phi_1(\delta)\,\phi_0(a)\,, \qquad a,b\in A, \,\, \delta \in \DDer(A)\,.
\end{equation}
\end{lem}

\begin{proof}
For any $\delta \in \DDer(A)$ and $s\in I$, $\phi_1(\delta)(e_s)=0$ because this is true for $\delta$.
Also, $\phi_1(\delta)$ is clearly linear. Thus, we need to check the derivation property.
For $a,b\in A$, one has
\begin{align*}
 \phi_1(\delta) (ab) &= \phi_0^{\otimes 2} ( [\delta(\phi_0(b)\phi_0(a))]^\circ) \\
 &= \phi_0^{\otimes 2} ( [\delta(\phi_0(b))\phi_0(a) + \phi_0(b) \delta(\phi_0(b)) ]^\circ) \\
 &= \phi_0^{\otimes 2} ( \delta(\phi_0(b))^\circ \ast \phi_0(a) + \phi_0(b) \ast \delta(\phi_0(b))^\circ) \\
 &= \phi_0^2(a)\, \phi_0^{\otimes 2} (\delta(\phi_0(b))^\circ)
 +  \phi_0^{\otimes 2}(\delta(\phi_0(b))^\circ) \, \phi_0^2(b)
 =a\, (\phi_1(\delta)(b)) + (\phi_1(\delta)(a))\, b\,.
\end{align*}
To show \eqref{phi1Bim}, we evaluate the LHS on $c\in A$ to get
\begin{align*}
\phi_1(a \delta b)(c)
 &=\phi_0^{\otimes 2} ( [a\ast \delta(\phi_0(c))\ast b]^\circ) \\
 &=\phi_0^{\otimes 2} ( a \, [\delta(\phi_0(c))]^\circ \, b)
 = \phi_0(b) \ast \phi_0^{\otimes 2}([\delta(\phi_0(c))]^\circ) \ast \phi_0(a) \,.
\end{align*}
This is the RHS of  \eqref{phi1Bim} evaluated on $c$ by definition of $\phi_1$ and the bimodule structure on $\DDer(A)$.
\end{proof}

Following Van den Bergh \cite{VdB1}, we consider the operations
\begin{equation}
 \begin{aligned}  \label{Eq:dbrDDer}
 &\dgal{-,-}_l : \DDer(A)^{\otimes 2}\to \DDer(A)\otimes A, \quad
 \dgal{\delta,\Delta}_l = \tau_{(23)} \circ \left[(\delta \otimes \id_A)\circ \Delta - (\id_A \otimes \Delta)\circ \delta \right], \\
  &\dgal{-,-}_r : \DDer(A)^{\otimes 2}\to A\otimes \DDer(A), \quad
 \dgal{\delta,\Delta}_r = \tau_{(12)} \circ \left[(\id_A \otimes \delta)\circ \Delta - (\Delta \otimes \id_A)\circ \delta  \right].
\end{aligned}
\end{equation}

(Here, we interpret $\tau_s$ for any permutation $s\in S_3$ as the map permutting tensor factors accordingly after applying the operation to an element $a\in A$ as we end up in $A^{\otimes 3}$.)
Next, we form the tensor algebra $\Tnc A:=T_A\DDer(A)$, which we see as an ungraded algebra.

\begin{thm} \label{Thm:dbrSN}
There is a unique structure of double Poisson algebra on $\Tnc A$ satisfying
for any $a,b\in A$ and $\delta,\Delta\in \DDer(A)$,
 \begin{equation}
  \label{Eq:dbrSN}
\dgal{a,b}=0, \qquad \dgal{\delta,a}=\delta(a),\qquad
\dgal{\delta,\Delta}=\dgal{\delta,\Delta}_l + \dgal{\delta,\Delta}_r\,.
 \end{equation}
\end{thm}
\begin{proof}
For a \emph{graded} double Poisson bracket (of degree $-1$), this is \cite[Thm.~3.2.2]{VdB1}.
It is folklore that the statement also holds in the ungraded case, see e.g. \cite[Prop.~3.2.2]{DA}.
\end{proof}

\begin{prop}
There is a well-defined involutive anti-automorphism $\phi:\Tnc A \to \Tnc A$
that agrees with $\phi_0$ on $(\Tnc A)_0=A$ and with $\phi_1$ \eqref{Def-phi1} on $(\Tnc A)_1=\DDer(A)$.
We call  $\phi$ the \emph{cotangent lift} of $\phi_0$.
\end{prop}
\begin{proof}
 An arbitrary element of $\Tnc A$ is a linear combination of terms
$\pi_k := a_1 \delta_1 a_2 \delta_2 \cdots a_k \delta_k a_{k+1}$ with
$a_i\in A$, $\delta_i \in \DDer(A)$, $k\geq 0$.
We define $\phi$ by extending it linearly from
 \begin{equation*}
\phi(\pi_k) := \phi_0(a_{k+1}) \phi_1(\delta_k) \phi_0(a_k) \cdots \phi_0(a_2) \phi_1(\delta_1) \phi_0(a_{1})\,.
 \end{equation*}
This is well-defined by \eqref{phi1Bim}.
\end{proof}

\begin{thm} \label{Thm:NCBrComp}
The cotangent lift $\phi$ of the involutive anti-automorphism $\phi_0:A \to A$ is bracket-compatible with the (Van den Bergh's) double Poisson bracket on $\Tnc A$ from Theorem \ref{Thm:dbrSN}.
\end{thm}
\begin{proof}
It suffices to check \eqref{Eq:BrComp} on generators of $\Tnc A$, i.e. on elements of $A$ and $\DDer(A)$.
For $a\in A$, $\phi(a)=\phi_0(a)\in A$ so for any $a,b\in A$ both sides of \eqref{Eq:BrComp} are trivially zero by \eqref{Eq:dbrSN}.

Given $\delta\in \DDer(A)$ and $a\in A$, we compute using \eqref{Def-phi1} and \eqref{Eq:dbrSN},
\begin{align*}
 \dgal{\phi(\delta),\phi(a)}^\circ = [\phi_1(\delta)\big(\phi_0(a)\big)]^\circ
 = [\phi_0^{\otimes 2} \delta(a)^\circ ]^\circ
 =\phi_0^{\otimes 2} \circ \delta(a) = \phi^{\otimes 2} \dgal{\delta,a}\,.
\end{align*}
This is precisely \eqref{Eq:BrComp}.
Finally, we need to verify for $\delta,\Delta \in \DDer(A)$ that
\begin{equation} \label{NCBrComp1}
\dgal{\phi(\delta),\phi(\Delta)}^\circ =  \phi^{\otimes 2} \dgal{\delta,\Delta}\,.
\end{equation}
Permutation of factors in $\Tnc A\otimes \Tnc A$ on the LHS must be understood as follows,
cf. \eqref{Eq:dbrDDer}-\eqref{Eq:dbrSN},
\begin{equation} \label{NCBrComp2}
\eqref{NCBrComp1}_{LHS} =
\tau_{(123)} \dgal{\phi_1(\delta),\phi_1(\Delta)}_l
+ \tau_{(132)}  \dgal{\phi_1(\delta),\phi_1(\Delta)}_r\,.
\end{equation}
Let us first compute for any $a\in A$ using \eqref{Def-phi1},
\begin{align*}
 [\phi_1(\delta) \otimes \id_A] \circ \phi_1(\Delta)(a)
 =& [(\phi_0^{\otimes 2}\otimes \id_A)\circ \tau_{(12)}\circ (\delta\circ \phi_0 \otimes \id_A)]
 \circ \phi_0^{\otimes 2}\circ (\Delta\circ \phi_0(a))^\circ \\
 =& \phi_0^{\otimes 3} \circ \tau_{(12)}\circ \big(\tau_{(132)} (\id_A\otimes \delta) \circ \Delta \big) \circ \phi_0(a) \\
 =& \tau_{(13)}\circ  \phi_0^{\otimes 3} \circ \big((\id_A\otimes \delta) \circ \Delta\big)\circ \phi_0(a)\,; \\
[\id_A\otimes \phi_1(\Delta)] \circ \phi_1(\delta)(a)
=& \tau_{(13)}\circ  \phi_0^{\otimes 3} \circ \big((\Delta \otimes \id_A) \circ \delta\big)\circ \phi_0(a)\,.
\end{align*}
We can also find from \eqref{Def-phi1} and \eqref{Eq:dbrDDer},
\begin{align*}
\phi^{\otimes 2}(\dgal{\delta,\Delta}_r)
&= \phi_0^{\otimes 3} \circ \tau_{(23)} \circ \dgal{\delta,\Delta}_r \circ \phi_0 \\
&= \tau_{(132)} \circ \phi_0^{\otimes 3} \circ
\left[(\id_A \otimes \delta)\circ \Delta - (\Delta \otimes \id_A)\circ \delta  \right] \circ \phi_0\,.
\end{align*}
(The first equality was obtained since $\phi^{\otimes 2}$ acts as $\phi_0\otimes \phi_1$.)
These equalities and \eqref{Eq:dbrDDer} yield
\begin{equation}
 \begin{aligned}  \label{NCBrComp3}
\tau_{(123)} \dgal{\phi_1(\delta),\phi_1(\Delta)}_l
&= \tau_{(12)} \circ \left[(\phi_1(\delta) \otimes \id_A)\circ \phi_1(\Delta)
- (\id_A \otimes \phi_1(\Delta))\circ \phi_1(\delta) \right] \\
&= \tau_{(132)}\circ \phi_0^{\otimes 3} \circ
\left[  (\id_A\otimes \delta) \circ \Delta - (\Delta \otimes \id_A) \circ \delta\right]
\circ \phi_0 \\
&= \phi^{\otimes 2}(\dgal{\delta,\Delta}_r)\,.
 \end{aligned}
\end{equation}
Using that $\dgal{-,-}_r=-\tau_{(123)}\circ \dgal{-,-}_l$, \eqref{NCBrComp3} entails
\begin{equation}
 \begin{aligned}  \label{NCBrComp4}
\tau_{(132)} \dgal{\phi_1(\delta),\phi_1(\Delta)}_r
&=- \tau_{(132)} \phi^{\otimes 2}(\dgal{\delta,\Delta}_r)
= \phi^{\otimes 2}(\dgal{\delta,\Delta}_l)\,.
 \end{aligned}
\end{equation}
Combining \eqref{NCBrComp2} with \eqref{NCBrComp3}-\eqref{NCBrComp4} yields \eqref{NCBrComp1}.
\end{proof}

\subsubsection{Quivers and NC cotangent spaces} \label{sss:P-exmp}

Let $\Upsilon=(\Upsilon,I,t,h)$ be a finite quiver. I.e. $I$ is the finite vertex set, $\Upsilon$ the finite arrow set, and
$t,h:\Upsilon\to I$ are the tail and head maps. For simplicity, we write $a:i\to j$ if $t(a)=i$ and $h(a)=j$.

Choose a set $\Lambda\subset \Upsilon$ of loops, i.e. $a\in \Lambda$ when $h(a)=t(a)$. (This can be the empty set, e.g. if there are no loops.)
We construct the quiver $Q=(Q_{\Upsilon,\Lambda},I,t,h)$ as follows: its vertex set is the same as $\Upsilon$, its arrow set is
obtained by adjoining an arrow $a':h(a)\to t(a)$ to each $a\in \Upsilon \setminus \Lambda$.

Finally, construct the double $\overline{Q}$ of $Q$ in the usual way by adding an arrow $a^\ast:h(a)\to t(a)$ for each $a\in Q$. We extend
$(-)^\ast:\overline{Q}\to \overline{Q}$ by putting $(a^\ast)^\ast=a$ for all $a\in Q$.
In that way, going from $\Upsilon$ to $\overline{Q}$ quadruples the quiver, except on $\Lambda$ which is only doubled. An example is depicted in Figure \ref{fig:M1}.

\begin{figure}
\centering
   \begin{tikzpicture}[scale=0.8]
  \node[circle,thick,fill=black,inner sep=1pt]  (v0) at (0,0) {};
  \node[circle,thick,fill=black,inner sep=1pt]  (vu) at (-2,1) {};
   \node[circle,thick,fill=black,inner sep=1pt] (vd) at (-2,-1) {};
  \node  (vname) at (-1,-1.8) {$Q_{\Upsilon,\Lambda}$};
\draw[->,thick,>=latex] (vu) edge [bend left=12] node[above,font=\small]{$a$}   (v0) ;
\draw[->,thick,>=latex,blue] (v0) edge [bend left=12] (vu) ;
\node[font=\small,blue] (vAname) at (-1.6,0.4) {$a'$} ;
\draw[->,thick,>=stealth] (vu) to[out=140,in=220,looseness=15] node[left,font=\small]{$c$}  (vu);
\draw[->,thick,>=stealth,blue] (vd) to[out=90,in=170,looseness=15] node[left,font=\small]{$d'$}  (vd);
\draw[->,thick,>=stealth] (vd) to[out=200,in=280,looseness=15] node[left,font=\small]{$d$}  (vd);
\draw[->,thick,>=latex] (vd) edge [bend right=12] node[below,font=\small]{$b$}   (v0) ;
\draw[->,thick,>=latex,blue] (v0) edge [bend right=12] node[left,font=\small]{$b'$}   (vd) ;
  \node  (vname) at (-3.5,0) {$\longrightarrow$};
  \node[circle,thick,fill=black,inner sep=1pt]  (Lv0) at (-5,0) {};
  \node[circle,thick,fill=black,inner sep=1pt]  (Lvu) at (-7,1) {};
   \node[circle,thick,fill=black,inner sep=1pt] (Lvd) at (-7,-1) {};
  \node  (vname) at (-6,-1.8) {$\Upsilon$};
\draw[->,thick,>=latex] (Lvu) -- node[above,font=\small]{$a$}   (Lv0) ;
\draw[->,thick,>=stealth] (Lvu) to[out=140,in=220,looseness=15] node[left,font=\small]{$c$}  (Lvu);
\draw[->,thick,>=stealth] (Lvd) to[out=140,in=220,looseness=15] node[left,font=\small]{$d$}  (Lvd);
\draw[->,thick,>=latex] (Lvd) -- node[below,font=\small]{$b$}   (Lv0) ;
  \node  (vname) at (1.5,0) {$\longrightarrow$};
  \node[circle,fill=black,inner sep=1pt]  (Rv0) at (5,0) {};
  \node[circle,fill=black,inner sep=1pt]  (Rvu) at (4,1.2) {};
   \node[circle,fill=black,inner sep=1pt] (Rvd) at (4,-1.2) {};
  \node  (vname) at (-6,-1.8) {$\Upsilon$};
\draw[->,>=latex] (Rvu) edge [bend left=10] (Rv0) ;
\draw[->,>=latex,blue] (Rv0) edge [bend left=10] (Rvu) ;
\draw[->,>=latex,teal] (Rv0) edge [bend right=35] node[right,font=\scriptsize]{$a^\ast$} (Rvu) ;
\draw[->,>=latex,red] (Rvu) edge [bend right=35] node[left,font=\scriptsize]{${a'}^\ast$} (Rv0) ;
\draw[->,>=stealth] (Rvu) to[out=70,in=130,looseness=20] (Rvu);
\draw[->,>=stealth,teal] (Rvu) to[out=150,in=200,looseness=20] node[left,font=\small]{$c^\ast$}  (Rvu);
\draw[->,>=stealth,blue] (Rvd) to[out=100,in=140,looseness=20] (Rvd);
\draw[->,>=stealth] (Rvd) to[out=280,in=320,looseness=20] (Rvd);
\draw[->,>=stealth,red] (Rvd) to[out=160,in=200,looseness=20] node[left,font=\scriptsize]{${d'}^\ast$} (Rvd);
\draw[->,>=stealth,teal] (Rvd) to[out=220,in=260,looseness=20] node[left,font=\scriptsize]{$d^\ast$} (Rvd);
\draw[->,>=latex] (Rvd) edge [bend right=12] (Rv0) ;
\draw[->,>=latex,blue] (Rv0) edge [bend right=12] (Rvd) ;
\draw[->,>=latex,red] (Rv0) edge [bend right=35] node[left,font=\scriptsize]{${b'}^\ast$} (Rvd) ;
\draw[->,>=latex,teal] (Rvd) edge [bend right=35] node[right,font=\scriptsize]{$b^\ast$} (Rv0) ;
   \end{tikzpicture}
\caption{Example of construction of the quivers $Q_{\Upsilon,\Lambda}$ and $\overline Q_{\Upsilon,\Lambda}$ from a quiver $\Upsilon$ and a choice of loops $\Lambda=\{c\}$. }
\label{fig:M1}
\end{figure}
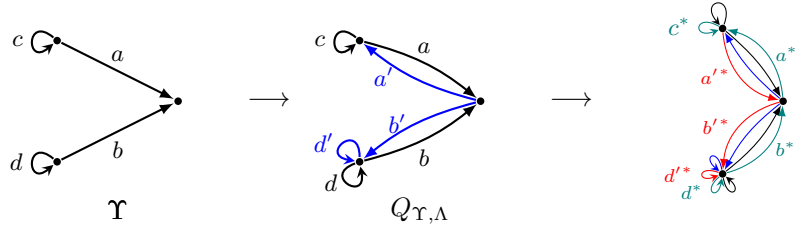

With the above notation, consider the path algebra $\CC\overline{Q}$ generated by the arrows of $\overline{Q}$ and the complete set of orthogonal idempotents $\{e_s \mid s \in I\}$  subject to the relations\footnote{We write paths from left to right as in \cite{VdB1}.}
$$a = e_{t(a)} a e_{h(a)}, \qquad a\in \overline{Q}\,.$$
Fix $\gamma:\Lambda\to \{\pm 1\}$, $c\mapsto \gamma_c$, so that $\gamma^2\equiv 1$.
The path algebra $\CC\overline{Q}$ is equipped with an involutive anti-automorphism $\phi$ that fixes the idempotents and such that on arrows
\begin{subequations} \label{Eq:PhiQ}
 \begin{align}
&\phi(a)=a', \,\, \phi(a')=a, \,\, \text{ for }a\in \Upsilon\setminus\Lambda,
&&\phi(c)=\gamma_c \,c,\,\, \text{ for }c\in \Lambda ; \label{Eq:PhiQa} \\
&\phi(a^\ast)={a'}^\ast, \,\, \phi({a'}^\ast)=a^\ast, \,\, \text{ for }a\in \Upsilon\setminus\Lambda,
&&\phi(c^\ast)=\gamma_c \,c^\ast,\,\, \text{ for }c\in \Lambda . \label{Eq:PhiQb}
 \end{align}
\end{subequations}

\begin{prop} \label{Pr:Upsilon}
Consider a quiver $Q$ constructed from a pair $(\Upsilon,\Lambda)$ as above.
The involutive anti-automorphism $\phi$ \eqref{Eq:PhiQ} is bracket-compatible for
the canonical double Poisson bracket on $\CC\overline{Q}$ given by Van den Bergh \cite[\S6.3]{VdB1}.
Furthermore, the following moment map satisfies \eqref{mum-select}:
\begin{equation} \label{Eq:MomapQ}
 \mu=\sum_{a\in \Upsilon}  [a,a^\ast] + \sum_{a\in \Upsilon\setminus\Lambda} [a',a'^\ast]\,.
\end{equation}
\end{prop}
\begin{proof}
This is just an application of Theorem \ref{Thm:NCBrComp} with $A=\CC Q$, $A_0=\oplus_{s\in I} \CC e_s$,
because $\Tnc A\simeq \CC\overline{Q}$, cf. \cite[\S6.2-6.3]{VdB1}. The original anti-automorphism is \eqref{Eq:PhiQa} which lifts to \eqref{Eq:PhiQb} on $\DDer(\CC Q)$.
The result can be checked explicitly: up to an overall sign, the double Poisson bracket satisfies
\begin{equation} \label{Eq:dbr-Q}
 \begin{aligned}
  &\dgal{a,a^\ast}=e_{h(a)}\otimes e_{t(a)}, \quad &&\dgal{a^\ast,a}=-e_{t(a)}\otimes e_{h(a)}, \quad &&&a\in \Upsilon, \\
  &\dgal{a',{a'}^\ast}=e_{h(a')}\otimes e_{t(a')}, \quad &&\dgal{{a'}^\ast,a'}=-e_{t(a')}\otimes e_{h(a')}, \quad &&&a\in \Upsilon\setminus \Lambda,
 \end{aligned}
\end{equation}
with all other brackets being trivially zero. The compatibility \eqref{Eq:BrComp} is easily checked using $t(a')=h(a)$ and $h(a')=t(a)$ for $a\in \Upsilon\setminus \Lambda$.
For the second part, it is again a consequence of \cite{VdB1} that \eqref{Eq:MomapQ} is a moment map.
The fact that it satisfies $\mu+\phi(\mu)=0$ \eqref{mum-select} is easy to obtain from \eqref{Eq:PhiQ}.
\end{proof}

\begin{exmp} \label{Exmp:Almost}
Consider the quiver $\Upsilon$ made of two vertices $\{0,\infty\}$, and two arrows $x:0\to 0$, $v:0\to \infty$.
Put $\Lambda=\{x\}$.
The quiver $Q$ has an extra arrow $v':\infty \to 0$, and its double $\overline{Q}$ has three extra arrows
$x^\ast:0\to 0$, $v^\ast:\infty \to 0$, $v'^\ast:0\to \infty$.
The anti-involution $\phi$ satisfies
$\phi(x)=-x$, $\phi(x^\ast)=-x^\ast$,  $\phi(v)=v'$ and $\phi(v^\ast)={v'}^\ast$.
By proposition \ref{Pr:Upsilon}, $\phi$ is compatible with Van den Bergh's double Poisson bracket on $\CC\overline{Q}$.
We also note that the moment map \eqref{Eq:MomapQ} is
\begin{equation*}
 \mu = [x,x^\ast]+[v,v^\ast] + [v',v'^\ast]\,.
\end{equation*}

Take the dimension vector $(n_0,n_\infty)=(N,1)$ for some $N\geq 1$.
In the orthogonal case, one has
$$\Rep^{\phi,\tau}(\CC \overline{Q}_{\Upsilon,x},(N,1))\simeq \{X,Y\in \og_N, \,\, V\in \CC^N,\,\, W\in (\CC^N)^\ast \}\,,$$
with $\og_N$-valued moment map
$\mu^{(\og_N)}:=\X(\mu_0)=[X,Y]+VW + (VW)^T$ by \eqref{Eq:DefRelO}.
(We could eliminate the prime variables thanks to \eqref{Eq:DefRelO}.)
The Poisson structure is simply obtained from \eqref{Eq:BrPO} as
\begin{equation*}
 \br{X_{ij},Y_{kl}}=\frac12 (\delta_{kj}\delta_{il}-\delta_{ki}\delta_{jl}), \quad
 \br{V_{i},W_{l}}=\frac12 \delta_{il}\,.
\end{equation*}
In the symplectic case with $N$ even, one has
$$\Rep^{\phi,\tau}(\CC \overline{Q}_{\Upsilon,x},(N,1))\simeq \{X,Y\in \spg_N, \,\, V\in \CC^N,\,\, W\in (\CC^N)^\ast \}\,,$$
with $\spg_N$-valued moment map
$\mu^{(\spg_N)}:=[X,Y]+VW + \Omega_{N}(VW)^T\Omega_{N}^T$ by \eqref{Eq:DefRelSp} and
Poisson structure \eqref{Eq:BrPsymp} written as
\begin{equation*}
 \br{X_{ij},Y_{kl}}=\frac12 (\delta_{kj}\delta_{il}-\sgn_{N}(i,j)\, \delta_{k,i+\frac{N}{2}}\,\delta_{j+\frac{N}{2},l}), \quad
 \br{V_{i},W_{l}}=\frac12 \delta_{il}\,,
\end{equation*}
with indices understood modulo $N$, cf. Remark \ref{Rem:Sympl}.
Performing reduction, both $(\mu^{(\og_N)})^{-1}(0_N)/\!/\Orm_N$ and $(\mu^{(\spg_N)})^{-1}(0_N)/\!/\Sp_N$ inherit a Poisson structure.
These are rank one versions of the commuting variety.
\end{exmp}

\subsubsection{Orthogonal and symplectic symmetries for untwisted representation spaces} \label{ss:Untwist}

In the previous \ref{sss:P-exmp}, starting from $\Upsilon$ without choosing loops ($\Lambda=\emptyset$), $\overline{Q}$ is the twice doubled quiver $\overline{\overline{\Upsilon}}$ with arrows
\[
 a,a'^\ast : t(a) \to h(a), \quad
 a',a^\ast: h(a) \to t(a), \qquad \forall a\in \Upsilon\,.
\]
Using the defining relation \eqref{TwIdeal} of
the twisted representation space $\Rep^{\phi,\tau}(\CC \overline{\overline{\Upsilon}},\alpha)$
(for a fixed type $\tau$ orthogonal or symplectic) and the involution \eqref{Eq:PhiQ},
the generators $a'_{ij}, a'^\ast_{ij}$ can be omitted.
I.e.,
\begin{equation} \label{Eq:Untwist}
 \CC[\Rep^{\phi,\tau}(\CC \overline{\overline{\Upsilon}},\alpha) ]
\simeq \CC[\Rep(\CC \overline{\Upsilon},\alpha)]
\end{equation}
is the coordinate ring of the (standard) representation space of the once-doubled quiver $\overline{\Upsilon}$.
In fact, since the double Poisson bracket \eqref{Eq:dbr-Q} restricted to $\CC \overline{\Upsilon}$ is the original one of Van den Bergh, the identification \eqref{Eq:Untwist} is a Poisson algebra isomorphism \emph{up to rescaling} the Poisson bracket on the right side by $\frac12$. This factor comes from \eqref{Eq:BrPO} and \eqref{Eq:BrPsymp}.
We can also transfer through \eqref{Eq:Untwist} the action of $\Orm_\alpha$ or $\Sp_\alpha$ (depending on $\tau$) and the corresponding moment map $\X(\mu)$
for $\mu$ \eqref{Eq:MomapQ}, cf. Theorem \ref{Thm:BrPO} or \ref{Thm:BrSmp}.
Then, the moment map reads (we assume $\Lambda=\emptyset$)
\begin{equation} \label{Eq:MomapQ-expl}
 \begin{aligned}
 \X(\mu)&=\sum_{a\in \Upsilon}  ( [\X(a),\X(a^\ast)] + [\X(\phi(a)),\X(\phi(a^\ast))] ) \\
 &=\left\{
 \begin{array}{ll}
\sum_{a\in \Upsilon}  [\X(a),\X(a^\ast)] + \sum_{a\in \Upsilon} [\X(a)^T,\X(a^\ast)^T]  & (\text{type }\typB,\typD) \\
\sum_{a\in \Upsilon}   [\X(a),\X(a^\ast)] + \sum_{a\in \Upsilon} \Omega[\X(a)^T,\X(a^\ast)^T]\Omega^T  & (\text{type }\typC)
 \end{array}
 \right.
\end{aligned}
\end{equation}
We gather these observations inside the following result.
\begin{cor}
The (untwisted) representation space $\Rep(\CC \overline{\Upsilon},\alpha)$ endowed with Van den Bergh's Poisson bracket rescaled by a factor $\frac12$ is equipped with an action of $\Orm_\alpha$ or $\Sp_\alpha$ by Poisson automorphism.
Furthermore, it admits a corresponding moment map given by \eqref{Eq:MomapQ-expl}.
\end{cor}

In fact, we simply derived a classical construction for turning
a (Hamiltonian) Poisson $\Gl_\alpha$-variety into one for the action of
$\Orm_\alpha\subset \Gl_\alpha$ or $\Sp_\alpha\subset \Gl_\alpha$.
In greater generalities the following holds.

\begin{prop} \label{Pr:RepP-Type}
 Let $(A,\dgal{-,-})$ be a double Poisson algebra over $A_0=\oplus_{s\in I}\CC e_s$.
Let $\alpha\in \N^I$ be a dimension vector.
For any $s\in I$, pick
$G_s\in \{\Gl(\alpha_s),\Orm(\alpha_s)\}$ if $\alpha_s$ is odd;
$G_s\in \{\Gl(\alpha_s),\Orm(\alpha_s),\Sp(\alpha_s)\}$ if $\alpha_s$ is even.
Then the (untwisted) representation space $\Rep(A,\alpha)$ endowed with the unique Poisson bracket satisfying
\begin{equation} \label{Eq:BrGen}
\br{a_{ij},b_{kl}}=\frac12 \dgal{a,b}_{kj,il}\,, \quad a,b\in A,
\end{equation}
is equipped with an action of $G:=\prod_{s\in I}G_s$ by Poisson automorphism.
Furthermore, it admits a corresponding $\operatorname{Lie}(G)$-valued moment map $Y=\prod_s Y_s$ given by
\begin{equation} \label{Eq:MomapGen}
 \begin{aligned}
 Y_s
 &=\left\{
 \begin{array}{ll}
 2 \X(\mu_s), & G_s=\Gl(\alpha_s); \\
\X(\mu_s) - \X(\mu_s)^T , & G_s=\Orm(\alpha_s);\\
\X(\mu_s)- \Omega_{\alpha_s} \X(\mu_s)^T \Omega_{\alpha_s}^T , & G_s=\Sp(\alpha_s).
 \end{array}
 \right.
\end{aligned}
\end{equation}
\end{prop}
\begin{proof}
This is classical.
If all $G_s=\Gl(\alpha_s)$, the fact that twice the Poisson bracket satisfying \eqref{Eq:BrGen} with the moment map $\frac12Y$ \eqref{Eq:MomapGen} satisfies the statement is a consequence of \cite[\S7]{VdB1}.
Then, passing to subgroups $\Orm(\alpha_s)\subset \Gl(\alpha_s)$ or $\Sp(\alpha_s)\subset \Gl(\alpha_s)$ is a standard calculation in Poisson geometry.
\end{proof}

In Section \ref{Sec:Mix}, we shall see how to mix the types at different $s\in I$.
This will allow us to give a proof of Proposition \ref{Pr:RepP-Type} purely based on noncommutative Poisson geometry, see \ref{ss:NCproof}.


\section{From double quasi-Poisson to quasi-Poisson in types \texorpdfstring{$\typB,\typC,\typD$}{B,C,D}}  \label{Sec:DqPoi}

We adapt the main results of Olshanski and Safonkin, as stated in Sec.\ref{Sec:DPoi}, to the quasi-Poisson setting.
We start with some examples of bracket-compatible involutions for double quasi-Poisson brackets on the free (Laurent) algebras
$$L_{g,r}=\CC\langle x_1^{\pm 1},y_1^{\pm 1},\ldots,x_g^{\pm 1},y_g^{\pm 1},z_1^{\pm 1},\ldots,z_r^{\pm 1}\rangle\,.$$
Note that $L_{g,r}$ has a double quasi-Poisson bracket if we realise it as $\CC\pi_1(\Sigma,\ast)$ for $\Sigma$ a genus $g$ surface with $r+1$ boundary components and a marked point $\ast\in \partial \Sigma$.
The construction is due to Massuyeau-Turaev \cite{MT14}, and an explicit formula for the bracket can be found in \cite[Thm.~3.5]{F2}.

\begin{exmp}[$g=0,r=1$] \label{Ex:g0r1}
On $L_{0,1}=\CC[ z^{\pm 1}]$, we have the double quasi-Poisson bracket
\begin{equation} \label{Eq:dqbr-z}
\dgal{z,z}= \frac12(z^2\otimes 1 - 1 \otimes z^2),
\end{equation}
associated with the moment map $\Phi=z$.
The involutive anti-automorphism $\phi:L_{0,1}\to L_{0,1}$ satisfying $z\mapsto z^{-1}$
is bracket-compatible for \eqref{Eq:dqbr-z} since
\begin{align*}
 \dgal{\phi(z),\phi(z)}^\circ
&=(z^{-1}\ast z^{-1}\dgal{z,z}z^{-1}\ast z^{-1})^\circ
=\frac12(z^{-2}\otimes 1 - 1 \otimes z^{-2})
=\phi^{\otimes 2}(\dgal{z,z}).
\end{align*}
\end{exmp}

\begin{exmp}[$g=1,r=0$] \label{Ex:g1r0}
On $L_{1,0}=\CC\langle x^{\pm 1},y^{\pm 1}\rangle$, we have the double quasi-Poisson bracket
\begin{equation} \label{Eq:dqbr-xy}
\begin{aligned}
 \dgal{x,x}&= \frac12(x^2\otimes 1 - 1 \otimes x^2), \quad
 \dgal{y,y}= -\frac12(y^2\otimes 1 - 1 \otimes y^2), \\
 \dgal{x,y}&= \frac12(yx\otimes 1 + 1\otimes xy - x \otimes y + y\otimes x),
\end{aligned}
\end{equation}
associated with the moment map $\Phi=[x,y]_\mathrm{m}:=xyx^{-1}y^{-1}$.
The involutive anti-automorphism $\phi:L_{1,0}\to L_{1,0}$ satisfying $x\mapsto x^{-1}$, $y\mapsto y^{-1}$,
is bracket-compatible for \eqref{Eq:dqbr-xy}.
Indeed, we can readily check \eqref{Eq:BrComp} for $a=b=x$ or $a=b=y$ as in Example \ref{Ex:g0r1}.
Then it remains to compute
\begin{align*}
& \dgal{\phi(x),\phi(y)}^\circ
=(x^{-1}\ast y^{-1} \dgal{x,y}y^{-1} \ast x^{-1})^\circ \\
&=\frac12(1\otimes x^{-1}y^{-1} + y^{-1}x^{-1}\otimes 1 - y^{-1} \otimes x^{-1} + x^{-1}\otimes y^{-1})^\circ \\
&=\frac12(\phi(yx) \otimes 1 + 1\otimes \phi(xy) - \phi(x)\otimes \phi(y) + \phi(y) \otimes \phi(x))
=\phi^{\otimes 2}(\dgal{x,y}).
\end{align*}
\end{exmp}

\begin{exmp}[General case] \label{Ex:grGen}
The involutive anti-automorphism $\phi:L_{g,r}\to L_{g,r}$ satisfying $x_i\mapsto x_i^{-1}$, $y_i\mapsto y_i^{-1}$, for $1\leq i\leq g$ and $z_k\mapsto z_k^{-1}$ for $1\leq k \leq r$ is bracket-compatible for
the Massuyeau-Turaev double quasi-Poisson bracket \cite{MT14}.
This follows from  Examples \ref{Ex:g0r1} and \ref{Ex:g1r0}, because the double bracket on $L_{g,r}$ is obtained by fusion from those simpler cases and bracket-compatibility is preserved under fusion, see Proposition \ref{Pr:CompFus} below.
Note that the moment map is then
$\Phi=[x_1,y_1]_\mathrm{m}\cdots [x_g,y_g]_\mathrm{m} z_1\cdots z_r$.
\end{exmp}

We finish with some preliminary results.

\begin{prop}[Fusion preserves bracket-compatibility]\label{Pr:CompFus}
Let $(A,\dgal{-,-})$ be a double quasi-Poisson algebra over $A_0=\oplus_{s\in I} \CC e_s$ ($|I|\geq 2$)
with bracket-compatible anti-involution $\phi$.
Let $A^f$ be the algebra obtained by fusing $e_2$ onto $e_1$, equipped with the double quasi-Poisson bracket $\dgal{-,-}^f$ \eqref{dgalf} obtained in Proposition \ref{PropIndbr}.
Then the anti-involution $\phi_f$ on $A^f$ induced by $\phi$  is  bracket-compatible.
\end{prop}
\begin{proof}
Recall that $A^f$ is generated by elements of the four kinds \eqref{type1}--\eqref{type4}.
By construction, $\phi$ induces the anti-involution $\phi_f:A^f\to A^f$ defined on those generators by \eqref{Eq:phiExt}.
We are left to check the compatibility of $\phi_f$ and $\dgal{-,-}^f$, which amounts to checking in view of \eqref{Eq:BrComp} and  \eqref{dgalf},
\begin{equation} \label{Eq:FusComp}
 \phi_f^{\otimes 2}(\dgal{a,b})+ \phi_f^{\otimes 2}(\dgal{a,b}_{\fus})
 =\dgal{\phi_f(a),\phi_f(b)}^\circ + \dgal{\phi_f(a),\phi_f(b)}_{\fus}^\circ \qquad \forall a,b\in A^f.
\end{equation}
This condition must be checked for the four kinds of generators of $A^f$.
In fact, we will show that the first terms on both sides of \eqref{Eq:FusComp} agree, and this is also the case for the second terms.

We consider the case where $a=t$ is a generator of first kind \eqref{type1} and $b=e_{12}we_{21}$ is a generator of fourth kind \eqref{type4};
all the other cases are similar and collected in \ref{App:Fus}.
On the one hand, $\dgal{t,e_{12}we_{21}}\in e_{12}A \epsilon \otimes \epsilon A e_{21}$, so the tensor factors are generators of the second and third kinds. Thus
\begin{align*}
\phi_f^{\otimes 2}(\dgal{t,e_{12}we_{21}})
= e_{12}\ast(\phi^{\otimes 2}(\dgal{t,w}))\ast e_{21}
&= e_{12}\ast(\dgal{\phi(t),\phi(w)}^\circ)\ast e_{21} \\
&=(e_{12}\dgal{\phi(t),\phi(w)} e_{21})^\circ
= \dgal{\phi_f(t),\phi_f(e_{12} w e_{21})}^\circ
\end{align*}
where in the second equality we used the bracket-compatibility in $A$,
and in the first/fourth equality we used the $\Mat_2(\CC)$-linearity of the induced double bracket.
On the other hand, we have
\begin{align*}
&\phi_f^{\otimes 2}( \dgal{t , e_{12} w e_{21}}_{\fus}) \\
=&\frac12 \left( \phi(t)e_{12}\phi(w)e_{21} \otimes e_1 + e_1 \otimes e_{12} \phi(w) e_{21} \phi(t)
- e_{12} \phi(w) e_{21} \otimes e_1 \phi(t) - \phi(t) e_1 \otimes e_{12} \phi(w) e_{21}\right) \\
=&\frac12 \left( e_{12} \phi(w) e_{21}\phi(t) \otimes e_1 + e_1 \otimes \phi(t) e_{12} \phi(w) e_{21}
- e_{12} \phi(w) e_{21} \otimes \phi(t) e_1 - e_1 \phi(t) \otimes e_{12} \phi(w) e_{21}\right)^\circ \\
=&\dgal{\phi_f( t) , \phi_f(e_{12} w e_{21})}_{\fus}^\circ
\end{align*}
where we used \cite[(2.14d)]{F2} and the fact that $\phi_f$ is an anti-involution in the first equality,
and we used again \cite[(2.14d)]{F2} in the last equality.
\end{proof}

\begin{prop} \label{Pr:Phiselect}
 Let $(A,\dgal{-,-},\Phi)$ be a double quasi-Hamiltonian algebra over $A_0=\oplus_{s\in I} \CC e_s$
 with bracket-compatible anti-involution $\phi$.
Then, up to multiplication $\Phi \mapsto c \Phi$ with $c\in A_0^\times$,
we can assume
 \begin{equation}\label{phim-select}
 \phi(\Phi)\,\Phi=1 \,.
\end{equation}
\end{prop}
\begin{proof}
First, we combine \eqref{Phim} and bracket-compatibility \eqref{Eq:BrComp} to get for any $b\in A$,
\begin{equation} \label{Phim-tw}
 \dgal{\phi(\Phi_s),b}=\frac12 (\phi(\Phi_s)\otimes e_s b - b \phi(\Phi_s) \otimes e_s
 +  e_s \otimes \phi(\Phi_s) b - b e_s\otimes  \phi(\Phi_s) )\,.
\end{equation}
Making use of \eqref{Eq:inder}, this entails
\begin{equation*}
 \dgal{\phi(\Phi_s)^{-1},b}=\frac12 (be_s\otimes \phi(\Phi_s)^{-1}-e_s \otimes \phi(\Phi_s)^{-1} b
 +  b \phi(\Phi_s)^{-1} \otimes e_s- \phi(\Phi_s)^{-1} \otimes e_s b)\,,
\end{equation*}
hence $\phi(\Phi^{-1})=\sum_{s\in I}\phi(\Phi_s)^{-1}$ satisfies \eqref{Phim}, i.e. it is also a moment map.
It follows from \cite[Lem.~4.3]{F3} that $\phi(\Phi^{-1})=c\Phi$ for some $c\in A_0^\times$, hence $\phi(\Phi)=c^{-1}\Phi^{-1}$.
The same result also gives that $\Phi':=\sqrt{c}\Phi$ is a moment map, and one easily sees that
$\phi(\Phi')\,\Phi'=1$.
Thus, up to constant multiplication, we can pick the moment map such that \eqref{phim-select} holds.
\end{proof}

\subsection{The orthogonal case (types \texorpdfstring{$\typB,\typD$}{B,D})}

\begin{thm}[Orthogonal quasi-Poisson case] \label{Thm:BrqPO}
 Let $(A,\dgal{-,-})$ be a double quasi-Poisson algebra over the semi-simple algebra $A_0=\oplus_{s\in I} \CC e_s$.
Assume that $\phi:A\to A$ is an involutive anti-automorphism compatible with $\dgal{-,-}$ in the sense of \eqref{Eq:BrComp}.
For any dimension vector $\alpha\in \N^I$, $N=\sum_{s\in I}\alpha_s$, considering the involution $\tau$ on $\gl_N$
given by transposition, the antisymmetric biderivation $\br{-,-}_{\phi,\Orm}$ on $\Rep^{\phi,\tau}(A,\alpha)$ uniquely determined by \eqref{Eq:BrPO}
is a quasi-Poisson bracket for the action of $\Orm_\alpha:=\prod_{s\in I} \Orm(\alpha_s)$.

Furthermore, if $\Phi \in A^\times$ is a moment map chosen to satisfy \eqref{phim-select}, then
$\X(\Phi):\Rep^{\phi,\tau}(A,\alpha) \to \Orm_\alpha$ is a moment map for $\br{-,-}_{\phi,\Orm}$.
\end{thm}
\begin{proof}
 For the first statement, it suffices to check \eqref{Jac-qP} when evaluated on the triple of generators $a_{ij},b_{kl},c_{uv}$ with $a,b,c\in A$ and $1\leq i,j,k,l,u,v\leq N$.
The LHS can be read from\footnote{There is a factor $(1/2)^2$ coming from the choice of constant in \eqref{Eq:BrPO}. See \eqref{Eq:JacType} and Remark \ref{Rem:TypeDegen} for a general proof.} \cite[(3.5)]{OS} with $x=E_{ij}^\ast$, $y=E_{kl}^\ast$, $z=E_{uv}^\ast$ and $\tau$ given by transposition as
\begin{equation}
 \begin{aligned} \label{Eq:pfqPO1}
\Jac_{\br{-,-}_{\phi,\Orm}} ( a_{ij},b_{kl},c_{uv}) =& \quad
\frac14\dgal{a,b,c}_{uj,il,kv} -\frac14 \dgal{a,c,b}_{kj,iv,ul} \\
&+\frac14 \dgal{\phi(a),b,c}_{ui,jl,kv} -\frac14 \dgal{\phi(a),c,b}_{ki,jv,ul} \\
&+\frac14\dgal{a,\phi(b),c}_{uj,ik,lv} -\frac14 \dgal{a,c,\phi(b)}_{lj,iv,uk} \\
&+ \frac14\dgal{a,b,\phi(c)}_{vj,il,ku} -\frac14 \dgal{a,\phi(c),b}_{kj,iu,vl} \,.
 \end{aligned}
\end{equation}
Since $\dgal{-,-}$ is quasi-Poisson, expressions in the RHS of \eqref{Eq:pfqPO1} can be explicitly written using \eqref{qPabc}.

Meanwhile, we get by summing the various Cartan trivectors \eqref{Eq:CartO} of the different copies $\og(\alpha_s)$ that
\begin{align*}
\frac12\psi_M^{\og_\alpha} (a_{ij},b_{kl},c_{uv})
 =\frac{1}{16}\sum_{s\in I}\sum_{i_s,j_s,u_s} \, (F_{i_sj_s}^{(s)})_{M}(a_{ij}) \,
 (F_{u_sj_s}^{(s)})_{M}(b_{kl}) \, (F_{u_si_s}^{(s)})_{M}(c_{uv})\,,
\end{align*}
where $\xi_M$ is the infinitesimal action of $\xi\in \og_\alpha$ on $M:=\Rep^{\phi,\tau}(A,\alpha)$, and we sum over
$i_s,j_s,u_s\in\mathtt{R}_s$ \eqref{Eq:sIndices} for $s\in I$.
For $F_{uv}^{(s)}$ as in Example \ref{Ex:Ortho} and
the infinitesimal action induced by \eqref{InfAct}, we get
\begin{equation} \label{InfAct-O}
 (F_{uv}^{(s)})_{M}(a_{ij}) = [\X(a),E_{uv}]_{ij}-[\X(a),E_{vu}]_{ij}\,,
\end{equation}
and then
\begin{equation}
\begin{aligned}  \label{Eq:pfqPO2}
\frac12\psi^{\og_\alpha} (a_{ij},b_{kl},c_{uv})
 =\frac{1}{16}&\sum_{s\in I}\sum_{i_s,j_s,u_s} \, ([\X(a),E_{i_sj_s}]_{ij}-[\X(a),E_{j_si_s}]_{ij}) \,
\\
&\,\,
 ([\X(b),E_{u_sj_s}]_{kl}-[\X(b),E_{j_su_s}]_{kl}) \,  ([\X(c),E_{u_si_s}]_{uv}-[\X(c),E_{i_su_s}]_{uv}) \,.
\end{aligned}
\end{equation}
The eight possibilities to take factors in \eqref{Eq:pfqPO2} will match the eight terms appearing on the RHS of \eqref{Eq:pfqPO1}. Let us do one such case in details. One has
\begin{align}
\eqref{Eq:pfqPO2}_A:=&\frac{1}{16}\sum_{s\in I}\sum_{i_s,j_s,u_s} \, [\X(a),E_{i_sj_s}]_{ij} \,
 [\X(b),E_{u_sj_s}]_{kl} \,  [\X(c),E_{u_si_s}]_{uv}\, \nonumber \\
 =& \frac{1}{16}\sum_{s\in I}\sum_{i_s,j_s,u_s}
(a_{ii_s} \delta_{j_sj}-\delta_{ii_s} a_{j_sj})
(b_{ku_s} \delta_{j_sl}-\delta_{ku_s} b_{j_sl})
(c_{uu_s} \delta_{i_sv}-\delta_{uu_s} c_{i_sv}) \,. \label{Eq:pfqPO-A1}
\end{align}
As $(e_s)_{i_sj_s}=\delta_{i_sj_s}$ when both $i_s,j_s$ are in $\mathtt{R}_s$ \eqref{Eq:sIndices} and is zero otherwise, \eqref{Eq:pfqPO-A1} can be rewritten
\begin{align*}
\eqref{Eq:pfqPO2}_A=&
\frac{1}{16}\sum_{s\in I}\sum_{i_s,j_s,u_s=1}^N
((ae_s)_{ii_s} (e_s)_{j_sj}-(e_s)_{ii_s} (e_sa)_{j_sj})   \\
&\qquad
((be_s)_{ku_s} (e_s)_{j_sl}-(e_s)_{ku_s} (e_sb)_{j_sl})  ((ce_s)_{uu_s} (e_s)_{i_sv}-(e_s)_{uu_s} (e_sc)_{i_sv}) \,.
\end{align*}
In $\CC[\Rep^{\phi,\tau}(A,\alpha)]$, one has $a_{ij}=\phi(a)_{ji}$ and $e_{ij}=e_{ji}$ ($a\in A$, $e\in A_0$, $1\leq i,j\leq N$), cf. \eqref{Eq:DefRelO}, so that
\begin{equation}
 \begin{aligned}
\eqref{Eq:pfqPO2}_A
 = \frac{1}{16}\sum_{s\in I}& \Big(
(ae_s)_{iv} (e_s)_{lj} (ce_s\phi(b))_{uk}  
-(e_s)_{iv} (e_sa)_{lj} (ce_s\phi(b))_{uk} \\ 
&- (ae_s)_{iv} (\phi(b)e_s)_{lj} (ce_s)_{uk}  
+ (e_s)_{iv} (\phi(b)e_sa)_{lj} (ce_s)_{uk}  
\\
&-(ae_sc)_{iv} (e_s)_{lj} (e_s\phi(b))_{uk}  
+(e_sc)_{iv} (e_sa)_{lj} (e_s\phi(b))_{uk} \\ 
&+(ae_sc)_{iv} (\phi(b)e_s)_{lj} (e_s)_{uk}  
-(e_sc)_{iv} (\phi(b) e_sa)_{lj} (e_s)_{uk}  
\Big) .  \label{Eq:pfqPO-A2}
\end{aligned}
\end{equation}

Meanwhile, we can write thanks to \eqref{qPabc},
\begin{equation}
 \begin{aligned}  \label{Eq:pfqPO-A3}
  &-\frac14 \dgal{a,c,\phi(b)}_{lj,iv,uk} \\
  =&  \frac{-1}{16} \sum_{s\in I} \Big(
\phi(b) e_s a \otimes e_s c \otimes e_s  - \phi(b) e_s a \otimes e_s \otimes c e_s - \phi(b) e_s \otimes a e_s c \otimes e_s
+ \phi(b) e_s \otimes a e_s \otimes c e_s \\
&\quad - e_s a \otimes e_s c \otimes e_s \phi(b) + e_s a \otimes e_s \otimes c e_s \phi(b)
+ e_s \otimes a e_s c \otimes e_s \phi(b) - e_s \otimes a e_s \otimes c e_s \phi(b) \Big)_{lj,iv,uk}
 \end{aligned}
\end{equation}
Since \eqref{Eq:pfqPO-A2} and \eqref{Eq:pfqPO-A3} agree, we get that
$\eqref{Eq:pfqPO2}_A$ is the sixth term in \eqref{Eq:pfqPO1}.
Next, we compute
\begin{align*}
\eqref{Eq:pfqPO2}_B:=&
\frac{-1}{16}\sum_{s\in I}\sum_{i_s,j_s,u_s} \, [\X(a),E_{j_si_s}]_{ij} \,
 [\X(b),E_{u_sj_s}]_{kl} \,  [\X(c),E_{u_si_s}]_{uv} \\
 =& \frac{1}{16}
 \sum_{s\in I} \Big(
(\phi(c) e_s a)_{vj} (e_s b)_{il} (e_s)_{ku}  - (\phi(c) e_s a)_{vj} (e_s)_{il} (b e_s)_{ku}
 - (\phi(c) e_s)_{vj}  (a e_s b)_{il} (e_s)_{ku} \\
&\qquad + (\phi(c) e_s)_{vj} (a e_s)_{il} (b e_s)_{ku}
 - (e_s a)_{vj} (e_s b)_{il} (e_s \phi(c))_{ku} + (e_s a)_{vj} (e_s)_{il} (b e_s \phi(c))_{ku} \\
&\qquad + (e_s)_{vj} (a e_s b)_{il} (e_s \phi(c))_{ku} - (e_s)_{vj} (a e_s)_{il} (b e_s \phi(c))_{ku} \Big)
\end{align*}
which is $\frac14 \dgal{a,b,\phi(c)}_{vj,il,ku}$ by \eqref{qPabc}, i.e. this is the seventh term of \eqref{Eq:pfqPO1};
\begin{align*}
\eqref{Eq:pfqPO2}_C:=&
\frac{-1}{16}\sum_{s\in I}\sum_{i_s,j_s,u_s} \, [\X(a),E_{i_sj_s}]_{ij} \,
 [\X(b),E_{j_su_s}]_{kl} \,  [\X(c),E_{u_si_s}]_{uv} \\
 =& \frac{-1}{16}
 \sum_{s\in I} \Big(
(b e_s a)_{kj} (e_s c)_{iv} (e_s)_{ul}  - (b e_s a)_{kj} (e_s)_{iv} (c e_s)_{ul}
 - (b e_s)_{kj}  (a e_s c)_{iv} (e_s)_{ul} \\
&\qquad + (b e_s)_{kj} (a e_s)_{iv} (c e_s)_{ul}
 - (e_s a)_{kj} (e_s c)_{iv} (e_s b)_{ul} + (e_s a)_{kj} (e_s)_{iv} (c e_s b)_{ul} \\
&\qquad + (e_s)_{kj} (a e_s c)_{iv} (e_s b)_{ul} - (e_s)_{kj} (a e_s)_{iv} (c e_s b)_{ul} \Big)
\end{align*}
which  is $-\frac14 \dgal{a,c,b}_{kj,iv,ul}$ by \eqref{qPabc}, i.e. this is  the second term of \eqref{Eq:pfqPO1};
\begin{align*}
\eqref{Eq:pfqPO2}_D:=&
\frac{-1}{16}\sum_{s\in I}\sum_{i_s,j_s,u_s} \, [\X(a),E_{i_sj_s}]_{ij} \,
 [\X(b),E_{u_sj_s}]_{kl} \,  [\X(c),E_{i_su_s}]_{uv} \\
 =& \frac{1}{16}
 \sum_{s\in I} \Big(
(c e_s \phi(a))_{ui} (e_s b)_{jl} (e_s)_{kv}  - (c e_s \phi(a))_{ui} (e_s)_{jl} (b e_s)_{kv}
 - (c e_s)_{ui}  (\phi(a) e_s b)_{jl} (e_s)_{kv} \\
&\qquad + (c e_s)_{ui} (\phi(a) e_s)_{jl} (b e_s)_{kv}
 - (e_s \phi(a))_{ui} (e_s b)_{jl} (e_s c)_{kv} + (e_s \phi(a))_{ui} (e_s)_{jl} (b e_s c)_{kv} \\
&\qquad + (e_s)_{ui} (\phi(a) e_s b)_{jl} (e_s c)_{kv} - (e_s)_{ui} (\phi(a) e_s)_{jl} (b e_s c)_{kv} \Big)
\end{align*}
which is $\frac14 \dgal{\phi(a),b,c}_{ui,jl,kv}$ by \eqref{qPabc}, i.e. this is the third term of \eqref{Eq:pfqPO1};
\begin{align*}
\eqref{Eq:pfqPO2}_E:=&
\frac{1}{16}\sum_{s\in I}\sum_{i_s,j_s,u_s} \, [\X(a),E_{j_si_s}]_{ij} \,
 [\X(b),E_{j_su_s}]_{kl} \,  [\X(c),E_{u_si_s}]_{uv} \\
 =& \frac{-1}{16}
 \sum_{s\in I} \Big(
(b e_s \phi(a))_{ki} (e_s c)_{jv} (e_s)_{ul}  - (b e_s \phi(a))_{ki} (e_s)_{jv} (c e_s)_{ul}
 - (b e_s)_{ki}  (\phi(a) e_s c)_{jv} (e_s)_{ul} \\
&\qquad + (b e_s)_{ki} (\phi(a) e_s)_{jv} (c e_s)_{ul}
 - (e_s \phi(a))_{ki} (e_s c)_{jv} (e_s b)_{ul} + (e_s \phi(a))_{ki} (e_s)_{jv} (c e_s b)_{ul} \\
&\qquad + (e_s)_{ki} (\phi(a) e_s c)_{jv} (e_s b)_{ul} - (e_s)_{ki} (\phi(a) e_s)_{jv} (c e_s b)_{ul} \Big)
\end{align*}
which is $-\frac14 \dgal{\phi(a),c,b}_{ki,jv,ul}$ by \eqref{qPabc}, i.e. this is the fourth term of \eqref{Eq:pfqPO1};
\begin{align*}
\eqref{Eq:pfqPO2}_F:=&
\frac{1}{16}\sum_{s\in I}\sum_{i_s,j_s,u_s} \, [\X(a),E_{j_si_s}]_{ij} \,
 [\X(b),E_{u_sj_s}]_{kl} \,  [\X(c),E_{i_su_s}]_{uv} \\
 =& \frac{1}{16}
 \sum_{s\in I} \Big(
(c e_s a)_{uj} (e_s b)_{il} (e_s)_{kv}  - (c e_s a)_{uj} (e_s)_{il} (b e_s)_{kv}
 - (c e_s)_{uj}  (a e_s b)_{il} (e_s)_{kv} \\
&\qquad + (c e_s)_{uj} (a e_s)_{il} (b e_s)_{kv}
 - (e_s a)_{uj} (e_s b)_{il} (e_s c)_{kv} + (e_s a)_{uj} (e_s)_{il} (b e_s c)_{kv} \\
&\qquad + (e_s)_{uj} (a e_s b)_{il} (e_s c)_{kv} - (e_s)_{uj} (a e_s)_{il} (b e_s c)_{kv} \Big)
\end{align*}
which is  $\frac14 \dgal{a,b,c}_{uj,il,kv}$ by \eqref{qPabc}, i.e. this is the first term of \eqref{Eq:pfqPO1};
\begin{align*}
\eqref{Eq:pfqPO2}_G:=&
\frac{1}{16}\sum_{s\in I}\sum_{i_s,j_s,u_s} \, [\X(a),E_{i_sj_s}]_{ij} \,
 [\X(b),E_{j_su_s}]_{kl} \,  [\X(c),E_{i_su_s}]_{uv} \\
 =& \frac{-1}{16}
 \sum_{s\in I} \Big(
(b e_s a)_{kj} (e_s \phi(c))_{iu} (e_s)_{vl}  - (b e_s a)_{kj} (e_s)_{iu} (\phi(c) e_s)_{vl}
 - (b e_s)_{kj}  (a e_s \phi(c))_{iu} (e_s)_{vl} \\
&\qquad + (b e_s)_{kj} (a e_s)_{iu} (\phi(c) e_s)_{vl}
 - (e_s a)_{kj} (e_s \phi(c))_{iu} (e_s b)_{vl} + (e_s a)_{kj} (e_s)_{iu} (\phi(c) e_s b)_{vl} \\
&\qquad + (e_s)_{kj} (a e_s \phi(c))_{iu} (e_s b)_{vl} - (e_s)_{kj} (a e_s)_{iu} (\phi(c) e_s b)_{vl} \Big)
\end{align*}
which is  $-\frac14 \dgal{a,\phi(c),b}_{kj,iu,vl}$ by \eqref{qPabc}, i.e. this is the eighth term of \eqref{Eq:pfqPO1};
\begin{align*}
\eqref{Eq:pfqPO2}_H:=&
\frac{-1}{16}\sum_{s\in I}\sum_{i_s,j_s,u_s} \, [\X(a),E_{j_si_s}]_{ij} \,
 [\X(b),E_{j_su_s}]_{kl} \,  [\X(c),E_{i_su_s}]_{uv} \\
 =& \frac{1}{16}
 \sum_{s\in I} \Big(
(c e_s a)_{uj} (e_s \phi(b))_{ik} (e_s)_{lv}  - (c e_s a)_{uj} (e_s)_{ik} (\phi(b) e_s)_{lv}
 - (c e_s)_{uj}  (a e_s \phi(b))_{ik} (e_s)_{lv} \\
&\qquad + (c e_s)_{uj} (a e_s)_{ik} (\phi(b) e_s)_{lv}
 - (e_s a)_{uj} (e_s \phi(b))_{ik} (e_s c)_{lv} + (e_s a)_{uj} (e_s)_{ik} (\phi(b) e_s c)_{lv} \\
&\qquad + (e_s)_{uj} (a e_s \phi(b))_{ik} (e_s c)_{lv} - (e_s)_{uj} (a e_s)_{ik} (\phi(b) e_s c)_{lv} \Big)
\end{align*}
which is $\frac14 \dgal{a,\phi(b),c}_{uj,ik,lv}$ by \eqref{qPabc}, i.e. this is the fifth term of \eqref{Eq:pfqPO1}.
Summing up all these expressions, we get \eqref{Eq:pfqPO1}=\eqref{Eq:pfqPO2} as desired.

\medskip

Assume that  $\Phi \in A^\times$ is a moment map, so that \eqref{Phim} holds.
Let us first check that $\X(\Phi)$ is $\Orm_\alpha$-valued\footnote{
Recall from Footnote \ref{ftnote:TypeO} that we can replace some copies of $\Orm(\alpha_s)$ with $\operatorname{SO}(\alpha_s)$.
In such a case, if $\det(\X(\Phi_s))=-1$, it suffices to multiply $\Phi_s$ by a $\alpha_s$-th root of $-1$ to get $\det(\X(\Phi_s))=+1$.
Such multiplications are possible without breaking the moment map property by \cite[Lem.~4.3]{F3}.%
}.
By Proposition \ref{Pr:Phiselect}, recall that we can assume $\phi(\Phi)\, \Phi=1$.
Making use of \eqref{Eq:DefRelO}, $\X(\Phi)$ is indeed $\Orm_\alpha$-valued because
$\X(\phi(\Phi)\, \Phi) = \X(\Phi)^T\,\X(\Phi) = \Id_N$ on $\Rep^{\phi,\tau}(A,\alpha)$.

Next, we have to check \eqref{Gmomap} where $\Xi=\X(\Phi)$ and, without loss of generality, we take the function $g_{i_sj_s}:\Orm_\alpha\to \CC$ with $i_s,j_s\in\mathtt{R}_s$ \eqref{Eq:sIndices} returning the $(i_s,j_s)$-entry.
This is equivalent to establishing
\begin{equation} \label{Eq:pfqPO5}
  \br{(\Phi_s)_{i_sj_s},b_{kl}}_{\phi,\Orm}
= -\frac14 \sum_{t\in I}\sum_{u_t<v_t} \X(\Phi)^\ast\big((F_{u_tv_t}^L+F_{u_tv_t}^R)(g_{i_sj_s})\big)\, (F_{u_tv_t})_M(b_{kl})\,,
\end{equation}
after using the dual bases from Example \ref{Ex:Ortho}, and where we sum over $u_t,v_t\in\mathtt{R}_t$ \eqref{Eq:sIndices} subject to $u_t<v_t$.
(The infinitesimal action is taken on $M:=\Rep^{\phi,\tau}(A,\alpha)$.)
For the LHS,
\eqref{Phim-tw} and \eqref{Eq:BrPO} yield
\begin{equation} \label{Eq:BrqPO6}
\begin{aligned}
\br{(\Phi_s)_{i_sj_s},b_{kl}}_{\phi,\Orm}=&\frac12 \dgal{\Phi_s,b}_{kj_s,i_sl} + \frac12 \dgal{\phi(\Phi_s),b}_{ki_s,j_sl} \\
=& \frac14 \big(be_s\otimes \Phi_s-e_s \otimes \Phi_s b +  b \Phi_s \otimes e_s-\Phi_s \otimes e_s b\big)_{kj_s,i_sl} \\
&+\frac14 \big(\phi(\Phi_s)\otimes e_s b - b \phi(\Phi_s) \otimes e_s
 +  e_s \otimes \phi(\Phi_s) b - b e_s\otimes  \phi(\Phi_s)\big)_{ki_s,j_sl}\,.
\end{aligned}
\end{equation}

For the RHS,
we need the following equalities on $\Orm_\alpha$ which are obtained similarly to the $\Gl_\alpha$-case at the beginning of the proof of \cite[Prop.~7.13.2]{VdB1}:
\begin{equation*}
F_{u_tv_t}^L(g_{i_sj_s})=g_{i_su_t}\delta_{v_tj_s} - g_{i_sv_t}\delta_{u_tj_s}, \quad
F_{u_tv_t}^R(g_{i_sj_s})=\delta_{i_su_t} g_{v_tj_s} - \delta_{i_sv_t}g_{u_tj_s}.
\end{equation*}
Thanks to \eqref{InfAct-O}, we find
\begin{align*} \eqref{Eq:pfqPO5}_{RHS}
=& -\frac14 \sum_{t\in I}\sum_{u_t<v_t}
\big(\Phi_{i_su_t}\delta_{v_tj_s} - \Phi_{i_sv_t}\delta_{u_tj_s}
+ \delta_{i_su_t} \Phi_{v_tj_s} - \delta_{i_sv_t}\Phi_{u_tj_s}\big)\\
&\qquad \qquad (b_{ku_t} \delta_{v_tl} - b_{kv_t} \delta_{u_tl}
- \delta_{ku_t} b_{v_tl} + \delta_{kv_t} b_{u_tl}) \,, \\
= -\frac14 \sum_{u_s,v_s} &
\big((\Phi_s)_{i_su_s} (e_s)_{v_sj_s}
+ (e_s)_{i_su_s} (\Phi_s)_{v_sj_s} \big)\,
 (b_{ku_s} \delta_{v_sl} - b_{kv_s} \delta_{u_sl}
- \delta_{ku_s} b_{v_sl} + \delta_{kv_s} b_{u_sl}) \,,
\end{align*}
where we sum over unordered pairs $u_s,v_s\in\mathtt{R}_s$ \eqref{Eq:sIndices} for the fixed $s\in I$ we started with.
Summing together these various terms and remembering that $\phi(\Phi)_{ij}=\Phi_{ji}$ for $1\leq i,j\leq N$, we precisely get \eqref{Eq:BrqPO6}.
Hence \eqref{Eq:pfqPO5} holds.
\end{proof}

\subsection{The symplectic case (type \texorpdfstring{$\typC$}{C})}

We use all the notations introduced in \ref{sss:Not-symp}.

\begin{thm}[Symplectic quasi-Poisson case] \label{Thm:BrqPsp}
 Let $(A,\dgal{-,-})$ be a double quasi-Poisson algebra over the semi-simple algebra $A_0=\oplus_{s\in I} \CC e_s$.
Assume that $\phi:A\to A$ is an involutive anti-automorphism compatible with $\dgal{-,-}$ in the sense of \eqref{Eq:BrComp}.
For any dimension vector $\alpha\in 2\N^I$, $N=\sum_{s\in I}\alpha_s$, considering the involution $\tau$ on $\gl_N$
given by $\xi \mapsto \Omega \xi^T \Omega^T$ (for $\Omega$ in \eqref{Eq:omega}),
the antisymmetric biderivation $\br{-,-}_{\phi,\Sp}$ on $\Rep^{\phi,\tau}(A,\alpha)$ uniquely determined by \eqref{Eq:BrPsymp}
is a quasi-Poisson bracket for the action of $\Sp_\alpha :=\prod_{s\in I} \Sp(\alpha_s)$.

Furthermore, if $\Phi \in A^\times$ is a moment map chosen to satisfy \eqref{phim-select}, then
$\X(\Phi):\Rep^{\phi,\tau}(A,\alpha) \to \Sp_\alpha$ is a moment map for $\br{-,-}_{\phi,\Sp}$.
\end{thm}
\begin{proof}
 For the first statement, it suffices to check \eqref{Jac-qP} when evaluated on the triple of generators $a_{ij},b_{kl},c_{uv}$ with $a,b,c\in A$ and $1\leq i,j,k,l,u,v\leq N$.
Similarly to the orthogonal case, the LHS can be read from \cite[(3.5)]{OS} (or Remark \ref{Rem:TypeDegen})
by making use of \eqref{Eq:TauSp} and Lemma \ref{lem:TauSp}. It yields
\begin{equation}
 \begin{aligned} \label{Eq:pfqPsp1}
\Jac_{\br{-,-}_{\phi,\Sp}} &( a_{ij},b_{kl},c_{uv}) =
\frac14\dgal{a,b,c}_{uj,il,kv} -\frac14 \dgal{a,c,b}_{kj,iv,ul} \\
&+\frac14 \sgn_\alpha(i,j) \, (\dgal{\phi(a),b,c}_{u\tau(i),\tau(j)l,kv}
-\dgal{\phi(a),c,b}_{k\tau(i),\tau(j)v,ul} )\\
&+\frac14 \sgn_\alpha(k,l) \, (\dgal{a,\phi(b),c}_{uj,i\tau(k),\tau(l)v}
- \dgal{a,c,\phi(b)}_{\tau(l)j,iv,u\tau(k)} )\\
&+ \frac14 \sgn_\alpha(u,v) \, (\dgal{a,b,\phi(c)}_{\tau(v)j,il,k\tau(u)}
- \dgal{a,\phi(c),b}_{kj,i\tau(u),\tau(v)l} ) \,.
 \end{aligned}
\end{equation}
Since $\dgal{-,-}$ is quasi-Poisson, expressions in the RHS of \eqref{Eq:pfqPO1} can be explicitly written using \eqref{qPabc}.

Meanwhile, we get by summing the various Cartan trivectors \eqref{Eq:CartSp} of the different copies $\spg(\alpha_s)$ that
\begin{align*}
\frac12\psi_M^{\spg_\alpha} (a_{ij},b_{kl},c_{uv})
 =\frac{1}{32}\sum_{s\in I}\sum_{i_s,j_s,u_s} \, \Big(
 &(F_{i_sj_s}^{(s)})_{M}(a_{ij}) \,(F_{u_si_s}^{(s)})_{M}(b_{kl}) \, (F_{j_su_s}^{(s)})_{M}(c_{uv}) \\
&- (F_{i_sj_s}^{(s)})_{M}(a_{ij}) \,  (F_{j_su_s}^{(s)})_{M}(b_{kl}) \, (F_{u_si_s}^{(s)})_{M}(c_{uv}) \Big)\,,
\end{align*}
where $\xi_M$ denotes the infinitesimal action of $\xi\in \spg_\alpha$ on $M:=\Rep^{\phi,\tau}(A,\alpha)$, and for fixed $s\in I$ we sum over
$i_s,j_s,u_s  \in\mathtt{R}_s$ \eqref{Eq:sIndices}.
The infinitesimal action can be written for the basis elements $(F_{uv}^{(s)})$ \eqref{Eq:BasSp} in terms of elementary matrices as
\begin{equation} \label{InfAct-sp}
 (F_{uv}^{(s)})_{M}(a_{ij}) = [\X(a),E_{uv}]_{ij}-\sgn_\alpha(u,v)\,[\X(a),E_{\tau(v),\tau(u)}]_{ij}\,.
\end{equation}
Therefore, we can write
\begin{equation} \label{Eq:pfqPspPsi}
 \frac12\psi^{\spg_\alpha} (a_{ij},b_{kl},c_{uv}) = \Psi^{(1)} + \Psi^{(2)},
\end{equation}
with the expressions (we use the identities \eqref{Eq:sgnProp} of $\sgn_\alpha(-,-)$)
\begin{subequations}
 \begin{align}
\Psi^{(1)}:=
\frac{1}{32}\sum_{s\in I}&\sum_{i_s,j_s,u_s}\Big(
 [\X(a),E_{i_sj_s}]_{ij}\, [\X(b),E_{u_s i_s}]_{kl} \, [\X(c),E_{j_su_s}]_{uv} \label{Psi1-000} \\
&-\sgn_\alpha(i_s,j_s) \, [\X(a),E_{\tau(j_s),\tau(i_s)}]_{ij}\, [\X(b),E_{u_s i_s}]_{kl}
\, [\X(c),E_{j_su_s}]_{uv} \label{Psi1-100} \\
&-\sgn_\alpha(i_s,u_s) \, [\X(a),E_{i_sj_s}]_{ij}\, [\X(b),E_{\tau(i_s),\tau(u_s)}]_{kl}
\, [\X(c),E_{j_su_s}]_{uv} \label{Psi1-010} \\
&-\sgn_\alpha(j_s,u_s) \, [\X(a),E_{i_sj_s}]_{ij}\, [\X(b),E_{u_s i_s}]_{kl}
\, [\X(c),E_{\tau(u_s),\tau(j_s)}]_{uv} \label{Psi1-001} \\
&+ \sgn_\alpha(j_s,u_s) \, [\X(a),E_{\tau(j_s),\tau(i_s)}]_{ij}\, [\X(b),E_{\tau(i_s),\tau(u_s)}]_{kl}
\, [\X(c),E_{j_su_s}]_{uv} \label{Psi1-110} \\
&+\sgn_\alpha(i_s,u_s) \, [\X(a),E_{\tau(j_s),\tau(i_s)}]_{ij}\, [\X(b),E_{u_s i_s}]_{kl}
\, [\X(c),E_{\tau(u_s),\tau(j_s)}]_{uv} \label{Psi1-101} \\
&+\sgn_\alpha(i_s,j_s) \, [\X(a),E_{i_sj_s}]_{ij}\, [\X(b),E_{\tau(i_s),\tau(u_s)}]_{kl}
\, [\X(c),E_{\tau(u_s),\tau(j_s)}]_{uv} \label{Psi1-011} \\
&- [\X(a),E_{\tau(j_s),\tau(i_s)}]_{ij}\, [\X(b),E_{\tau(i_s),\tau(u_s)}]_{kl}
\, [\X(c),E_{\tau(u_s),\tau(j_s)}]_{uv} \Big) \,, \label{Psi1-111} \\
\Psi^{(2)}:=
\frac{1}{32}\sum_{s\in I}&\sum_{i_s,j_s,u_s}\Big(
 -[\X(a),E_{i_sj_s}]_{ij}\, [\X(b),E_{j_s u_s}]_{kl} \, [\X(c),E_{u_si_s}]_{uv} \label{Psi2-000} \\
&+\sgn_\alpha(i_s, j_s) \, [\X(a),E_{\tau(j_s),\tau(i_s)}]_{ij}\, [\X(b),E_{j_s u_s}]_{kl}
\, [\X(c),E_{u_si_s}]_{uv} \label{Psi2-100} \\
&+\sgn_\alpha(j_s,u_s) \, [\X(a),E_{i_sj_s}]_{ij}\, [\X(b),E_{\tau(u_s),\tau(j_s)}]_{kl}
\, [\X(c),E_{u_si_s}]_{uv} \label{Psi2-010} \\
&+\sgn_\alpha(i_s,u_s) \, [\X(a),E_{i_sj_s}]_{ij}\, [\X(b),E_{j_s u_s}]_{kl}
\, [\X(c),E_{\tau(i_s),\tau(u_s)}]_{uv} \label{Psi2-001} \\
&- \sgn_\alpha(i_s,u_s) \, [\X(a),E_{\tau(j_s),\tau(i_s)}]_{ij}\, [\X(b),E_{\tau(u_s),\tau(j_s)}]_{kl}
\, [\X(c),E_{u_si_s}]_{uv} \label{Psi2-110} \\
&-\sgn_\alpha(j_s,u_s) \, [\X(a),E_{\tau(j_s),\tau(i_s)}]_{ij}\, [\X(b),E_{j_s u_s}]_{kl}
\, [\X(c),E_{\tau(i_s),\tau(u_s)}]_{uv} \label{Psi2-101} \\
&-\sgn_\alpha(i_s,j_s) \, [\X(a),E_{i_sj_s}]_{ij}\, [\X(b),E_{\tau(u_s),\tau(j_s)}]_{kl}
\, [\X(c),E_{\tau(i_s),\tau(u_s)}]_{uv} \label{Psi2-011} \\
&+ [\X(a),E_{\tau(j_s),\tau(i_s)}]_{ij}\, [\X(b),E_{\tau(u_s),\tau(j_s)}]_{kl}
\, [\X(c),E_{\tau(i_s),\tau(u_s)}]_{uv} \Big) \,.\label{Psi2-111}
 \end{align}
\end{subequations}
Under relabelling indices through $(i_s,j_s,u_s)\mapsto(\tau(j_s),\tau(i_s),\tau(u_s))$, which is an involution on the summation set with fixed $s\in I$, we get the equalities:
 \eqref{Psi1-000}=\eqref{Psi2-111},
 \eqref{Psi1-100}=\eqref{Psi2-011},
 \eqref{Psi1-010}=\eqref{Psi2-101},
 \eqref{Psi1-001}=\eqref{Psi2-110},
 \eqref{Psi1-110}=\eqref{Psi2-001},
 \eqref{Psi1-101}=\eqref{Psi2-010},
 \eqref{Psi1-011}=\eqref{Psi2-100},
 \eqref{Psi1-111}=\eqref{Psi2-000}.

We can now proceed to compare these terms with \eqref{Eq:pfqPsp1}, as we did in the orthogonal case.
First of all,
\begin{align*}
&\eqref{Psi1-000}+\eqref{Psi2-111}
=
\frac{1}{16}\sum_{s\in I}\sum_{i_s,j_s,u_s}
 [\X(a),E_{i_sj_s}]_{ij}\, [\X(b),E_{u_s i_s}]_{kl} \, [\X(c),E_{j_su_s}]_{uv} \\
=& \frac{1}{16}\sum_{s\in I}\sum_{i_s,j_s,u_s}
(a_{ii_s} \delta_{j_sj}-\delta_{ii_s} a_{j_sj})
(b_{ku_s} \delta_{i_sl}-\delta_{ku_s} b_{i_sl})
(c_{uj_s} \delta_{u_sv}-\delta_{uj_s} c_{u_sv}) \\
=&  \frac{1}{16}\sum_{s\in I} \Big(
(ae_s)_{il} (be_s)_{kv} (ce_s)_{uj}  
-(ae_s)_{il} (be_sc)_{kv} (e_s)_{uj} 
- (ae_s b)_{il} (e_s)_{kv} (ce_s)_{uj} \\ 
&\qquad+ (ae_sb)_{il} (e_sc)_{kv} (e_s)_{uj}  
 -(e_s)_{il} (be_s)_{kv} (ce_sa)_{uj}  
+(e_s)_{il} (be_sc)_{kv} (e_sa)_{uj} \\ 
&\qquad+(e_sb)_{il} (e_s)_{kv} (ce_sa)_{uj}  
-(e_sb)_{il} (e_sc)_{kv} (e_sa)_{uj}  
\Big)
\end{align*}
which is $\frac14 \dgal{a,b,c}_{uj,il,kv}$ by \eqref{qPabc}, i.e. this is the first term of \eqref{Eq:pfqPsp1}. Similarly,
\begin{align*}
&\eqref{Psi1-111}+\eqref{Psi2-000}
= \frac{-1}{16}\sum_{s\in I}\sum_{i_s,j_s,u_s}
(a_{ii_s} \delta_{j_sj}-\delta_{ii_s} a_{j_sj})
(b_{kj_s} \delta_{u_sl}-\delta_{kj_s} b_{u_sl})
(c_{uu_s} \delta_{i_sv}-\delta_{uu_s} c_{i_sv}) \\
=&  \frac{-1}{16}\sum_{s\in I} \Big(
(ae_s)_{iv} (be_s)_{kj} (ce_s)_{ul}  
-(e_s)_{iv} (be_sa)_{kj} (ce_s)_{ul} 
- (ae_s)_{iv} (e_s)_{kj} (ce_sb)_{ul} \\ 
&\qquad + (e_s)_{iv} (e_sa)_{kj} (ce_sb)_{ul}  
 - (ae_sc)_{iv} (be_s)_{kj} (e_s)_{ul}  
+ (e_sc)_{iv} (be_sa)_{kj} (e_s)_{ul} \\ 
&\qquad + (ae_sc)_{iv} (e_s)_{kj} (e_sb)_{ul}  
-(e_sc)_{iv} (e_sa)_{kj} (e_sb)_{ul}  
\Big)
\end{align*}
which is $-\frac14 \dgal{a,c,b}_{kj,iv,ul}$ by \eqref{qPabc}, i.e. this is the second term of \eqref{Eq:pfqPsp1}.
To handle the next cases, recall that $\tau$ \eqref{Eq:TauSp} is such that $\sgn_\alpha(\tau(\ell))=-\sgn_\alpha(\ell)$ by \eqref{Eq:sgnAlp}, $(\sgn_\alpha(\ell))^2=+1$, and we work with the relation \eqref{Eq:DefRelSp}. This entails
\begin{equation} \label{Eq:pfqPsp2}
 \begin{aligned}
  &\sgn_\alpha(i_s,j_s) \left(a_{p,\tau(j_s)}\delta_{\tau(i_s),q}
  - \delta_{p,\tau(j_s)}a_{\tau(i_s),q}\right) \\
=&\sgn_\alpha(p,q) \left(\phi(a)_{j_s,\tau(p)}(e_s)_{\tau(q),i_s}
- (e_s)_{j_s,\tau(p)}\phi(a)_{\tau(q),i_s}\right)
 \end{aligned}
\end{equation}
for any $a\in A$, $1\leq p,q\leq N$, and $i_s,j_s\in\mathtt{R}_s$ \eqref{Eq:sIndices}. In matrix form, this reads,
\begin{equation} \label{Eq:pfqPsp3}
 \sgn_\alpha(i_s,j_s) [\X(a),E_{\tau(j_s),\tau(i_s)}]_{pq}
 =  \sgn_\alpha(p,q) [\X(\phi(a)),E_{\tau(p),\tau(q)}]_{j_si_s}\,.
\end{equation}
Next, we compute thanks to \eqref{Eq:pfqPsp3}
\begin{align*}
&\eqref{Psi1-100}+\eqref{Psi2-011}
= \frac{-1}{16}\sum_{s\in I}\sum_{i_s,j_s,u_s}
\sgn_\alpha(i,j) [\X(\phi(a)),E_{\tau(i)\tau(j)}]_{j_si_s}\,
[\X(b),E_{u_s i_s}]_{kl}\, [\X(c),E_{j_su_s}]_{uv}\\
=&\frac{-\sgn_\alpha(i,j)}{16}\sum_{s\in I}\sum_{i_s,j_s,u_s}
(\phi(a)_{j_s\tau(i)} \delta_{\tau(j)i_s} - \delta_{j_s\tau(i)} \phi(a)_{\tau(j)i_s})
(b_{ku_s} \delta_{i_sl}-\delta_{ku_s} b_{i_sl})
(c_{uj_s} \delta_{u_sv}-\delta_{uj_s} c_{u_sv}) \\
=&  \frac{-1}{16}\sgn_\alpha(i,j)\sum_{s\in I} \Big(
(ce_s\phi(a))_{u\tau(i)}  (e_s)_{\tau(j)l}  (be_s)_{kv}  
- (ce_s)_{u\tau(i)}  (\phi(a)e_s)_{\tau(j)l}  (be_s)_{kv} \\ 
&\qquad- (ce_s\phi(a))_{u\tau(i)}  (e_sb)_{\tau(j)l}  (e_s)_{kv}   
 + (ce_s)_{u\tau(i)}  (\phi(a)e_sb)_{\tau(j)l}  (e_s)_{kv}   
 - (e_s\phi(a))_{u\tau(i)}  (e_s)_{\tau(j)l}  (be_sc)_{kv} \\  
&\qquad+ (e_s)_{u\tau(i)}  (\phi(a)e_s)_{\tau(j)l}  (be_sc)_{kv}   
 + (e_s\phi(a))_{u\tau(i)}  (e_sb)_{\tau(j)l}  (e_sc)_{kv}  
- (e_s)_{u\tau(i)}  (\phi(a)e_sb)_{\tau(j)l}  (e_sc)_{kv}   
\Big)
\end{align*}
which is $\frac14 \sgn_\alpha(i,j)\,\dgal{\phi(a),b,c}_{u\tau(i),\tau(j)l,kv}$ by \eqref{qPabc}, i.e. this is the third term of  \eqref{Eq:pfqPsp1};
\begin{align*}
&\eqref{Psi1-011}+\eqref{Psi2-100}
= \frac{1}{16}\sum_{s\in I}\sum_{i_s,j_s,u_s}
\sgn_\alpha(i,j) [\X(\phi(a)),E_{\tau(i)\tau(j)}]_{j_si_s}\,
[\X(b),E_{j_s u_s}]_{kl}\, [\X(c),E_{u_si_s}]_{uv}\\
&=  \frac{1}{16} \sgn_\alpha(i,j)\sum_{s\in I} \Big(
(be_s\phi(a))_{k\tau(i)}  (e_s)_{\tau(j)v}  (ce_s)_{ul}  
- (be_s)_{k\tau(i)}  (\phi(a)e_s)_{\tau(j)v}  (ce_s)_{ul} \\ 
&\qquad- (e_s\phi(a))_{k\tau(i)}  (e_s)_{\tau(j)v}  (ce_sb)_{ul}    
 + (e_s)_{k\tau(i)}  (\phi(a)e_s)_{\tau(j)v}  (ce_sb)_{ul}   
 - (be_s\phi(a))_{k\tau(i)}  (e_sc)_{\tau(j)v}  (e_s)_{ul}  \\ 
&\qquad+ (be_s)_{k\tau(i)}  (\phi(a)e_sc)_{\tau(j)v}  (e_s)_{ul}    
 + (e_s\phi(a))_{k\tau(i)}  (e_sc)_{\tau(j)v}  (e_sb)_{ul}   
- (e_s)_{k\tau(i)}  (\phi(a)e_sc)_{\tau(j)v}  (e_sb)_{ul}    
\Big)
\end{align*}
which is  $-\frac14 \sgn_\alpha(i,j)\,\dgal{\phi(a),c,b}_{k\tau(i),\tau(j)v,ul}$ by \eqref{qPabc}, i.e. this is  the fourth term of  \eqref{Eq:pfqPsp1};
\begin{align*}
&\eqref{Psi1-010}+\eqref{Psi2-101}
= \frac{-1}{16}\sum_{s\in I}\sum_{i_s,j_s,u_s}
\sgn_\alpha(k,l) [\X(a),E_{i_sj_s}]_{ij}\,
[\X(\phi(b)),E_{\tau(k)\tau(l)}]_{i_su_s}\, [\X(c),E_{j_su_s}]_{uv}\\
&=  \frac{-1}{16}\sgn_\alpha(k,l)\sum_{s\in I} \Big(
(ce_s)_{uj}  (ae_s\phi(b))_{i\tau(k)} (e_s)_{\tau(l)v}    
-(ce_sa)_{uj}  (e_s\phi(b))_{i\tau(k)} (e_s)_{\tau(l)v}  \\ 
&\qquad- (ce_s)_{uj}  (ae_s)_{i\tau(k)} (\phi(b)e_s)_{\tau(l)v}  
+(ce_sa)_{uj}  (e_s)_{i\tau(k)} (\phi(b)e_s)_{\tau(l)v}   
- (e_s)_{uj}  (ae_s\phi(b))_{i\tau(k)} (e_sc)_{\tau(l)v} \\ 
&\qquad+ (e_sa)_{uj}  (e_s\phi(b))_{i\tau(k)} (e_sc)_{\tau(l)v}  
+ (e_s)_{uj}  (ae_s)_{i\tau(k)} (\phi(b)e_sc)_{\tau(l)v}  
- (e_sa)_{uj}  (e_s)_{i\tau(k)} (\phi(b)e_sc)_{\tau(l)v}  
 \Big)
\end{align*}
which is  $\frac14 \sgn_\alpha(k,l)\,\dgal{a,\phi(b),c}_{uj,i\tau(k),\tau(l)v}$ by \eqref{qPabc}, i.e. this is  the fifth term of  \eqref{Eq:pfqPsp1};
\begin{align*}
&\eqref{Psi1-101}+\eqref{Psi2-010}
= \frac{1}{16}\sum_{s\in I}\sum_{i_s,j_s,u_s}
\sgn_\alpha(k,l) [\X(a),E_{i_sj_s}]_{ij}\,
[\X(\phi(b)),E_{\tau(k)\tau(l)}]_{u_sj_s} [\X(c),E_{u_si_s}]_{uv}\\
&=  \frac{1}{16}\sgn_\alpha(k,l)\sum_{s\in I} \Big(
(e_s)_{\tau(l)j}  (ae_s)_{iv} (ce_s\phi(b))_{u\tau(k)}    
- (e_sa)_{\tau(l)j}  (e_s)_{iv} (ce_s\phi(b))_{u\tau(k)}    \\ 
&\qquad- (\phi(b)e_s)_{\tau(l)j}  (ae_s)_{iv} (ce_s)_{u\tau(k)}   
+ (\phi(b)e_sa)_{\tau(l)j}  (e_s)_{iv} (ce_s)_{u\tau(k)}    
- (e_s)_{\tau(l)j}  (ae_sc)_{iv} (e_s\phi(b))_{u\tau(k)}  \\ 
&\qquad+ (e_sa)_{\tau(l)j}  (e_sc)_{iv} (e_s\phi(b))_{u\tau(k)}    
+ (\phi(b)e_s)_{\tau(l)j}  (ae_sc)_{iv} (e_s)_{u\tau(k)}    
- (\phi(b)e_sa)_{\tau(l)j}  (e_sc)_{iv} (e_s)_{u\tau(k)}    
 \Big)
\end{align*}
which is $-\frac14 \sgn_\alpha(k,l)\,\dgal{a,c,\phi(b)}_{\tau(l)j,iv,u\tau(k)}$ by \eqref{qPabc}, i.e. this is the sixth term of  \eqref{Eq:pfqPsp1};
\begin{align*}
&\eqref{Psi1-001}+\eqref{Psi2-110}
= \frac{-1}{16}\sum_{s\in I}\sum_{i_s,j_s,u_s}
\sgn_\alpha(u,v) [\X(a),E_{i_sj_s}]_{ij}\,
[\X(b),E_{u_s i_s}]_{kl}\, [\X(\phi(c)),E_{\tau(u)\tau(v)}]_{u_sj_s} \\
&=  \frac{-1}{16}\sgn_\alpha(u,v)\sum_{s\in I} \Big(
(e_s)_{\tau(v)j}  (ae_s)_{il} (be_s\phi(c))_{k\tau(u)}    
-(e_sa)_{\tau(v)j}  (e_s)_{il} (be_s\phi(c))_{k\tau(u)}    \\ 
&\qquad- (e_s)_{\tau(v)j}  (ae_sb)_{il} (e_s\phi(c))_{k\tau(u)}   
+ (e_sa)_{\tau(v)j}  (e_sb)_{il} (e_s\phi(c))_{k\tau(u)}     
- (\phi(c)e_s)_{\tau(v)j}  (ae_s)_{il} (be_s)_{k\tau(u)}  \\ 
&\qquad+ (\phi(c)e_sa)_{\tau(v)j}  (e_s)_{il} (be_s)_{k\tau(u)}   
+ (\phi(c)e_s)_{\tau(v)j}  (ae_sb)_{il} (e_s)_{k\tau(u)}   
- (\phi(c)e_sa)_{\tau(v)j}  (e_sb)_{il} (e_s)_{k\tau(u)}   
 \Big)
\end{align*}
which is $\frac14 \sgn_\alpha(u,v)\,\dgal{a,b,\phi(c)}_{\tau(v)j,il,k\tau(u)}$ by \eqref{qPabc}, i.e. this is the seventh term of  \eqref{Eq:pfqPsp1};
\begin{align*}
&\eqref{Psi1-110}+\eqref{Psi2-001}
= \frac{1}{16}\sum_{s\in I}\sum_{i_s,j_s,u_s}
\sgn_\alpha(u,v) [\X(a),E_{i_sj_s}]_{ij}\,
[\X(b),E_{j_s u_s}]_{kl} \, [\X(\phi(c)),E_{\tau(u)\tau(v)}]_{i_su_s}  \\
&=  \frac{1}{16}\sgn_\alpha(u,v)\sum_{s\in I} \Big(
(be_s)_{kj}  (ae_s\phi(c))_{i\tau(u)} (e_s)_{\tau(v)l}    
- (be_sa)_{kj}  (e_s\phi(c))_{i\tau(u)} (e_s)_{\tau(v)l} \\ 
&\qquad- (e_s)_{kj}  (ae_s\phi(c))_{i\tau(u)} (e_sb)_{\tau(v)l}   
+ (e_sa)_{kj}  (e_s\phi(c))_{i\tau(u)} (e_sb)_{\tau(v)l} 
- (be_s)_{kj}  (ae_s)_{i\tau(u)} (\phi(c)e_s)_{\tau(v)l} \\ 
&\qquad+ (be_sa)_{kj}  (e_s)_{i\tau(u)} (\phi(c)e_s)_{\tau(v)l}  
+ (e_s)_{kj}  (ae_s)_{i\tau(u)} (\phi(c)e_sb)_{\tau(v)l}  
- (e_sa)_{kj}  (e_s)_{i\tau(u)} (\phi(c)e_sb)_{\tau(v)l} 
 \Big) \\
&=-\frac14 \sgn_\alpha(u,v)\,\dgal{a,\phi(c),b}_{kj,i\tau(u),\tau(v)l},
\end{align*}
which is $-\frac14 \sgn_\alpha(u,v)\,\dgal{a,\phi(c),b}_{kj,i\tau(u),\tau(v)l}$ by \eqref{qPabc}, i.e. this is the eighth term of  \eqref{Eq:pfqPsp1}.
Summing up all these expressions, we get the desired equality of \eqref{Eq:pfqPsp1} and \eqref{Eq:pfqPspPsi}.

\medskip

Assume that  $\Phi \in A^\times$ is a moment map, so that \eqref{Phim} holds.
First, the matrix-valued function  $\X(\Phi):\Rep^{\phi,\tau}(A,\alpha)\to \Gl_\alpha$ should be $\Sp_\alpha$-valued.
By assumption, $\phi(\Phi)\, \Phi=1$, so that making use of \eqref{Eq:DefRelSp}, we get
$\X(\phi(\Phi)\, \Phi) = \Omega \X(\Phi)^T \Omega^T\,\X(\Phi) = \Id_N$ on $\Rep^{\phi,\tau}(A,\alpha)$;
this proves the claim.
Next, we have to check \eqref{Gmomap} where $\Xi=\X(\Phi)$ and, without loss of generality,
we take the function $g_{i_sj_s}:\Sp_\alpha\to \CC$ with $i_s,j_s\in\mathtt{R}_s$ \eqref{Eq:sIndices} returning the $(i_s,j_s)$-entry for fixed $s\in I$.
This is equivalent to establishing
\begin{equation} \label{Eq:pfqPsp5}
\begin{aligned}
  \br{(\Phi_s)_{i_sj_s},b_{kl}}_{\phi,\Sp}
=& \frac12 \sum_{t\in I}\sum_{u_t,v_t} \X(\Phi)^\ast\big((F_{u_t v_t}^{(1),L}+F_{u_t v_t}^{(1),R})(g_{i_sj_s})\big)\, (\check{F}_{u_t v_t}^{(1)})_M(b_{kl}) \\
&\frac12 \sum_{t\in I}\sum_{u_t \leq v_t} \X(\Phi)^\ast\big((F_{u_t v_t}^{(2),L}+F_{u_t v_t}^{(2),R})(g_{i_sj_s})\big)\,
(\check{F}_{u_t v_t}^{(2)})_M(b_{kl}) \\
&\frac12 \sum_{t\in I}\sum_{u_t \leq v_t} \X(\Phi)^\ast\big((F_{u_t v_t}^{(3),L}+F_{u_t v_t}^{(3),R})(g_{i_sj_s})\big)\,
(\check{F}_{u_t v_t}^{(3)})_M(b_{kl})\,,
\end{aligned}
\end{equation}
after using the dual bases from \ref{App:CartSP} of
$\spg(\alpha_t)\hookrightarrow \spg_\alpha$ for $t\in I$, and where we sum over
\begin{equation} \label{CondSp-t}
 \alpha_1+\ldots+\alpha_{t-1}+1 \leq u_t,v_t \leq \alpha_1+\ldots+\alpha_{t-1}+\frac12\alpha_{t}\,.
\end{equation}

For the LHS of \eqref{Eq:pfqPsp5}, \eqref{Phim} and
\eqref{Phim-tw} together with \eqref{Eq:BrPsymp} yield for any $b\in A$ and indices $k,l$,
\begin{equation} \label{Eq:BrqPsp6}
\begin{aligned}
&\br{(\Phi_s)_{i_sj_s},b_{kl}}_{\phi,\Sp}=\frac12 \dgal{\Phi_s,b}_{kj_s,i_sl}
+ \frac12 \sgn_\alpha(i_s,j_s) \dgal{\phi(\Phi_s),b}_{k\tau(i_s),\tau(j_s)l} \\
=& \frac14 (be_s\otimes \Phi_s-e_s \otimes \Phi_s b +  b \Phi_s \otimes e_s-\Phi_s \otimes e_s b)_{kj_s,i_sl} \\
&+\frac14\sgn_\alpha(i_s,j_s)  \Big(\phi(\Phi_s)\otimes e_s b - b \phi(\Phi_s) \otimes e_s
 +  e_s \otimes \phi(\Phi_s) b - b e_s\otimes  \phi(\Phi_s)\Big)_{k\tau(i_s),\tau(j_s)l} \,.
\end{aligned}
\end{equation}
For the RHS,
we need the following equalities on $\Sp_\alpha$ which are obtained from \eqref{Eq:LRinvVF} similarly to the $\Gl_\alpha$- and $\Orm_\alpha$-cases using the basis \eqref{Eq:AppSP1} (where $n=\alpha_t$):
\begin{align*}
F_{u_tv_t}^{(1),L}(g_{i_sj_s})=g_{i_s,u_t}\delta_{v_t,j_s} - g_{i_s,\alpha_t+v_t}\delta_{\alpha_t+u_t,j_s}, \quad
&F_{u_tv_t}^{(1),R}(g_{i_sj_s})=\delta_{i_su_t} g_{v_tj_s} - \delta_{i_s,\alpha_t+v_t}g_{\alpha_t+u_t,j_s};  \\
F_{u_tv_t}^{(2),L}(g_{i_sj_s})=g_{i_su_t}\delta_{\alpha_t+v_t,j_s} + g_{i_sv_t}\delta_{\alpha_t+u_t,j_s}, \quad
&F_{u_tv_t}^{(2),R}(g_{i_sj_s})=\delta_{i_su_t} g_{\alpha_t+v_t,j_s} + \delta_{i_sv_t}g_{\alpha_t+u_t,j_s};  \\
F_{u_tv_t}^{(3),L}(g_{i_sj_s})=g_{i_s,\alpha_t+u_t}\delta_{v_tj_s} + g_{i_s,\alpha_t+v_t}\delta_{u_tj_s}, \quad
&F_{u_tv_t}^{(3),R}(g_{i_sj_s})=\delta_{i_s,\alpha_t+u_t} g_{v_tj_s} + \delta_{i_s,\alpha_t+v_t}g_{u_tj_s}.
\end{align*}
This entails,
\small
\begin{align*}
\X(\Phi)^\ast\big((F_{u_t v_t}^{(1),L}+F_{u_t v_t}^{(1),R})(g_{i_sj_s})\big)
&=\delta_{st} \left( (\Phi_s)_{i_s,u_t}\delta_{v_t,j_s} - (\Phi_s)_{i_s,\alpha_t+v_t}\delta_{\alpha_t+u_t,j_s}
+\delta_{i_su_t} (\Phi_s)_{v_tj_s} - \delta_{i_s,\alpha_t+v_t} (\Phi_s)_{\alpha_t+u_t,j_s}\right);  \\
\X(\Phi)^\ast\big((F_{u_t v_t}^{(2),L}+F_{u_t v_t}^{(2),R})(g_{i_sj_s})\big)
&=\delta_{st} \left( (\Phi_s)_{i_su_t}\delta_{\alpha_t+v_t,j_s} + (\Phi_s)_{i_sv_t}\delta_{\alpha_t+u_t,j_s}
+ \delta_{i_su_t} (\Phi_s)_{\alpha_t+v_t,j_s} + \delta_{i_sv_t}(\Phi_s)_{\alpha_t+u_t,j_s} \right);  \\
\X(\Phi)^\ast\big((F_{u_t v_t}^{(3),L}+F_{u_t v_t}^{(3),R})(g_{i_sj_s})\big)
&=\delta_{st} \left( (\Phi_s)_{i_s,\alpha_t+u_t}\delta_{v_tj_s} + (\Phi_s)_{i_s,\alpha_t+v_t}\delta_{u_tj_s}
+ \delta_{i_s,\alpha_t+u_t} (\Phi_s)_{v_tj_s} + \delta_{i_s,\alpha_t+v_t} (\Phi_s)_{u_tj_s}\right).
\end{align*}
\normalsize
In those expressions, we can always write $\alpha_t+u_t=\tau(u_t)$ and $\alpha_t+v_t=\tau(v_t)$ due to \eqref{Eq:TauSp} and the condition \eqref{CondSp-t}.
We also make use of \eqref{InfAct-sp} and \eqref{Eq:AppSP2} to deduce
\begin{align*}
 (\check{F}_{u_tv_t}^{(1)})_{M}(b_{kl}) &=
 \frac12(F_{v_tu_t}^{(1)})_{M}(b_{kl}) =\frac12[\X(b),E_{v_tu_t}]_{kl}-\frac12[\X(b),E_{\alpha_t+u_t,\alpha_t+v_t}]_{kl}\,, \\
 (\check{F}_{u_tv_t}^{(2)})_{M}(b_{kl}) &\stackrel{u<v}{=}
 \frac12(F_{u_tv_t}^{(3)})_{M}(b_{kl}) =\frac12[\X(b),E_{\alpha_t+u_t,v_t}]_{kl}+\frac12[\X(b),E_{\alpha_t+v_t,u_t}]_{kl}\,, \\
 (\check{F}_{u_tu_t}^{(2)})_{M}(b_{kl}) &=
 \frac14(F_{u_tu_t}^{(3)})_{M}(b_{kl}) =\frac12[\X(b),E_{\alpha_t+u_t,u_t}]_{kl}\,, \\
 (\check{F}_{u_tv_t}^{(3)})_{M}(b_{kl}) &\stackrel{u<v}{=}
 \frac12(F_{u_tv_t}^{(2)})_{M}(b_{kl}) =\frac12[\X(b),E_{u_t,\alpha_t+v_t}]_{kl}+\frac12[\X(b),E_{v_t,\alpha_t+u_t}]_{kl}\,, \\
 (\check{F}_{u_tu_t}^{(3)})_{M}(b_{kl}) &=
 \frac14(F_{u_tu_t}^{(2)})_{M}(b_{kl}) =\frac12[\X(b),E_{u_t,\alpha_t+u_t}]_{kl}\,.
\end{align*}
We are in position to compute the RHS of \eqref{Eq:pfqPsp5}. Its first term is given by
\begin{align*}
&\eqref{Eq:pfqPsp5}_{RHS}^{(1)}=
\frac12 \sum_{t\in I}\sum_{u_t,v_t} \X(\Phi)^\ast\big((F_{u_t v_t}^{(1),L}+F_{u_t v_t}^{(1),R})(g_{i_sj_s})\big)\, (\check{F}_{u_t v_t}^{(1)})_M(b_{kl}) \\
=\frac14& \sum_{u_s,v_s} \big( (be_s)_{kv_s} (e_s)_{u_sl} - (e_s)_{kv_s} (e_sb)_{u_sl} \big)
\Big[(\Phi_s)_{i_su_s}(e_s)_{v_sj_s}+(e_s)_{i_su_s}(\Phi_s)_{v_sj_s} \\
&\qquad \qquad - \sgn_\alpha(i_s,j_s)
\big(\, \phi(\Phi_s)_{v_s,\tau(i_s)}(e_s)_{\tau(j_s),u_s}+(e_s)_{v_s,\tau(i_s)}\, \phi(\Phi_s)_{\tau(j_s),u_s} \,\big)\Big]\\
+\frac14& \sum_{u_s,v_s} \big( (be_s)_{k,\alpha_s+u_s} (e_s)_{\alpha_s+v_s,l} - (e_s)_{k,\alpha_s+u_s} (e_sb)_{\alpha_s+v_s,l} \big)
\Big[(\Phi_s)_{i_s,\alpha_s+v_s}(e_s)_{\alpha_s+u_s,j_s}+(e_s)_{i_s,\alpha_s+v_s}(\Phi_s)_{\alpha_s+u_s,j_s} \\
&\qquad \qquad - \sgn_\alpha(i_s,j_s)
\big(\, \phi(\Phi_s)_{\alpha_s+u_s,\tau(i_s)}(e_s)_{\tau(j_s),\alpha_s+v_s}+(e_s)_{\alpha_s+u_s,\tau(i_s)}\, \phi(\Phi_s)_{\tau(j_s),\alpha_s+v_s} \,\big)\Big]
\end{align*}
where $u_s,v_s$ satisfy \eqref{CondSp-t} with $t=s$, and we used the symplectic notations from \ref{sss:Not-symp} and the defining relation \eqref{Eq:DefRelSp}.
If we put
\begin{align}
 \cF^{(i_s,j_s)}_{(\alpha,\beta)}&=(\Phi_s)_{i_s\alpha}(e_s)_{\beta j_s}+(e_s)_{i_s\alpha}(\Phi_s)_{\beta j_s} \nonumber \\
&\qquad  - \sgn_\alpha(i_s,j_s)
\big(\, \phi(\Phi_s)_{\beta,\tau(i_s)}\,(e_s)_{\tau(j_s),\alpha}+(e_s)_{\beta,\tau(i_s)}\, \phi(\Phi_s)_{\tau(j_s),\alpha}\,\big) \label{Eq:cFab}\\
\cB^{(k,l)}_{(\alpha,\beta)} &= (be_s)_{k \alpha} (e_s)_{\beta l} - (e_s)_{k \alpha} (e_sb)_{\beta l}\,, \label{Eq:cBab}
\end{align}
we can simply write
\begin{align*}
&\eqref{Eq:pfqPsp5}_{RHS}^{(1)}
=\frac14 \sum_{u_s,v_s} (\cB^{(k,l)}_{(v_s,u_s)}   \cF^{(i_s,j_s)}_{(u_s,v_s)}
+ \cB^{(k,l)}_{(\alpha_s+u_s,\alpha_s+v_s)}   \cF^{(i_s,j_s)}_{(\alpha_s+v_s,\alpha_s+u_s)} )\,.
\end{align*}
Similarly, we can compute for the second and third terms,
\begin{align*}
&\eqref{Eq:pfqPsp5}_{RHS}^{(2)}
=\frac14 \sum_{u_s<v_s} \Big(\cB^{(k,l)}_{(\alpha_s+v_s,u_s)}   \cF^{(i_s,j_s)}_{(u_s,\alpha_s+v_s)}
+ \cB^{(k,l)}_{(\alpha_s+u_s,v_s)}   \cF^{(i_s,j_s)}_{(v_s,\alpha_s+u_s)} \Big)
+ \frac14 \sum_{u_s} \cB^{(k,l)}_{(\alpha_s+u_s,u_s)}   \cF^{(i_s,j_s)}_{(u_s,\alpha_s+u_s)} ,\\
&\eqref{Eq:pfqPsp5}_{RHS}^{(3)}
= \frac14 \sum_{u_s<v_s} \Big(\cB^{(k,l)}_{(v_s,\alpha_s+u_s)}   \cF^{(i_s,j_s)}_{(\alpha_s+u_s,v_s)}
+ \cB^{(k,l)}_{(u_s,\alpha_s+v_s)}   \cF^{(i_s,j_s)}_{(\alpha_s+v_s,u_s)} \Big)
+ \frac14 \sum_{u_s} \cB^{(k,l)}_{(u_s,\alpha_s+u_s)}   \cF^{(i_s,j_s)}_{(\alpha_s+u_s,u_s)} .
\end{align*}
Gathering the different contributions from the RHS of \eqref{Eq:pfqPsp5}, one obtains
\begin{equation*}
 \eqref{Eq:pfqPsp5}_{RHS}=
 \frac14 \sum_{\tilde{u}_s,\tilde{v}_s} \cB^{(k,l)}_{(\tilde{v}_s,\tilde{u}_s)}   \cF^{(i_s,j_s)}_{(\tilde{u}_s,\tilde{v}_s)}\,,
\end{equation*}
where $\tilde{u}_s,\tilde{v}_s\in\mathtt{R}_s$ \eqref{Eq:sIndices}
(this is \emph{not} the range of $u_s,v_s$ as in \eqref{CondSp-t}).
Using the expressions \eqref{Eq:cFab} and \eqref{Eq:cBab}, one easily finds the same result for the LHS \eqref{Eq:BrqPsp6} after a routine calculation.
\end{proof}

\subsection{Applications}

\subsubsection{Fundamental groups}

Recall Example \ref{Ex:grGen} where we explained that the double quasi-Poisson bracket of \cite{MT14} on
$L_{g,r}\simeq\CC\pi_1(\Sigma,\ast)$ is compatible with the anti-involution $\phi$ induced by inverting generators.
It was remarked by Massuyeau and Turaev \cite[Sect.~8 \& App.~B]{MT14} that the quasi-Poisson bracket $\br{-,-}$ on
$\Rep(L_{g,r},N)$ induced by their double quasi-Poisson bracket through Van den Bergh's theory \cite{VdB1} is, up to scaling, the quasi-Poisson structure given in \cite{AKSM}.
Furthermore, its restriction to $\Rep(L_{g,r},N)/\!\!/\Gl_N$ coincides (up to scaling) with the Goldman bracket \cite{Gol86}.

Given $\alpha,\beta \in L_{g,r}$, one reads from \eqref{Eq:BrPO-tr} and \eqref{Eq:BrPSymp-tr} (which also hold in the quasi-Poisson case):
\begin{equation}
\begin{aligned} \label{Eq:Goldtr}
  \br{\tr(\alpha),\tr(\beta)}_{\phi,G}&=\frac12 \br{\tr(\alpha),\tr(\beta)}+ \frac12 \br{\tr(\alpha^{-1}),\tr(\beta)}  \\
  &=\frac12 \br{\tr(\alpha),\tr(\beta)}+ \frac12 \br{\tr(\alpha),\tr(\beta^{-1})},
\end{aligned}
\end{equation}
where $G=\Orm_N$ or $\Sp_N$.
Using \eqref{Eq:Goldtr} and by a direct inspection of the main Theorem of \cite{Gol86},
the Poisson bracket that we induce on
$\Rep^{\phi,\tau}(L_{g,r},N)/\!\!/G$ coincides with the one of Goldman when restricted to trace functions.
Since trace functions generate the ring of invariant functions for $G=\Sp_N$ or when $N$ is odd for $G=\Orm_{N}$ (see e.g. \cite{LS}),
our Poisson bracket $\br{-,-}_{\phi,G}$ coincides with the Goldman bracket in types $\typB,\typC$.
In type $\typD$ where $G=\Orm_{N}$ with even $N$, invariant functions involving Pfaffians also exist.

\begin{rem}
The formulas induced by Massuyeau-Turaev agree with those of Goldman \cite{Gol86}, see \cite[Sect.~12]{MT18}.
Thus, our comparison with \cite{MT18} made in Section~\ref{Sec:MT} gives an alternative proof of the previous claims.
Formulas for the (real) quasi-Poisson bracket can also be found in the works of Li-Bland and \v{S}evera \cite{LBS}, and Nie \cite{Nie}.
\end{rem}

\begin{rem}
This result can be adapted to the weighted surfaces of \cite[Sect.~10]{MT14}
obtained by quotienting $L_{g,r}$ by the ideal $\langle z_j^{n_j}\mid j\in J\rangle$ for some $J\subset \{1,\ldots,r\}$ and weights $(n_j)\in \Z_{>1}^J$.
Indeed, we can consider the $g=0,r=1$ case of Example \ref{Ex:g0r1} on $\CC[z^{\pm 1}]/\langle z^n-1\rangle$ for any $n>1$, and then obtain Example \ref{Ex:grGen} with weights by performing fusion.
\end{rem}

\subsubsection{Quivers} \label{ss:qPquivers}

We start by recalling the double quasi-Poisson bracket for quivers due to Van den Bergh \cite{VdB1}, with the formulation of \cite{F2}.
It will be used to introduce a multiplicative analogue of \ref{sss:P-exmp} under some extra invertibility conditions, amounting to get a groupoid version of Example \ref{Ex:grGen}.

Let $\Upsilon$ be a quiver.
Define $\epsilon:\overline{\Upsilon}\to \{\pm1\}$ as the map which takes value $+1$ on arrows originally in $\Upsilon$, and $-1$ on arrows in $\overline{\Upsilon}\setminus \Upsilon$.
For each $a\in \Upsilon$, choose $\gamma_a \in \CC$ and set $\gamma_{a^\ast}=\gamma_a$.
We encode these constants in $\gamma=(\gamma_a)\in \CC^{\Upsilon}$.
Finally, we define the localised path algebra
\begin{equation}
 \tA_{\Upsilon,\gamma} := (\CC \overline{\Upsilon})_{S_\gamma}, \quad
 S_\gamma= \{1+(\gamma_a-1)e_{t(a)}+a a^\ast \,|\, a \in \overline{\Upsilon}\}\,.
\end{equation}
This amounts to adding a generator $(\gamma_a e_{t(a)}+a a^\ast)^{-1}\in e_{t(a)} \CC \overline{\Upsilon}e_{t(a)}$ for each $a \in \overline{\Upsilon}$ and imposing that it is an inverse to $\gamma_a e_{t(a)}+a a^\ast$. If $\gamma_a=0$, then $a^{-1}:=a^\ast(a a^\ast)^{-1}$ is an inverse to $a$ (i.e. $a a^{-1}=e_{t(a)}$ and $a^{-1}a=e_{h(a)}$),
and $(a^\ast)^{-1}:=a(a^\ast a)^{-1}$ is an inverse to $a^\ast$.

For each vertex $s\in I$, consider a total ordering $<_s$ on the set $T_s=\{a\in \overline{\Upsilon}\mid t(a)=s\}$. Write $o_s(-,-)$ for the ordering function at vertex $s$ : on arrows $a,b$ we have $o_s(a,b)=+1$ if $a<_s b$, $o_s(a,b)=-1$ if $b<_s a$, while it is zero otherwise, i.e. if $a=b\in T_s$, if $a \notin T_s$ or if $b \notin T_s$.

\begin{thm} \label{ThmVdB}
   The algebra $\tA_{\Upsilon,\gamma}$ has a double quasi-Poisson bracket defined by
  \begin{subequations}
        \begin{align}
 \dgal{a,a}\,=\,&\frac{1}{2}o_{t(a)}(a,a^*)\left( a^2\otimes e_{t(a)}- e_{h(a)}\otimes a^2 \right) \qquad (a\in\overline{\Upsilon})\,, \label{loopG}\\
 \dgal{a,a^*}\,=\,&\gamma_a e_{h(a)}\otimes e_{t(a)}
 +\frac{1}{2} a^*a\otimes e_{t(a)} +\frac{1}{2} e_{h(a)}\otimes aa^* \nonumber\\
 & +\frac{1}{2}o_{t(a)}(a,a^*)\, (a^*\otimes a-a\otimes a^*)\qquad (a\in \Upsilon)\,, \label{aastG}
        \end{align}
  \end{subequations}
 and for $b,c\in\overline{\Upsilon}$ such that $ c\ne b,b^*$
 \begin{equation}
  \begin{aligned}
 \dgal{b,c}\,=\,&-\frac{1}{2}o_{t(b)}(b,c)\,(b\otimes c)-\frac{1}{2}o_{h(b)}(b^*,c^*)\,(c\otimes b)
\\ \label{a<bG}
 &+\frac{1}{2}o_{t(b)}(b,c^*)\, cb\otimes e_{t(b)} + \frac{1}{2}o_{h(b)}(b^*,c)\,e_{h(b)}\otimes bc   \,.
  \end{aligned}
 \end{equation}
Furthermore, $\tA_{\Upsilon,\gamma}$ admits the moment map
\begin{equation} \label{EqPhiVdB}
  \Phi=\sum_{s\in I} \Phi_s\,, \quad \Phi_s=\prod_{a\in T_s}^{\longrightarrow}(\gamma_a e_s+ a a^\ast)^{\epsilon(a)}\,.
\end{equation}
\end{thm}

A crucial observation is that the proof of Theorem \ref{ThmVdB} follows by fusion (cf. \ref{ss:Fus}) from the statement for the $1$-arrow quiver.
Since bracket-compatibility is preserved by fusion due to Proposition \ref{Pr:CompFus}, we need to understand the case of a 1-arrow quiver with $\gamma=0$ to adapt Example \ref{Ex:g1r0}.

\begin{rem}
Let $\Upsilon$ is a 1-loop quiver with arrow $x$ and denote its double by $y=x^\ast$.
For $\gamma_x=0$ and $x<y$, Theorem \ref{ThmVdB} gives the structure of Example \ref{Ex:g1r0}.
\end{rem}

\begin{exmp} \label{Ex:Grpd1}
Let $\Upsilon$ be the quiver with two vertices $\{1,2\}$ and one arrow
$1\stackrel{a}{\longrightarrow}2$.
Then $\tA_{\Upsilon,0}=\CC \overline{\Upsilon}_{S_0}$ for the set
$S_0=\{ aa^\ast+e_2,\,
a^\ast a+e_1 \}$
is equipped with the double quasi-Poisson bracket satisfying
\begin{equation} \label{qP-1arrow}
 \dgal{a,a^\ast} = \frac12 (a^\ast a \otimes e_1 + e_2 \otimes aa^\ast), \quad
 \dgal{a,a} = 0, \quad \dgal{a^\ast,a^\ast} = 0.
\end{equation}
By construction, one has in $\tA_{\Upsilon,0}$ the inverses $a^{-1}$ and $(a^\ast)^{-1}$ satisfying
\[
 a a^{-1}=e_{1}, \,\, a^{-1}a=e_{2}, \,\,
 a^\ast (a^\ast)^{-1}=e_{2}, \,\, (a^\ast)^{-1}a^\ast=e_{1}.
\]
The moment map reads
\begin{equation}
 \Phi=\Phi_1+\Phi_2, \quad \Phi_1=aa^\ast, \,\, \Phi_2=a^{-1}(a^\ast)^{-1}\,.
\end{equation}
If one introduces the anti-involution $\phi:\tA_{\Upsilon,0}\to \tA_{\Upsilon,0}$
such that $\phi(a)=a^{-1}$ and $\phi(a^\ast)=(a^\ast)^{-1}$,
it is easy to check that bracket-compatibility \eqref{Eq:BrComp} is satisfied.
\end{exmp}

\begin{prop} \label{Pr:Grpd}
The double quasi-Poisson structure on the algebra $\tA_{\Upsilon,0}$ given in Theorem \ref{ThmVdB} with $\gamma=0\in \CC^\Upsilon$ is $\phi$-adapted for the anti-involution satisfying $\phi(a)=a^{-1}$ for each $a\in \overline{\Upsilon}$.
\end{prop}
\begin{proof}
 This follows from the 1-arrow case in Example \ref{Ex:Grpd1} and Proposition \ref{Pr:CompFus}.
\end{proof}

\begin{rem}
 We are unable to provide an analogue of Proposition \ref{Pr:Grpd} when $\gamma\neq 0$.
A consequence of this drawback is that the representation space $\Rep(\tA_{\Upsilon,0},\alpha)$ (hence any twisted version $\Rep^{\phi,\tau}(\tA_{\Upsilon,0},\alpha)$) is empty except when all $\alpha_s$ are equal on connected components of $\Upsilon$ because we require the invertibility of each $\X(a)$ with $a\in \overline{\Upsilon}$.
\end{rem}

\begin{thm} \label{Thm:qP-Double}
 Fix $N\in \Z_{>0}$ and consider the one-arrow quiver $\Upsilon=1\longrightarrow2$.
Then the double quasi-Hamiltonian algebra on $\tA_{\Upsilon,0}$ given in Theorem \ref{ThmVdB}
equipped with the anti-involution $\phi$ given by inverting arrows
is such that $\Rep^{\phi,\tau}(\tA_{\Upsilon,0},(N,N))$ is isomorphic to the quasi-Poisson double
$D(\Orm_N)$ (if $\tau:\xi \mapsto \xi^T$)
or  $D(\Sp_{N})$ (if $\tau:\xi \mapsto \Omega\xi^T \Omega^T$ for $\Omega=\diag(\Omega_N,\Omega_N)$ with $N$ even).
\end{thm}
\begin{proof}
 We show the result in the symplectic case, and leave the orthogonal case as an easy exercise.

First of all, $\Rep(\tA_{\Upsilon,0},(N,N))$ is parametrised by the couple of $N\times N$ invertible matrices $(A,B)$ that correspond to the non-trivial blocks in the matrices representing the arrow $a\in \Upsilon$ and its double $2\stackrel{a^\ast}{\longrightarrow}1$, respectively.
Passing to the twisted representation space $\Rep^{\phi,\tau}(\tA_{\Upsilon,0},(N,N))$,
the defining relation \eqref{TwIdeal} entails that the matrices $(A,B)$ satisfy
\[
 A^{-1}=\Omega_N A^T \Omega_N^T, \quad B^{-1}=\Omega_N B^T \Omega_N^T,
\]
i.e. $\Rep^{\phi,\tau}(\tA_{\Upsilon,0},(N,N))= \Sp_N\times \Sp_N$. The $(\Sp_N\times \Sp_N)$-action induced by \eqref{Infgrp} reads
\[
 (g,h)\cdot (A,B) = (gAh^{-1},hBg^{-1}), \qquad g,h \in \Sp_N\,.
\]
The quasi-Poisson bracket, denoted $\br{-,-}$ for short, is determined by the following equalities:
 \begin{align*}
 \br{A_{ij},A_{kl}}=& \, 0, \qquad  \br{B_{ij},B_{kl}}= 0\\
 \br{A_{ij},B_{kl}}=&\frac14 (BA)_{kj} \delta_{il} + \frac14 \delta_{kj} (AB)_{il}
  -\frac14\sgn_N(i,l) B_{k,i+\frac{N}{2}}\, A_{l+\frac{N}{2},j} -\frac14\sgn_N(k,j) A_{i,k+\frac{N}{2}} \,B_{j+\frac{N}{2},l}.
 \end{align*}
(Indices are understood modulo $N$.) These are obtained by combining \eqref{Eq:BrPsymp} and \eqref{qP-1arrow}, e.g.
\begin{align*}
  &\br{A_{ij},B_{kl}}=\frac12\left(\dgal{a,a^{\ast}}_{kj;il}+\sgn_N(i,j)\,
  (-a^{-1}\ast\dgal{a,a^{\ast}}\ast a^{-1})_{k,i+\frac{N}{2} ;j+\frac{N}{2},l}\right)
  \Big|_{\X(a)\to A,\, \X(a^\ast)\to B}\\
  &= \frac14 (BA)_{kj} \delta_{il} + \frac14 \delta_{kj} (AB)_{il}
  -\frac14\sgn_N(i,j) B_{k,i+\frac{N}{2}} A^{-1}_{j+\frac{N}{2},l} -\frac14\sgn_N(i,j)
  A^{-1}_{k,i+\frac{N}{2}} B_{j+\frac{N}{2},l}
\end{align*}
and then concluding with $(A^{-1})_{uv}=-\sgn_N(u,v)\, A_{v+N/2,u+N/2}$.

This is just the quasi-Poisson structure from \cite[Ex.~5.3]{AKSM} after setting $(a_1,a_2):=(A,B)$. Indeed, the two moment maps directly coincide, and it is not hard to get the above quasi-Poisson bracket from the quasi-Poisson bivector
\begin{align*}
 P=\frac12\sum_\alpha (F_\alpha,0)^L \wedge (0,\check{F}_\alpha)^R + \frac12\sum_\alpha (F_\alpha,0)^R \wedge (0,\check{F}_\alpha)^L
\end{align*}
where $(F_\alpha)$ and $(\check{F}_\alpha)$ denote the dual bases of $\spg_N$ from \ref{App:CartSP} (with $n=N/2$) which we embed in
$\spg_N\oplus \spg_N$ before acting as the left- or right-invariant vector field on the corresponding component.
\end{proof}

The previous result has the following implication for an arbitrary quiver $\Upsilon$.
Fix $N\in \Z_{>0}$ and let $\alpha=(N,\ldots,N)\in \Z^{I}$.
Considering $\tA_{\Upsilon,0}$ as in Theorem \ref{ThmVdB}
equipped with the anti-involution $\phi$ given by inverting arrows,
$\Rep^{\phi,\tau}(\tA_{\Upsilon,0},\alpha)$ is the quasi-Poisson space obtained by fusion from $|\Upsilon|$ copies of the quasi-Poisson double
$D(G)$, where $G=\Orm_N$ or $G=\Sp_{N}$ depending on the chosen type.
Indeed, $\tA_{\Upsilon,0}$ is obtained by fusion of $|\Upsilon|$ copies of the $1$-arrow case, and each one of the latter corresponds to a quasi-Poisson double by Theorem \ref{Thm:qP-Double}.
Following \cite{Maiz26} and assuming that $\Upsilon$ has no isolated vertex, let
\[
 I_{\intern}=\{s\in I \mid \,|h^{-1}(s)|\neq 1 \text{ and }|t^{-1}(s)|\neq 1\,\}\,.
\]
We can consider the action of $G_{\intern}:=\prod_{s\in I_{\intern}} G\hookrightarrow \prod_{s\in I}G$ on $\Rep^{\phi,\tau}(\tA_{\Upsilon,0},\alpha)$, and the corresponding $G_{\intern}$-valued factor $\X(\Phi)_{\intern}:=\prod_{s\in I_{\intern}} \X(\Phi_s)$ of the moment map.
Then, we set
\begin{equation}
 \mathcal{N}_G(\Upsilon) := (\X(\Phi)_{\intern})^{-1}(1)/\!\!/ G_{\intern}\,.
\end{equation}
The reduced spaces constructed in this way are complex algebraic versions of the quasi-Hamiltonian spaces associated with a quiver $\Upsilon$ and a compact Lie group $G$ in \cite{Maiz26}.
Those real spaces were introduced as deformations of the Lax-Kirchhoff moduli spaces \cite{MM25}, which are used to build new TQFTs.

\begin{rem}
 The algebra $\tA_{\Upsilon,0}$ equipped with the double bracket from Theorem \ref{ThmVdB} and the anti-involution
 $\phi:a\mapsto a^{-1}$ should be viewed as a groupoid-version of Example \ref{Ex:grGen}.
Thus, one should be able to adapt \cite{MT14} for fundamental groupoids and recover this double quasi-Hamiltonian algebra.
Making these claims precise and comparing with the approach of \cite{FR} constitutes an important open problem.
\end{rem}


\section{Mixing types \texorpdfstring{$\typB,\typC,\typD$}{B,C,D}}  \label{Sec:Mix}

We modify all previous constructions in order to mix actions of the orthogonal and symplectic groups.
This requires to adapt the definitions of twisted representation space and of involutive algebra.

\subsection{Constructions from a choice of types}

Let $A$ be an algebra over $A_0=\oplus_{s\in I} \CC e_s$.
\begin{defn} \label{def:AlgType}
A \emph{choice of types} for $A$ consists in assigning to each $s\in I$ a label, called \emph{type}, given
either by $\Orm$ (type $\typB/\typD$) or by $\Sp$ (type $\typC$).
\end{defn}
From now on, we assume that a choice of types is fixed on $A$. It is encoded in a mapping
\begin{equation} \label{Eq:fs}
 \fs:I\to \{-1,+1\}, \quad
 \fs(s)=\left\{
\begin{array}{ll}
 +1&\text{if }s \text{ is of type }\Orm , \\
 -1&\text{if }s\text{ is of type }\Sp .
\end{array}
\right.
\end{equation}

\begin{defn}
A \emph{typed anti-involution} on $A$ is an anti-morphism $\varpi:A\to A$ such that $\varpi(e_s)=e_s$ for each $s\in I$,
and if $a\in e_s A e_t$ with $s,t\in I$, then
$\varpi^2(a)=\fs(s)\fs(t)\,a$.

\noindent If $A$ is endowed with a double bracket $\dgal{-,-}$, we say that $\dgal{-,-}$ is $\varpi$\emph{-typed} when
\begin{equation} \label{Eq:BrTyped}
 \varpi^{\otimes 2}(\dgal{a,b})=\dgal{\varpi(a),\varpi(b)}^\circ \qquad \forall a,b\in A.
\end{equation}
\end{defn}
\begin{rem}
 When the choice of types is such that we assign the same type to each $s\in I$, we are in the situation where $(A,\varpi)$ is an involutive algebra and $\dgal{-,-}$ is $\varpi$-adapted as in the previous sections.
\end{rem}
\begin{rem} \label{Rem:TypeInd}
One has $\varpi^4(a)=a$ for any $a\in A$.
If $a\in e_s Ae_t$, then $\varpi^{-1}(a)=\varpi^3(a)=\fs(s)\fs(t)\,\varpi(a)$.
When $a\in e_s A e_t$ and $b\in A$ is arbitrary,
\begin{equation} \label{Eq:BrTyped2}
 \varpi^{\otimes 2}(\dgal{\varpi(a),b})=\fs(s)\fs(t)\,\dgal{a,\varpi(b)}^\circ .
\end{equation}
\end{rem}

\begin{lem}
 Let $A$ be generated by $\{a_j \mid j\in J\}$ over $A_0$.
 If \eqref{Eq:BrTyped} holds when $a,b$ are two elements in $\{a_j \mid j\in J\}$, then \eqref{Eq:BrTyped} holds for any $a,b\in A$.
\end{lem}
\begin{proof}
 It suffices to reproduce \cite[Lem.~4.1]{OS}.
\end{proof}

Next, take $\alpha \in \N^I$ with the constraint that $\alpha_s$ is even if $s$ has type $\Sp$.
Such an $\alpha$ is a \emph{typed dimension vector}.
Put $N=\sum_{s\in I} \alpha_s$, and consider the decomposition \eqref{DecoCNrep} of $\CC^N$.
We let $G_s:=\Orm(\alpha_s)$ if $s$ has type $\Orm$, and $G_s:=\Sp(\alpha_s)$ if $s$ has type $\Sp$.
We also define the matrix $\Theta_s=\Id_{\alpha_s}$ if $s$ has type $\Orm$, and $\Theta_s:=\Omega_{\alpha_s}$ (cf. \eqref{Eq:omega}) if $s$ has type $\Sp$.
Then, we introduce an endomorphism $\theta$ on $\gl_N$ by
\begin{equation} \label{Eq:thetaType}
 \theta:\xi \mapsto \Theta \xi^T \Theta^T, \qquad \Theta= \diag(\Theta_1,\ldots,\Theta_{|I|})\,,
\end{equation}
where $\Theta_s$ depends on the type of $s$ accordingly.
If all types are $\Orm$ we are in the orthogonal case of \ref{ss:Rep};
if all types are $\Sp$ we are in the symplectic case of \ref{ss:Rep}.
If the types are mixed, note that $\theta$ is \emph{not} an involution and we only get $\theta^4=\id$ because
$\Theta^2=\diag(\fs(1) \Id_{\alpha_1},\ldots,\fs(|I|) \Id_{\alpha_{|I|}})$
with $\fs$ given in \eqref{Eq:fs}.

The twisted representation space (corresponding to our choice of types) is defined as
\begin{equation}
 \Rep^{\varpi,\theta}(A,\alpha) = \{\rho \in \Rep(A,\alpha) \mid \rho \circ \varpi = \theta \circ \rho\}\,.
\end{equation}
In terms of the coordinate ring, one has
$\CC[\Rep^{\varpi,\theta}(A,\alpha)]=\CC[\Rep(A,\alpha)]/I^{\varpi,\theta}$ for the ideal
\begin{equation} \label{TypeIdeal}
I^{\varpi,\theta} := \langle(\varpi(a)|E_{ij})-(a|\theta(E_{ij}))  \,|\,
 a\in A,\, 1\leq i,j\leq N\rangle \,.
\end{equation}
This twisted representation space  is naturally equipped with an action of $G:=\prod_s G_s$ and its Lie algebra
$\g=\prod_s \operatorname{Lie}(G_s)$ induced by \eqref{Infgrp} and \eqref{InfAct}, respectively.
This is easily seen by realizing $g\in G$  as a block-diagonal matrix in $\Gl_N(\CC)$, with blocks of respective sizes $\alpha_1,\ldots,\alpha_{|I|}$, satisfying $g \Theta g^T \Theta^T=\Id_N$.

\medskip

Next, define $\iota,\hat{\alpha}$ as in \eqref{Eq:IotaAlph}.
Similarly to \eqref{Eq:sgnAlp}, introduce $\sgn_{\alpha}: \{1,\ldots,N\} \to \{-1,+1\}$ by
\begin{equation}
 \sgn_{\alpha}(k):= \left\{
 \begin{array}{ll}
 +1 &\iota(k) \text{ has type } \Orm; \\
\sgn_{\hat{\alpha}(k)}\Big(k- \sum_{1\leq t< \iota(k)}\alpha_t\Big) &\iota(k) \text{ has type } \Sp;
 \end{array} \right. \label{Eq:sgnType}
\end{equation}
where we used \eqref{Eq:sgn} in type $\Sp$.
We can then consider $\theta$ at the level of indices as
\begin{equation} \label{Eq:ThetaSpType}
 \theta: \{1,\ldots,N\} \to \{1,\ldots,N\}, \qquad \theta(k)=\left\{
 \begin{array}{ll}
 k &\iota(k) \text{ has type } \Orm ;\\
k+\frac12 \sgn_{\alpha}(k)\,  \hat{\alpha}(k) &\iota(k) \text{ has type } \Sp;
 \end{array} \right.
\end{equation}
which satisfies $\theta^2=\id$ (although $\theta$ \eqref{Eq:thetaType} may \emph{not}). We also deduce $\sgn_\alpha(\theta(j))=\fs(\iota(j))\,\sgn_\alpha(j)$.

\begin{lem} \label{lem:thetaType}
 On elementary matrices, one has $\theta(E_{ij})=\fs(\iota(i))\sgn_\alpha(i)\, \fs(\iota(j))\sgn_\alpha(j)\, E_{\theta(j),\theta(i)}$.
\end{lem}
\begin{proof}
It suffices to calculate $\theta(E_{ij})=\sum_{p,q=1}^N (\Theta E_{ji} \Theta^T)_{pq} E_{pq}=\sum_{p,q=1}^N (\Theta)_{pj} (\Theta)_{qi} E_{pq}$.
The formula follows once we note that for any $1\leq p,k\leq N$, we can write
\begin{equation*}
 \Theta_{pk}=\fs(\iota(k)) \sgn_\alpha(k) \delta_{p,\theta(k)}\,.
\end{equation*}
Indeed, this last formula holds when $\iota(k)$ has type $\Orm$ because both sides evaluate to $\delta_{pk}$,
and when $\iota(k)$ has type $\Sp$ because both sides evaluate to $\sgn_\alpha(p)\delta_{p,\theta(k)}$
(cf. \eqref{Eq:fs}, \eqref{Eq:sgnType} and \eqref{Eq:ThetaSpType}).
\end{proof}

We are led to define the $2$-parameter typed sign function generalizing \eqref{Eq:sgnType}:
\begin{equation} \label{Eq:sgnTypeDble}
 \fsgn_\alpha:\{1,\ldots,N\}\times\{1,\ldots,N\}\to \{-1,+1\}, \quad
 \fsgn_\alpha(i,j)=\fs(\iota(i))\sgn_\alpha(i)\, \fs(\iota(j))\sgn_\alpha(j)\,,
 \end{equation}
 which satisfies ($1\leq i,j,k\leq N$)
\begin{equation} \label{Eq:sgnTypePr}
 \fsgn_\alpha(i,j)=\fsgn_\alpha(j,i), \quad \fsgn_\alpha(i,j)\fsgn_\alpha(j,k)=\fsgn_\alpha(i,k), \quad
 \fsgn_\alpha(i,\theta(j))=\fs(\iota(j))\,\fsgn_\alpha(i,j).
\end{equation}
Owing to Lemma \ref{lem:thetaType}, the defining ideal \eqref{TypeIdeal} of the twisted representation space is generated by
\begin{equation} \label{Eq:DefRelType2}
 \varpi(a)_{ij}-\fsgn_\alpha(i,j)\, a_{\theta(j),\theta(i)}, \qquad a\in A, \quad 1\leq i,j\leq N,
\end{equation}
or equivalently
\begin{equation} \label{Eq:DefRelType3}
 a_{ij}-\fsgn_\alpha(\theta(i),\theta(j))\, \varpi(a)_{\theta(j),\theta(i)}, \qquad a\in A, \quad 1\leq i,j\leq N.
\end{equation}

\begin{exmp}
Using \eqref{Eq:DefRelType2}, we recover \eqref{Eq:DefRelO} whenever $\iota(i),\iota(j)$ both have type $\Orm$.
Similarly, we recover \eqref{Eq:DefRelSp} whenever $\iota(i),\iota(j)$ both have type $\Sp$.
If $\iota(i)$ has type $\Orm$ and $\iota(j)$ type $\Sp$ (or vice-versa), one has
\begin{equation*}
 \varpi(a)_{ij} = - \sgn_\alpha(j) a_{\theta(j),i}, \qquad
 a_{ji}=-\sgn_\alpha(j)\, \varpi(a)_{i,\theta(j)},
\end{equation*}
which is consistent with $\varpi^2(a)_{ji}=-a_{ji}$ in that case (since $\varpi$ is a typed anti-involution).
\end{exmp}

\subsection{Inducing a biderivation for mixed types}

\begin{prop} \label{Pr:BrType}
 Let $A$ be an algebra over the semi-simple algebra $A_0=\oplus_{s\in I} \CC e_s$, and fix a choice of types for $A$.
Assume that $\varpi:A\to A$ is a typed anti-involution and $A$ is equipped with a double bracket  $\dgal{-,-}$ which is $\varpi$-typed.
Fix a typed dimension vector $\alpha\in \N^I$, $N=\sum_{s\in I}\alpha_s$,
with the corresponding involution $\theta$ on $\gl_N$ given by \eqref{Eq:thetaType}.
Then, the antisymmetric biderivation $\br{-,-}_{\varpi,\theta}$ on $\Rep^{\varpi,\theta}(A,\alpha)$ uniquely determined by
\begin{equation} \label{Eq:BrPtype}
\br{a_{ij},b_{kl}}_{\varpi,\theta}=\frac12 \dgal{a,b}_{kj,il}
+\frac12 \fsgn_\alpha(\theta(i),\theta(j))\, \dgal{\varpi(a),b}_{k\theta(i),\theta(j)l} \,, \quad a,b\in A,
\end{equation}
is well-defined. Moreover, the associated map \eqref{Eq:Jac} satisfies
\begin{equation}
 \begin{aligned}
\label{Eq:JacType}
\Jac_{\br{-,-}_{\varpi,\theta}}&(a_{ij},b_{kl},c_{uv})
=\,\,\frac14\dgal{a,b,c}_{uj,il,kv} -\frac14 \dgal{a,c,b}_{kj,iv,ul} \\
&+\frac14 \fsgn_\alpha(\theta(i),\theta(j)) \, (\dgal{\varpi(a),b,c}_{u\theta(i),\theta(j)l,kv}
-\dgal{\varpi(a),c,b}_{k\theta(i),\theta(j)v,ul} )\\
&+\frac14 \fsgn_\alpha(\theta(k),\theta(l)) \, (\dgal{a,\varpi(b),c}_{uj,i\theta(k),\theta(l)v}
- \dgal{a,c,\varpi(b)}_{\theta(l)j,iv,u\theta(k)} )\\
&+ \frac14 \fsgn_\alpha(\theta(u),\theta(v)) \, (\dgal{a,b,\varpi(c)}_{\theta(v)j,il,k\theta(u)}
- \dgal{a,\varpi(c),b}_{kj,i\theta(u),\theta(v)l} ) \,.
 \end{aligned}
\end{equation}
\end{prop}
\begin{rem} \label{Rem:TypeDegen}
 If all $s\in I$ have type $\Orm$, then \eqref{Eq:BrPtype} is \eqref{Eq:BrPO} and \eqref{Eq:JacType} is \eqref{Eq:pfqPO1} with $\phi=\varpi$.

\noindent  If all $s\in I$ have type $\Sp$, then \eqref{Eq:BrPtype} is \eqref{Eq:BrPsymp} and \eqref{Eq:JacType} is \eqref{Eq:pfqPsp1} with $\phi=\varpi$ and $\theta=\tau$ \eqref{Eq:TauSp}.
\end{rem}

\begin{proof}
We write $\br{-,-}:=\br{-,-}_{\varpi,\theta}$ throughout the proof.
This operation is antisymmetric on generators by \eqref{Eq:cycanti}, and it can be uniquely extended to an antisymmetric biderivation by the Leibniz rule to $\CC[\Rep^{\varpi,\theta}(A,\alpha)]$.
To be well-defined, we need to check that it is compatible with the relations
\begin{equation*}
 1_{ij}=\delta_{ij}, \quad (\lambda a)_{ij}=\lambda a_{ij}, \quad (a+c)_{ij}=a_{ij}+c_{ij}, \quad
 (ac)_{ij}=\sum_{1\leq u\leq N} a_{iu}c_{uj},
\end{equation*}
where $a,c\in A$, $\lambda\in \CC$, $1\leq i,j\leq N$,
together with \eqref{Eq:DefRelType2}.
For the first $3$ relations, this is obvious. For the fourth one, we compute thanks to the derivation rules \eqref{Eq:outder}-\eqref{Eq:inder}
\begin{align*}
& \br{(ac)_{ij},b_{kl}}=\frac12 \dgal{ac,b}_{kj,il}
+\frac12 \fsgn_\alpha(\theta(i),\theta(j))\, \dgal{\varpi(c)\varpi(a),b}_{k\theta(i),\theta(j)l} \\
&=\frac12 (a\ast\dgal{c,b}+\dgal{a,b}\ast c)_{kj,il}
+\frac12 \fsgn_\alpha(\theta(i),\theta(j))\, \left(\varpi(c)\ast\dgal{\varpi(a),b}+\dgal{\varpi(c),b}\ast\varpi(a)\right)_{k\theta(i),\theta(j)l}.
\end{align*}
Meanwhile, we obtain by making use of \eqref{Eq:DefRelType2} then \eqref{Eq:sgnTypePr},
\begin{align*}
& \sum_u \br{a_{iu}c_{uj},b_{kl}}
= \sum_u a_{iu}\br{c_{uj},b_{kl}}
+\sum_u \br{a_{iu},b_{kl}}  c_{uj} \\
=&\frac12 \sum_u (a_{iu} \dgal{c,b}_{kj,ul} + \fsgn_\alpha(\theta(i),\theta(u))\fsgn_\alpha(\theta(u),\theta(j))  \varpi(a)_{\theta(u)\theta(i)} \dgal{\varpi(c),b}_{k\theta(u),\theta(j)l} ) \\
&+\frac12\sum_u (\dgal{a,b}_{ku,il} c_{uj} + \fsgn_\alpha(\theta(u),\theta(j))\fsgn_\alpha(\theta(i),\theta(u)) \dgal{\varpi(a),b}_{k\theta(i),\theta(u)l}   \varpi(c)_{\theta(j)\theta(u)}) \\
=&\frac12 \sum_u a_{iu} \dgal{c,b}_{kj,ul}
+\frac12 \fsgn_\alpha(\theta(i),\theta(j)) \sum_u   \dgal{\varpi(c),b}_{k\theta(u),\theta(j)l}\varpi(a)_{\theta(u)\theta(i)} \\
&+\frac12 \sum_u \dgal{a,b}_{ku,il} c_{uj}
+\frac12 \fsgn_\alpha(\theta(i),\theta(j)) \sum_u  \varpi(c)_{\theta(j)\theta(u)} \dgal{\varpi(a),b}_{k\theta(i),\theta(u)l}\,.
\end{align*}
In particular, we find the desired equality
$\br{(ac)_{ij},b_{kl}} -  \sum_u \br{a_{iu}c_{uj},b_{kl}} =0$.

Finally, to check that the operation $\br{-,-} $ is compatible with \eqref{Eq:DefRelType2}, we compute
\begin{align*}
& \br{\varpi(a)_{ij} - \fsgn_\alpha(i,j) a_{\theta(j),\theta(i)},b_{kl}}  \\
=&\frac12 \dgal{\varpi(a),b}_{kj,il}
+\frac12 \fsgn_\alpha(\theta(i),\theta(j))\, \dgal{\varpi^2(a),b}_{k\theta(i),\theta(j)l} \\
&-\frac12 \fsgn_\alpha(i,j) \dgal{a,b}_{k\theta(i),\theta(j)l}
-\frac12  \fsgn_\alpha(i,j)\fsgn_\alpha(\theta^2(i),\theta^2(j))\, \dgal{\varpi(a),b}_{kj,il}
\end{align*}
Using \eqref{Eq:sgnTypePr}, the first and fourth term directly cancel out. As this bracket is non-trivially zero only when $\varpi(a)_{ij}\neq 0$, this let us assume that $a\in e_{\iota(j)}A e_{\iota(i)}$, and therefore by  Remark \ref{Rem:TypeInd}
$\varpi^2(a)=\fs(\iota(i))\fs(\iota(j))a$.
This last equality and \eqref{Eq:sgnTypePr} entail that the second and third terms cancel out.

\medskip

Next, we verify \eqref{Eq:JacType}.
Without loss of generality, we can assume that $a\in e_{\iota(i)}A e_{\iota(j)}$,
$b\in e_{\iota(k)}A e_{\iota(l)}$ and $c\in e_{\iota(u)}A e_{\iota(v)}$.
E.g. if $\{c\}\cap e_{\iota(u)}A e_{\iota(v)}=0$, one has that \eqref{Eq:JacType} trivially vanishes because the double Jacobiator \eqref{Eq:TripBr} is a derivation in its last argument for the outer bimodule structure on $A^{\otimes 3}$.
Thus, we are free to use \eqref{Eq:BrTyped2} with $a,b,c$ in any of the two arguments.

First, using Sweedler's notation, we compute using \eqref{Eq:BrPtype},
\begin{subequations}
 \begin{align}
&\br{a_{ij} , \br{ b_{kl},c_{uv} }}
=\frac12 \br{a_{ij} , \dgal{b,c}'_{ul}\dgal{b,c}''_{kv} }
+ \frac12 \fsgn_\alpha(\theta(k),\theta(l)) \, \br{a_{ij} , \dgal{\varpi(b),c}'_{u\theta(k)} \dgal{\varpi(b),c}''_{\theta(l)v}} \nonumber \\
=&\frac14 \dgal{a, \dgal{b,c}'}_{uj,il} \dgal{b,c}''_{kv}
+\frac14 \fsgn_\alpha(\theta(i),\theta(j)) \,\dgal{\varpi(a), \dgal{b,c}'}_{u\theta(i),\theta(j)l} \dgal{b,c}''_{kv} \label{JAC12}\\
&+ \frac14 \dgal{b,c}'_{ul} \dgal{a , \dgal{b,c}''}_{kj,iv}
+ \frac14 \fsgn_\alpha(\theta(i),\theta(j)) \dgal{b,c}'_{ul} \dgal{\varpi(a) , \dgal{b,c}''}_{k\theta(i),\theta(j)v}  \label{JAC34}\\
&+\frac14 \fsgn_\alpha(\theta(k),\theta(l)) \dgal{a, \dgal{\varpi(b),c}'}_{uj,i\theta(k)} \dgal{\varpi(b),c}''_{\theta(l)v}  \label{JAC5} \\
&+\frac14 \fsgn_\alpha(\theta(i),\theta(j))\fsgn_\alpha(\theta(k),\theta(l)) \,\dgal{\varpi(a), \dgal{\varpi(b),c}'}_{u\theta(i),\theta(j)\theta(k)} \dgal{\varpi(b),c}''_{\theta(l)v} \label{JAC6}\\
&+ \frac14 \fsgn_\alpha(\theta(k),\theta(l)) \dgal{\varpi(b),c}'_{u\theta(k)} \dgal{a , \dgal{\varpi(b),c}''}_{\theta(l)j,iv} \label{JAC7} \\
&+ \frac14 \fsgn_\alpha(\theta(i),\theta(j)) \fsgn(\theta(k),\theta(l)) \dgal{\varpi(b),c}'_{u\theta(k)} \dgal{\varpi(a) , \dgal{\varpi(b),c}''}_{\theta(l)\theta(i),\theta(j)v}\,.  \label{JAC8}
\end{align}
\end{subequations}
Recalling the notation $\dgal{a,d'\otimes d''}_L=\dgal{a,d'}\otimes d''$, we can write
\begin{equation*}
 \eqref{JAC12}=
\frac14 \, (\dgal{a, \dgal{b,c}}_L)_{uj,il,kv}
+ \frac14 \fsgn_\alpha(\theta(i),\theta(j)) \,(\dgal{\varpi(a), \dgal{b,c}}_L)_{u\theta(i),\theta(j)l,kv}\,.
\end{equation*}
Similarly, using the cyclic antisymmetry \eqref{Eq:cycanti}, one has
\begin{equation*}
 \eqref{JAC34}=
-\frac14\, (\dgal{a , \dgal{c,b}}_L)_{kj,iv,ul}
-\frac14 \fsgn_\alpha(\theta(i),\theta(j)) \,  (\dgal{\varpi(a) , \dgal{c,b}}_L)_{k\theta(i),\theta(j)v,ul}\,.
\end{equation*}
It is straightforward to compute
\begin{align*}
 \eqref{JAC5}&=
 \frac14 \fsgn_\alpha(\theta(k),\theta(l)) \, (\dgal{a, \dgal{\varpi(b),c}}_L)_{uj,i\theta(k),\theta(l)v} \,, \\
 \eqref{JAC7}&=
-\frac14 \fsgn_\alpha(\theta(k),\theta(l))\, (\dgal{a , \dgal{c,\varpi(b)}}_L)_{\theta(l)j,iv,u\theta(k)} \,.
\end{align*}
Then, we also need the $\varpi$-typed condition \eqref{Eq:BrTyped} to obtain
\begin{align*}
\eqref{JAC6}&=
\frac14 \fsgn_\alpha(\theta(i),\theta(j))\fsgn_\alpha(\theta(k),\theta(l)) \,(\dgal{\varpi(a), \dgal{\varpi(b),c}}_L)_{u\theta(i),\theta(j)\theta(k),\theta(l)v} \\
&=\frac14 \fsgn_\alpha(\theta(i),\theta(j))\fsgn_\alpha(\theta(k),\theta(l)) \,(\dgal{\varpi(a),\varpi^{\otimes 2} \dgal{b,\varpi^{-1}(c)}^{\circ}}_L)_{u\theta(i),\theta(j)\theta(k),\theta(l)v} \\
&=-\frac14 \fsgn_\alpha(\theta(i),\theta(j))\fsgn_\alpha(\theta(k),\theta(l)) \,(\varpi^{\otimes 3}\tau_{(12)}\dgal{a, \dgal{\varpi^{-1}(c),b}}_L)_{u\theta(i),\theta(j)\theta(k),\theta(l)v}\,.
\end{align*}
The defining relation \eqref{Eq:DefRelType2} of $\CC[\Rep^{\varpi,\theta}(A,\alpha)]$ and $\varpi^{-1}(c)=\fs(\iota(u))\fs(\iota(v))\varpi(c)$ entail
\begin{align*}
\eqref{JAC6}
&=-\frac14 \fsgn_\alpha(\theta(u),\theta(v)) \,(\dgal{a, \dgal{\varpi(c),b}}_L)_{kj,i \theta(u),\theta(v)l}
\end{align*}
after making the following simplification that follows from \eqref{Eq:sgnTypePr}
\begin{align*}
&\,\fsgn_\alpha(\theta(i),\theta(j))\fsgn_\alpha(\theta(k),\theta(l))\fs(\iota(u))\fs(\iota(v))\, \fsgn_\alpha(u,\theta(i))\fsgn_\alpha(\theta(j),\theta(k))\fsgn_\alpha(\theta(l),v) \\
&=\fs(\iota(u))\fs(\iota(v)) \fsgn_\alpha(u,v) = \fsgn_\alpha(\theta(u),\theta(v))\,.
\end{align*}
In the same way, we get
\begin{align*}
 \eqref{JAC8}
 &=-\frac14 \fsgn_\alpha(\theta(i),\theta(j)) \fsgn(\theta(k),\theta(l)) \,
 (\dgal{\varpi(a) , \dgal{c,\varpi(b)}}_L)_{\theta(l)\theta(i),\theta(j)v,u\theta(k)} \\
&=\frac14 \fsgn_\alpha(\theta(i),\theta(j)) \fsgn(\theta(k),\theta(l)) \,
( \varpi^{\otimes 3}\tau_{(12)}\,\dgal{a ,\dgal{b,\varpi^{-1}(c)}}_L)_{\theta(l)\theta(i),\theta(j)v,u\theta(k)} \\
&=\frac14 \fsgn_\alpha(\theta(u),\theta(v)) \,(\dgal{a, \dgal{b,\varpi(c)}}_L)_{\theta(v)j,il,k\theta(u)}
\end{align*}
after simplifying factors as
\begin{align*}
&\,\fsgn_\alpha(\theta(i),\theta(j))\fsgn_\alpha(\theta(k),\theta(l))\fs(\iota(u))\fs(\iota(v))\, \fsgn_\alpha(\theta(j),v)\fsgn_\alpha(\theta(l),\theta(i)) \fsgn_\alpha(u,\theta(k)) \\
&=\fs(\iota(u))\fs(\iota(v)) \fsgn_\alpha(u,v) = \fsgn_\alpha(\theta(u),\theta(v))\,.
\end{align*}
Gathering these expressions, we can write
\begin{align*}
&\br{a_{ij} , \br{ b_{kl},c_{uv} }}
\,=\,\frac14 (\dgal{a, \dgal{b,c}}_L)_{uj,il,kv} -\frac14 (\dgal{a, \dgal{c,b}}_L)_{kj,iv,ul} \\
&+\frac14 \fsgn_\alpha(\theta(i),\theta(j)) \, \left((\dgal{\varpi(a), \dgal{b,c}}_L)_{u\theta(i),\theta(j)l,kv}
-(\dgal{\varpi(a), \dgal{c,b}}_L)_{k\theta(i),\theta(j)v,ul} \right)\\
&+\frac14 \fsgn_\alpha(\theta(k),\theta(l)) \, \left((\dgal{a, \dgal{\varpi(b),c}}_L)_{uj,i\theta(k),\theta(l)v}
- (\dgal{a, \dgal{c,\varpi(b)}}_L)_{\theta(l)j,iv,u\theta(k)} \right)\\
&+ \frac14 \fsgn_\alpha(\theta(u),\theta(v)) \, \left((\dgal{a, \dgal{b,\varpi(c)}}_L)_{\theta(v)j,il,k\theta(u)}
- (\dgal{a, \dgal{\varpi(c),b}}_L)_{kj,i\theta(u),\theta(v)l} \right) \,.
\end{align*}
We easily deduce from this equality the expressions for $\br{b_{kl} , \br{ c_{uv},a_{ij} }}$ and $\br{c_{uv}, \br{a_{ij} , b_{kl} }}$ by permuting factors; summing everything up and using the definition of the triple bracket \eqref{Eq:TripBr} yields  \eqref{Eq:JacType}.
\end{proof}

\begin{lem} \label{Lem:ActType}
The action of $G=\prod_{s\in I} G_s$ (where $G_s=\Orm(\alpha_s)$ or $\Sp(\alpha_s)$ according to the type of $s$) on $\Rep^{\varpi,\theta}(A,\alpha)$
 leaves invariant  the antisymmetric biderivation $\br{-,-}_{\varpi,\theta}$ of Proposition \ref{Pr:BrType}.
\end{lem}
\begin{proof}
 Recall that any $g\in G$ satisfies $g^{-1}=\Theta g^T \Theta^T$ hence $(g^{-1})_{ij}=\fsgn_\alpha(\theta(i),\theta(j)) g_{\theta(j),\theta(i)}$. Using the $G$-action induced by \eqref{Infgrp}, we need to check
\begin{equation*}
 \br{g\cdot a_{ij},g\cdot b_{kl}}_{\varpi,\theta}
 =g\cdot \br{a_{ij},b_{kl}}_{\varpi,\theta}\,.
\end{equation*}
This is an easy computation based on \eqref{Eq:BrPtype} that uses \eqref{Eq:sgnTypePr}.
\end{proof}

\subsection{Main results for mixed types}

We start by generalizing Theorems \ref{Thm:BrPO} and \ref{Thm:BrSmp}.

\begin{thm}[Mixed Poisson case] \label{Thm:PType}
Under the assumptions of Proposition \ref{Pr:BrType}, if $\dgal{-,-}$ is a double Poisson bracket,
then the antisymmetric biderivation $\br{-,-}_{\varpi,\theta}$ on $\Rep^{\varpi,\theta}(A,\alpha)$ uniquely determined by
\eqref{Eq:BrPtype} is a Poisson bracket for which $G :=\prod_{s\in I} G_s$ acts by Poisson automorphisms.

\noindent Furthermore, if $\mu \in A$ is a moment map, it can be chosen (up to $\mu \mapsto \mu+c$ with $c\in A_0$) to satisfy
 \begin{equation}\label{mum-Type}
 \mu+\varpi(\mu)=0 \,;
\end{equation}
then $\X(\mu):\Rep^{\varpi,\theta}(A,\alpha) \to \g$ is a moment map for $\br{-,-}_{\varpi,\theta}$ after identification of $\g$ and its dual.
\end{thm}
\begin{proof}
In view of \eqref{Eq:JacType}, the Poisson property $\dgal{-,-,-}=0$ of the double bracket implies that $\br{-,-}_{\varpi,\theta}$ is a Poisson bracket.
That $G$ acts by Poisson automorphisms is a consequence of Lemma \ref{Lem:ActType}.

Next, we check that $\X(\mu)$ is a moment map.
Notice that, in the $\varpi$-adapted case, Proposition \ref{Pr:muselect} can be rewritten to require \eqref{mum-Type}.
Indeed, it is based on the equality $\dgal{\varpi(\mu_s),b}=-\dgal{\mu_s,b}$ which can be proven as in \eqref{mum-b} by \eqref{Eq:BrTyped2} because $\mu_s \in e_s A e_s$ so that $\varpi^2(\mu_s)=\mu_s$.
Then, we can combine \eqref{Eq:thetaType} and  \eqref{mum-Type} to derive $\X(\mu)+\Theta \X(\mu)^T \Theta^T=0_N$, which amounts to say that $\X(\mu)$ is $\g$-valued.

Finally, it remains to verify \eqref{momap}.
If $s$ is of type $\Orm$, one needs to check the analogue of \eqref{Eq:MuFO}; its proof can be reproduced easily.
If $s$ is of type $\Sp$, one needs to check the analogue of \eqref{Eq:MuFSp}. Again this can be done without effort:
$\iota(i),\iota(j)$ are of type $\Sp$, so $\fsgn_\alpha(\theta(i),\theta(j))=\sgn_\alpha(i,j)$ in \eqref{Eq:BrPtype} and we can rewrite all the equalities after \eqref{Eq:MuFSp}.
\end{proof}

Then, we can extend  Theorems \ref{Thm:BrqPO} and \ref{Thm:BrqPsp}.

\begin{thm}[Mixed quasi-Poisson case] \label{Thm:qPType}
Under the assumptions of Proposition \ref{Pr:BrType}, if $\dgal{-,-}$ is a double quasi-Poisson bracket,
then the antisymmetric biderivation $\br{-,-}_{\varpi,\theta}$ on $\Rep^{\varpi,\theta}(A,\alpha)$ uniquely determined by
\eqref{Eq:BrPtype} is a quasi-Poisson bracket for the action of $G :=\prod_{s\in I} G_s$.

\noindent Furthermore, if $\Phi \in A^\times$ is a moment map, it can be chosen (up to $\Phi \mapsto c \Phi$ with $c\in A_0^\times$) to satisfy
\begin{equation}\label{phim-Type}
\Phi\, \varpi(\Phi)=1 \,;
\end{equation}
then $\X(\Phi):\Rep^{\varpi,\theta}(A,\alpha) \to G$ is a moment map for $\br{-,-}_{\varpi,\theta}$.
\end{thm}
\begin{proof}
The $G$-invariance of the bracket is again a consequence of Lemma \ref{Lem:ActType}.
Showing that it satisfies the quasi-Poisson property \eqref{Jac-qP} depends on the type of each $s\in I$.
As the double bracket is quasi-Poisson, one has the defining identity \eqref{qPabc} of the triple bracket.
By trilinearity of the triple bracket, we have a decomposition
\[
 \dgal{-,-,-} = \sum_{s\in I} \dgal{-,-,-}^{(s)}\,,
\]
where $\dgal{-,-,-}^{(s)}:A^{\times 3}\to A^{\otimes 3}$ is the operation returning the $s$-th summand of \eqref{qPabc} when evaluated on the triple $(a, b, c)$. This allows to write thanks to \eqref{Eq:JacType},
\begin{equation} \label{Eq:pfType0}
 \begin{aligned}
 &\Jac_{\br{-,-}_{\varpi,\theta}}(a_{ij},b_{kl},c_{uv})
 =\sum_{s\in I} \Jac^{(s)}(a_{ij},b_{kl},c_{uv})\,, \quad \text{where }\\
\Jac^{(s)}&(a_{ij},b_{kl},c_{uv})
:=\,\,\frac14\dgal{a,b,c}^{(s)}_{uj,il,kv} -\frac14 \dgal{a,c,b}^{(s)}_{kj,iv,ul} \\
&+\frac14 \fsgn_\alpha(\theta(i),\theta(j)) \, (\dgal{\varpi(a),b,c}^{(s)}_{u\theta(i),\theta(j)l,kv}
-\dgal{\varpi(a),c,b}^{(s)}_{k\theta(i),\theta(j)v,ul} )\\
&+\frac14 \fsgn_\alpha(\theta(k),\theta(l)) \, (\dgal{a,\varpi(b),c}^{(s)}_{uj,i\theta(k),\theta(l)v}
- \dgal{a,c,\varpi(b)}^{(s)}_{\theta(l)j,iv,u\theta(k)} )\\
&+ \frac14 \fsgn_\alpha(\theta(u),\theta(v)) \, (\dgal{a,b,\varpi(c)}^{(s)}_{\theta(v)j,il,k\theta(u)}
- \dgal{a,\varpi(c),b}^{(s)}_{kj,i\theta(u),\theta(v)l} ) \,.
\end{aligned}
\end{equation}

The quasi-Poisson property will then be a consequence of the following equality (for each $s\in I$)
\begin{equation} \label{Eq:pfType1}
 \Jac^{(s)}(a_{ij},b_{kl},c_{uv})=\left\{
\begin{array}{ll}
 \frac12\psi_M^{\og(\alpha_s)} (a_{ij},b_{kl},c_{uv}) & s\text{ is of type }\Orm, \\ & \\
 \frac12\psi_M^{\spg(\alpha_s)} (a_{ij},b_{kl},c_{uv}) & s\text{ is of type }\Sp.
\end{array}
 \right.
\end{equation}
Firstly, assume that $s$ is of type $\Orm$.
Then, the RHS of \eqref{Eq:pfType1} is the $s$-th summand of \eqref{Eq:pfqPO2}.
Proceeding as in the proof of Theorem \ref{Thm:BrqPO} with notations therein, but using the relation
($j_s,k_s\in\mathtt{R}_s$ \eqref{Eq:sIndices})
\begin{align*}
&[\X(a),E_{j_sk_s}]_{il}= ((ae_s)_{i j_s} (e_s)_{k_sl}- (e_s)_{ij_s}(e_s a)_{k_sl})\\
=&\fsgn_\alpha(\theta(i),\theta(j_s)) \fsgn_\alpha(\theta(k_s),\theta(l)) \,
((e_s\varpi(a))_{\theta(j_s),\theta(i)} (e_s)_{\theta(l),\theta(k_s)}- (e_s)_{\theta(j_s),\theta(i)}(\varpi(a) e_s)_{\theta(l),\theta(k_s)})\\
=&\fsgn_\alpha(\theta(i),\theta(l)) \,
((e_s\varpi(a))_{j_s,\theta(i)} (e_s)_{\theta(l)k_s}- (e_s)_{j_s\theta(i)}(\varpi(a) e_s)_{\theta(l),k_s}) \,.
\end{align*}
Here, the first equality follows from \eqref{Eq:DefRelType3}, and the second equality from the fact that $s$ is of type $\Orm$.
This yields (where \eqref{Eq:pfqPO2}$^{(s)}_A$ denotes the summand corresponding to $s\in I$ in \eqref{Eq:pfqPO-A1}, etc.)
\small
\begin{align*}
 \eqref{Eq:pfqPO2}^{(s)}_A&=
 - \frac14 \fsgn_\alpha(\theta(k),\theta(l))\, \dgal{a,c,\varpi(b)}^{(s)}_{\theta(l)j,iv,u\theta(k)}, \quad
&&\eqref{Eq:pfqPO2}^{(s)}_B=
\frac14 \fsgn_\alpha(\theta(u),\theta(v)) \, \dgal{a,b,\varpi(c)}^{(s)}_{\theta(v)j,il,k\theta(u)}, \\
 \eqref{Eq:pfqPO2}^{(s)}_C&=
-\frac14 \dgal{a,c,b}^{(s)}_{kj,iv,ul} ,
 &&\eqref{Eq:pfqPO2}^{(s)}_D=
 \frac14 \fsgn_\alpha(\theta(i),\theta(j)) \, \dgal{\varpi(a),b,c}^{(s)}_{u\theta(i),\theta(j)l,kv}, \\
 \eqref{Eq:pfqPO2}^{(s)}_E&=
-\frac14 \fsgn_\alpha(\theta(i),\theta(j)) \, \dgal{\varpi(a),c,b}^{(s)}_{k\theta(i),\theta(j)v,ul},
&& \eqref{Eq:pfqPO2}^{(s)}_F=
\frac14\dgal{a,b,c}^{(s)}_{uj,il,kv} ,\\
 \eqref{Eq:pfqPO2}^{(s)}_G&=
-\frac14 \fsgn_\alpha(\theta(u),\theta(v)) \, \dgal{a,\varpi(c),b}^{(s)}_{kj,i\theta(u),\theta(v)l},
&&\eqref{Eq:pfqPO2}^{(s)}_H=
\frac14 \fsgn_\alpha(\theta(k),\theta(l)) \, \dgal{a,\varpi(b),c}^{(s)}_{uj,i\theta(k),\theta(l)v}.
\end{align*}
\normalsize
Gathering those terms and using \eqref{Eq:pfType0}, we have established \eqref{Eq:pfType1} in type $\Orm$.

Secondly, assume that $s$ of type $\Sp$.
In that case, we can rewrite \eqref{InfAct-sp} and after that the expressions for $\Psi^{(1)},\Psi^{(2)}$ simply by replacing $\tau$ with $\theta$ and $\sgn_\alpha$ with $\fsgn_\alpha$ (this requires \eqref{Eq:sgnTypePr}).
After relabelling indices, one can again identify $\Psi^{(1)}$ and $\Psi^{(2)}$ term by term.
We directly deduce,
\begin{equation*}
 \eqref{Psi1-000}^{(s)}+\eqref{Psi2-111}^{(s)}
=\frac14 \dgal{a,b,c}^{(s)}_{uj,il,kv}, \qquad
\eqref{Psi1-111}^{(s)}+\eqref{Psi2-000}^{(s)}
= -\frac14 \dgal{a,c,b}_{kj,iv,ul}^{(s)}.
\end{equation*}
For the other terms, we use \eqref{Eq:sgnTypePr} and \eqref{Eq:DefRelType3} to get the following analogue of \eqref{Eq:pfqPsp3}:
\begin{equation} \label{Eq:pfType5}
 \fsgn_\alpha(i_s,j_s) [\X(a),E_{\theta(j_s),\theta(i_s)}^{(s)}]_{pq}
 =  \fsgn_\alpha(\theta(p),\theta(q)) [\X(\varpi(a)),E_{\theta(p),\theta(q)}^{(s)}]_{j_si_s}\,.
\end{equation}
It is then easy to adapt the other calculations to deduce:
\begin{align*}
 \eqref{Psi1-100}^{(s)}+\eqref{Psi2-011}^{(s)}
 &=+\frac14 \fsgn_\alpha(\theta(i),\theta(j))\,\dgal{\varpi(a),b,c}_{u\theta(i),\theta(j)l,kv}^{(s)}, \\
\eqref{Psi1-011}^{(s)}+\eqref{Psi2-100}^{(s)}
&=-\frac14 \fsgn_\alpha(\theta(i),\theta(j))\,\dgal{\varpi(a),c,b}_{k\theta(i),\theta(j)v,ul}^{(s)}, \\
\eqref{Psi1-010}^{(s)}+\eqref{Psi2-101}^{(s)}
&=+\frac14 \fsgn_\alpha(\theta(k),\theta(l))\,\dgal{a,\varpi(b),c}_{uj,i\theta(k),\theta(l)v}^{(s)}, \\
\eqref{Psi1-101}^{(s)}+\eqref{Psi2-010}^{(s)}
&=-\frac14 \fsgn_\alpha(\theta(k),\theta(l))\,\dgal{a,c,\varpi(b)}_{\theta(l)j,iv,u\theta(k)}^{(s)}, \\
\eqref{Psi1-001}^{(s)}+\eqref{Psi2-110}^{(s)}
&=+\frac14 \fsgn_\alpha(\theta(u),\theta(v))\,\dgal{a,b,\varpi(c)}_{\theta(v)j,il,k\theta(u)}^{(s)}, \\
\eqref{Psi1-110}^{(s)}+\eqref{Psi2-001}^{(s)}
&=-\frac14 \fsgn_\alpha(\theta(u),\theta(v))\,\dgal{a,\varpi(c),b}_{kj,i\theta(u),\theta(v)l}^{(s)}.
\end{align*}
Using \eqref{Eq:pfType0} again, we can conclude for the type $\Sp$.

\medskip

Next, we check that $\X(\Phi)$ is a moment map.
Firstly, we can adapt Proposition \ref{Pr:Phiselect} to the $\varpi$-adapted case so as to require \eqref{phim-Type}.
This follows from the fact that we can derive \eqref{Phim-tw} with $\varpi$ in place of $\phi$.
Then,  \eqref{phim-Type} combined with \eqref{Eq:thetaType} yields $\X(\Phi)\Theta \X(\Phi)^T \Theta^T=\Id_N$, thus $\X(\Phi)$ is $G$-valued.

Secondly, it remains to verify \eqref{Gmomap}.
If $s$ is of type $\Orm$, one needs to check the analogue of \eqref{Eq:pfqPO5}; its proof can be reproduced easily.
If $s$ is of type $\Sp$, one needs to check the analogue of \eqref{Eq:pfqPsp5}; this can also be reproduced almost verbatim with the same justification as in the proof of Theorem \ref{Thm:PType}.
\end{proof}

\subsection{Additional results}

\subsubsection{A proof of Proposition \ref{Pr:RepP-Type} in NC geometry} \label{ss:NCproof}
As in the standard proof, there is nothing to do if $G_s=\Gl(\alpha_s)$. Otherwise, we make the following construction.
Firstly, $A^{\op}$ inherits a double Poisson bracket by Lemma \ref{lem:Op-P}.
By performing fusion of $A$ and $A^{\op}$ as in Example \ref{Ex:Fus}, we get a double Poisson bracket $\dgal{-,-}^-$ on $\overline{A}=A\ast_{A_0} A^{\op}$.
For any $a\in A$, we write $a\in \overline{A}$ for its inclusion under $A\hookrightarrow \overline{A}$, and
$\bar{a}\in \overline{A}$ for its inclusion under $A^{\op}\hookrightarrow \overline{A}$. Then, $\dgal{-,-}^-$ is uniquely determined by, cf. \eqref{Eq:dgalOp},
\begin{equation} \label{Eq:BrGenPf1}
 \dgal{a,b}^- = \dgal{a,b}, \quad \dgal{a,\bar{b}}^- = 0, \quad
 \dgal{\bar a,\bar b}^- = \overline{\dgal{a,b}''} \otimes \overline{\dgal{a,b}'}.
\end{equation}
Also, if $\mu\in A$ is a moment map, then by Lemma \ref{lem:Op-P} $\mu-\bar{\mu}\in \bar{A}$ is a moment map.

We now define a mapping $\fs$ \eqref{Eq:fs} encoding a choice of types.
If, in the notation of the statement of the proposition, $G_s=\Orm(\alpha_s)$ or\footnote{We only give a type to $s$ when $G_s=\Gl(\alpha_s)$ so as to construct a twisted representation space depending on a type. This is a purely artificial choice.} $\Gl(\alpha_s)$ we put $\fs(s)=+1$
while we put $\fs(s)=-1$ if  $G_s=\Sp(\alpha_s)$.
Since $A=\oplus_{s,t\in I} e_s A e_t$, we put
\[
 \varpi(a)=\fs(s)\fs(t)\bar{a}, \quad \varpi(\bar a)= a, \qquad a\in e_s Ae_t, \,\, s,t\in I,
\]
which we can extend to a typed anti-involution on $\overline{A}$ since $\bar{b}\cdot_{\op}\bar{a}=\overline{ab}$.
It is immediate to verify from \eqref{Eq:BrGenPf1} that $\dgal{-,-}^-$ is $\varpi$-adapted \eqref{Eq:BrTyped}.

Secondly, we consider the twisted representation space $\Rep^{\varpi,\theta}(\overline{A},\alpha)$ of mixed type.
Its coordinate ring is isomorphic to the one of $\Rep(A,\alpha)$ since by \eqref{Eq:DefRelType2},
$\bar{a}_{ij}=\fs(s)\fs(t)(a|\theta(E_{ij}))$ for all $a\in e_s Ae_t$ with $s,t\in I$.
By considering \eqref{Eq:BrGenPf1} in \eqref{Eq:BrPtype} (for the general mixed case of Theorem \ref{Thm:PType}) with $a,b\in A$,
one sees that the Poisson bracket reduces to the defining relation \eqref{Eq:BrGen}.
For the moment map, one has
\[
 Y_s=\X(\mu_s-\bar{\mu}_s)=\X(\mu_s) - \Theta\X(\mu_s)^T\Theta^T\,.
\]
If $G_s=\Orm(\alpha_s)$ or $G_s=\Orm(\alpha_s)$ (for $G_s$ as in the statement),
this reduces to the $\og(\alpha_s)$- or $\spg(\alpha_s)$-valued moment map in \eqref{Eq:MomapGen}.
 \qed

\subsubsection{Fusion for typed double quasi-Poisson brackets}

Let us mention that the fusion operation from \ref{ss:Fus} can be applied \emph{provided that} the fused idempotents $e_1,e_2$ are such that the corresponding indices $1,2$ have the same type, i.e. $\fs(1)=\fs(2)$.
In that case, the typed anti-involution $\varpi:A\to A$ descends to  $\varpi_f:A^f\to A^f$ by repeating the discussion leading to $\phi_f$ \eqref{Eq:phiExt}.
Moreover, $\varpi_f$ is also a typed anti-involution if we consider the mapping $\fs_f$ obtained from $\fs$ by omitting the index $2$ (or equivalently by omitting $1$).
One can then directly rewrite Proposition \ref{Pr:CompFus} in the typed case.
We also easily adapt Lemma \ref{Lem:Tw-Rep} as follows.

\begin{lem} \label{Lem:Type-Rep}
Assume that $\fs(1)=\fs(2)$, $\alpha_1=\alpha_2$ and let $\alpha^f=(\alpha_1,\alpha_3,\ldots,\alpha_{|I|})$.
Define the endomorphism $\theta_f$ on $\gl_{N-\alpha_1}$ as in \eqref{Eq:thetaType} with $\Theta=\diag(\Theta_1,\Theta_3,\ldots,\Theta_{|I|})$.
Consider the embedding $G_{\alpha^f}\hookrightarrow G_{\alpha}$ where the embedding on the first factor is diagonal
and on the other factors is the identity.
Then, there is a canonical isomorphism between $\Rep^{\varpi,\theta}(A,\alpha)$ and $\Rep^{\varpi_f,\theta_f}(A^f,\alpha^f)$ such that
the induced $G_{\alpha^f}$ action on $\Rep^{\varpi,\theta}(A,\alpha)$ is obtained by restriction from the $G_{\alpha}$-action.
\end{lem}

\subsection{Typed double Poisson brackets from quivers}

\subsubsection{Case I} \label{sss:CaseI}

We adapt the discussion in \ref{sss:P-exmp} to encompass a choice of types. 
Consider the path algebra $\CC\overline{Q}$ of a quiver $Q=(Q_{\Upsilon,\Lambda},I,t,h)$ as in \ref{sss:P-exmp}.
Fix a choice of types for $\CC\overline{Q}$.
We modify the involutive anti-automorphism $\phi$ \eqref{Eq:PhiQ} as follows:
$\varpi$ is the anti-automorphism on $\CC\overline{Q}$
that fixes the idempotents and such that on arrows
\begin{subequations} \label{Eq:VarpiQ}
 \begin{align}
&\varpi(a)=\fs(t(a))\fs(h(a))\, a', \,\, \varpi(a')=a, \,\, \text{ for }a\in \Upsilon\setminus\Lambda,
&&\varpi(c)=\gamma_c \,c,\,\, \text{ for }c\in \Lambda ; \label{Eq:VarpiQa} \\
&\varpi(a^\ast)=\fs(t(a))\fs(h(a))\, {a'}^\ast, \,\, \varpi({a'}^\ast)=a^\ast, \,\, \text{ for }a\in \Upsilon\setminus\Lambda,
&&\varpi(c^\ast)=\gamma_c \,c^\ast,\,\, \text{ for }c\in \Lambda . \label{Eq:VarpiQb}
 \end{align}
\end{subequations}

\begin{prop} \label{Pr:VarpiQ}
The anti-automorphism $\varpi$ given by \eqref{Eq:VarpiQ} is a typed anti-involution.
Moreover, the double Poisson bracket on $\CC\overline{Q}$ canonically defined on a double quiver by Van den Bergh is $\varpi$-typed.
Finally, it admits the moment map $\mu\in \CC\overline{Q}$ given by \eqref{Eq:MomapQ}, which satisfies \eqref{mum-Type}.
\end{prop}
\begin{proof}
To show that $\varpi$ is a typed anti-involution, it suffices to check the defining property on generators.
This is easy to verify since by \eqref{Eq:VarpiQ}:
$\varpi^2(b) = \fs(t(b))\fs(h(b))\,b$ for all $b \in \overline{Q}$.
To verify \eqref{Eq:BrTyped} and the rest of the statement, it suffices to adapt of the proof of Proposition \ref{Pr:Upsilon}.
\end{proof}

After that, take a typed dimension vector $\alpha \in \N^I$.
By Theorem \ref{Thm:PType}, the twisted representation space $\Rep^{\varpi,\theta}(\CC\overline{Q},\alpha)$ corresponding to our choice of types inherits the Poisson bracket $\br{-,-}_{\varpi,\theta}$ \eqref{Eq:BrPtype}.
The moment map on that space adopts a block decomposition similar to  \eqref{Eq:MomapQ-expl}, namely
\begin{equation*} 
 \X(\mu_s)=\sum_{\substack{a\in \Upsilon\\t(a)=s}}  (\X(a)\X(a^\ast) - \X(a^\ast)^T\X(a)^T)
 + \sum_{\substack{a\in \Upsilon\\h(a)=s}}  (-\X(a^\ast)\X(a) + \X(a)^T\X(a^\ast)^T)
\end{equation*}
if $s$ has type $\Orm$, and
\begin{equation*} 
 \X(\mu_s)=\sum_{\substack{a\in \Upsilon\\t(a)=s}}  (\X(a)\X(a^\ast) - \Omega_{\alpha_s} \X(a^\ast)^T\X(a)^T\Omega_{\alpha_s})
 + \sum_{\substack{a\in \Upsilon\\h(a)=s}}  (-\X(a^\ast)\X(a) + \Omega_{\alpha_s}\X(a)^T\X(a^\ast)^T\Omega_{\alpha_s})
\end{equation*}
if $s$ has type $\Sp$.
When $\Lambda=\emptyset$, by a discussion analogous to \ref{ss:Untwist}, we get an identification
\[
 \Rep^{\varpi,\theta}(\CC\overline{Q},\alpha)
 \simeq \Rep(\CC\overline{\Upsilon},\alpha).
\]
Here, $G\subset \prod_s \Gl_{\alpha_s}$ acts by \eqref{Infgrp}, and
the Poisson bracket is the usual one~\cite{Nak} scaled by a factor $\frac12$.

\subsubsection{Case II}

Let $\Upsilon$ be a quiver and $\overline{\Upsilon}$ its double. (We are not using a ``quadruple'' $\overline{Q}$ as in \ref{sss:CaseI}.)
Assume that we can fix a choice of types $\fs$ on the vertex set such that the following holds:
\begin{equation} \label{Eq:CondMin}
 \forall a \in \Upsilon, \qquad \fs(t(a))\neq \fs(h(a)).
\end{equation}

In other words, any two vertices in $\Upsilon$ that are connected by an arrow must be assigned two different types.  
We introduce an anti-automorphism $\varpi_-: \CC\overline{\Upsilon} \to \CC\overline{\Upsilon}$ uniquely determined by  
\begin{equation}  \label{Eq:VarpiMin}
 \varpi_-(a)=-a^\ast, \,\, \varpi_-(a^\ast)=a, \,\, \text{ for }a\in \Upsilon.
\end{equation}

\begin{prop} \label{Pr:VarpiMin}
The anti-automorphism $\varpi_-$ given by \eqref{Eq:VarpiMin} is a typed anti-involution.
Moreover, the double Poisson bracket on $\CC\overline{\Upsilon}$ canonically defined on a double quiver by Van den Bergh is $\varpi_-$-typed.
Finally, it admits the moment map $\mu=\sum_{a\in \Upsilon}[a,a^\ast]\in \CC\overline{\Upsilon}$, which satisfies \eqref{mum-Type}.
\end{prop}
\begin{proof}
Combining \eqref{Eq:CondMin} and \eqref{Eq:VarpiMin}, we get 
$\varpi_-^2(b) = -b =  \fs(t(b))\fs(h(b))\,b$ for any $b\in \overline{\Upsilon}$. 
The rest of the proof is again easily adapted from Proposition \ref{Pr:Upsilon}. 
\end{proof}
Take a typed dimension vector $\alpha \in \N^I$, and define $\theta$ as in \eqref{Eq:thetaType}.
Consider the corresponding twisted representation space $\Rep^{\varpi_-,\theta}(\CC\overline{Q},\alpha)$.
The defining relation \eqref{Eq:DefRelType2}  entails
\begin{subequations} \label{Eq:Min1}
 \begin{align}
   -(a^\ast)_{ji} = - \sgn_\alpha(j) a_{i,\theta(j)},  \quad \text{if } \fs(t(a))=+1,\, \fs(h(a))=-1 \,; \label{Eq:Min1a} \\
   -(a^\ast)_{lk} = - \sgn_\alpha(k) a_{\theta(k),l},  \quad \text{if } \fs(t(a))=-1,\, \fs(h(a))=+1 \,; \label{Eq:Min1b}
 \end{align}
\end{subequations}
where we considered the only pairs of indices such that these relations are not trivially vanishing, i.e. 
\begin{equation*}
 \sum_{\substack{s\in I\\s<t(a)}}\alpha_{s}+1 \leq i,k \leq \sum_{\substack{s\in I\\s\leq t(a)}}\alpha_{s}, \quad 
 \sum_{\substack{s\in I\\s<h(a)}}\alpha_{s}+1 \leq j,l \leq \sum_{\substack{s\in I\\s\leq h(a)}}\alpha_{s}.
\end{equation*}
(Recall that we identify $I$ and $\{1,\ldots,|I|\}$ to define the representation space.) This yields
\begin{equation} \label{Eq:Min2}
 \X(a^\ast) = \Theta \X(a)^T \Theta^T,
\end{equation}
for any $a\in \Upsilon$,
cf. \eqref{Eq:thetaType}. In particular, as varieties 
$\Rep^{\varpi_-,\theta}(\CC\overline{\Upsilon},\alpha)
 \simeq \Rep(\CC\Upsilon,\alpha)$. 
By Theorem \ref{Thm:PType}, $\Rep^{\varpi_-,\theta}(\CC\overline{\Upsilon},\alpha)$ inherits the Poisson bracket $\br{-,-}_{\varpi_-,\theta}$ \eqref{Eq:BrPtype} and the moment map $\X(\mu)$.
With the previous identification, we can write that the Poisson bracket is uniquely determined by ($a,b\in \Upsilon$)
\begin{align*}
 \br{a_{ij},b_{kl}}_{\varpi,\theta}&= -\frac12\delta_{ab} \,\sgn_\alpha(j) (e_{t(a)})_{ki} (e_{h(a)})_{\theta(j),l} , 
 \quad \text{if } \fs(t(a))=+1,\, \fs(h(a))=-1 \,; \\
 \br{a_{ij},b_{kl}}_{\varpi,\theta}&= -\frac12\delta_{ab} \,\sgn_\alpha(i) (e_{t(a)})_{k,\theta(i)} (e_{h(a)})_{jl} , 
 \quad \text{if } \fs(t(a))=-1,\, \fs(h(a))=+1 \,.
\end{align*}
Moreover, the moment map becomes $\X(\mu) = \sum_{a\in \Upsilon} [\X(a),\Theta \X(a)^T \Theta^T]$.

As a particular example, take $\Upsilon$ to be the $A_{2\ell}$ quiver. Fix a choice of types by alternating them: take $\Orm$ for odd indices and $\Sp$ for even indices of the vertices.
By considering the representation spaces, we get the (complexified) master phase spaces of the orthosymplectic gauge theories from \cite[\S5]{GW09}, which give these theories by performing Hamiltonian reduction.
(As explained in Footnote \eqref{ftnote:TypeO}, one can take $\mathrm{SO}(\alpha_s)$ instead of $\Orm(\alpha_s)$.)
For an in-depth geometric study related to other quivers, see \cite{Li19,Nak25}.

\subsubsection{Losev's almost commuting variety}

Consider the double $\overline{\Upsilon}$ of the quiver $\Upsilon$ of Example \ref{Exmp:Almost}
made of the four arrows $x,v,x^\ast,v^\ast$.
Put $\fs(0)=-1$ and $\fs(\infty)=1$.
There is a typed anti-involution $\varpi$ on $\CC\overline{\Upsilon}$ given by
$\varpi(x)=-x$, $\varpi(x^\ast)=-x^\ast$,  $\varpi(v)=v^\ast$ and $\varpi(v^\ast)=-v$.
One can check \eqref{Eq:BrTyped} (for Van den Bergh's double Poisson bracket) on the four generators, so that the double bracket is $\varpi$-typed.
The moment map is $\mu = [x,y]+[v,v^\ast]$.
Take the dimension vector $(n_0,n_\infty)=(2n,1)$ for some $n\geq 1$. One has
$$\Rep^{\varpi,\theta}(\CC \overline{\Upsilon},(2n,1))\simeq \{X,Y\in \spg_{2n}, \,\, V\in \CC^{2n} \}\,,$$
with $\spg_{2n}$-valued moment map
$\widetilde{\mu}:=[X,Y]+VV^T\Omega_{2n}^T$ by \eqref{Eq:DefRelType2}, and
Poisson structure, see \eqref{Eq:BrPtype},
\begin{equation*}
 \br{X_{ij},Y_{kl}}=\frac12 (\delta_{kj}\delta_{il}-\sgn_{2n}(i,j)\, \delta_{k,i+n}\delta_{j+n,l}), \quad
 \br{V_{i},V_{k}}=\frac12 \sgn_{2n}(i)\, \delta_{k,i+n}\,,
\end{equation*}
with indices understood modulo $2n$, cf. Remark \ref{Rem:Sympl}.
The scheme $\widetilde{\mu}^{-1}(0_{2n})$ has appeared as $X_{n}$ in Losev's work \cite{Los}, where it was shown to be irreducible, reduced complete intersection of dimension $2n^2 + 3n$. Furthermore, one has for its reduction
$\widetilde{\mu}^{-1}(0_{2n})/\!/\Sp_{2n}\simeq (\mathfrak{h}\times\mathfrak{h})/\!/W$ where $\mathfrak{h},W$ are the Cartan and Weyl group of $\spg_{2n}$.


\section{Fox pairings on Hopf algebras, after Massuyeau and Turaev} \label{Sec:MT}

In Section \ref{Sec:DqPoi}, we could consider the free Laurent algebra
$L_{g,r}$ thanks to its realisation as $\CC\pi_1(\Sigma,\ast)$ for $\Sigma$ a genus $g$ surface with $r+1$ boundary components and a marked point.
The construction of its double quasi-Poisson bracket is due to Massuyeau-Turaev \cite{MT14} and is based on the existence of a Fox pairing. In \cite{MT18}, this point of view was pursued to induce quasi-Poisson brackets on the representation algebra valued in a (graded) bialgebra. Here, we review the relevant parts of \cite{MT14,MT18} and then we explain the relation between their formalism (without grading, over $\CC$) and the results derived in Section \ref{Sec:DqPoi}.

\medskip

Fix a cocommutative Hopf algebra $\cH$, which is in particular involutive, i.e., the antipode $S$ satisfies $S^2=\id_\cH$.
Write the comultiplication as
$\Delta(a)=a_{[1]}\otimes a_{[2]}$ for $a\in \cH$, and by coassociativity we can put
\[
 \Delta^{k-1}(a)=a_{[1]}\otimes a_{[2]} \otimes \cdots \otimes a_{[k]}, \quad k\geq 2,
\]
where $\Delta^{k-1}$ is any iteration of $k-1$ copies of the comultiplication. 
(We warn the reader that this notation is different from the Sweedler's type notations used previously, though we can write $a_{[1]}\otimes a_{[2]}=\Delta(a)' \otimes \Delta(a)''$.)
Then we get the following usual identities involving the counit $\epsilon:\cH\to \CC$: 
$a=a_{[1]} \,\epsilon(a_{[2]})$, $\epsilon(a)1=a_{[1]}\,S(a_{[2]})$ and analogous formulas when acting on the first factor.
Note that the involutivity of $S$ entails  
$\Delta(S(a))=S(a_{[2]})\otimes S(a_{[1]})$.

A left, resp. right, Fox derivative is a $\CC$-linear map $\partial_L:\cH\to \cH$, resp. $\partial_R:\cH\to \cH$, satisfying
\begin{align*}
 \partial_L(ab) = \partial_L(a)\, \epsilon(b) + a \,\partial_L(b), \quad
 \text{resp.}\,\,\partial_R(ab) = \partial_R(a)\, b + \epsilon(a)\, \partial_R(b), \quad a,b\in \cH.
\end{align*}
A Fox pairing is a $\CC$-linear map  $\rho:\cH\times \cH \to \cH$ which is a left Fox derivative in its first argument, and a right Fox derivative in its second argument. 
One obtains from \cite[\S~5]{MT14} that the transpose 
\begin{equation*}
 \overline{\rho}: \cH\times \cH \to \cH, \quad \overline{\rho}(a,b)=S\big( \rho(S(b),S(a)) \big),
\end{equation*}
is also a Fox pairing. Moreover, $\overline{\overline{\rho}}=\rho$ because $\cH$ is involutive, and one has 
\begin{equation}
 \overline{\rho}(a,b)=a_{[1]}\, \rho(b_{[2]},a_{[2]}) \, b_{[1]}\,, \quad a,b\in \cH.
\end{equation}
A Fox pairing $\rho$ is symmetric if $\overline{\rho}=\rho$ and antisymmetric if $\overline{\rho}=-\rho$. Obviously, any Fox pairing can be decomposed in a unique sum of symmetric and antisymmetric parts.

\begin{prop}[\cite{MT14}, \S6.1]
 If $\rho$ is an antisymmetric Fox pairing, then $\dgal{-,-}_\rho : \cH\otimes \cH \to \cH\otimes \cH$,
\begin{equation} \label{Eq:dbr-Fox}
 \dgal{a,b}_\rho= b_{[1]} S(\rho(a_{[2]},b_{[2]})_{[1]}) a_{[1]} \otimes \rho(a_{[2]},b_{[2]})_{[2]}, \quad a,b\in \cH,
\end{equation}
is a double bracket. 
\end{prop}

Following \cite[\S9]{MT18} without grading, we associate to a Fox pairing $\rho$ an operation 
$\mathtt{F}_\rho : \cH^{\otimes 3} \to \cH^{\otimes 3}$,
\begin{align*}
&\mathtt{F}_\rho(a,b,c) \\
=& \, b_{[1]} S(\rho(a_{[2]},b_{[2]})_{[1]}) a_{[1]} \otimes
c_{[1]} S\big(\rho\big( \rho(a_{[2]},b_{[2]})_{[3]}  ,c_{[2]}\big)_{[1]}\big) \rho(a_{[2]},b_{[2]})_{[2]} \otimes
\rho\big( \rho(a_{[2]},b_{[2]})_{[3]}  ,c_{[2]}\big)_{[2]}
\end{align*}
Then, we form the following \emph{tritensor map} 
\begin{equation*}
 \| -,-,- \|_\rho = \sum_{s=0,1,2} \tau_{(123)}^{-s} \circ \mathtt{F}_\rho \circ \tau_{(123)}^s \,:\,
 \cH^{\otimes 3} \to \cH^{\otimes 3}. 
\end{equation*}

\begin{prop}[\cite{MT18}, \S9 \& App.~B] \label{Pr:Hopf1}
Let $\rho$ be an antisymmetric Fox pairing and  $\dgal{-,-}_\rho$ be the associated double bracket defined by
\eqref{Eq:dbr-Fox}. 
\begin{enumerate}
 \item If $ \| -,-,- \|_\rho$ identically vanishes, then $\dgal{-,-}_\rho$ is Poisson.
 \item If the quasi-Poisson Fox property\footnote{Let us warn the reader that there is an unimportant factor $\frac12$ between the quasi-Poisson conventions of Massuyeau-Turaev and the one of Van den Bergh that we follow.\label{Ft:ConvMT}} holds, i.e. if for any $x,y,z\in \cH$,
\begin{equation}
 \begin{aligned} \label{Eq:qP-Fox}
\| x,y,z \|_\rho \,=\,\frac14\Big(&1 \otimes y \otimes xz + yx \otimes 1 \otimes z + x \otimes zy \otimes 1 + y\otimes z \otimes x \\
&- 1\otimes zy \otimes x - y \otimes 1 \otimes xz - yx \otimes z \otimes 1 - x \otimes y \otimes z \Big)  
 \end{aligned}
\end{equation}
then $\dgal{-,-}_\rho$ is quasi-Poisson. 
\end{enumerate}
\end{prop}
\begin{proof}
One easily computes the equality
$\mathtt{F}_\rho(a,b,c) = \tau_{(13)}\circ \{\!\!\{ c , \dgal{b,a}_\rho \}\!\!\}_{\rho,L}$.
This gives with \eqref{Eq:TripBr},
\begin{align*}
 \| -,-,- \|_\rho &= \sum_{s=0,1,2} \tau_{(123)}^{-s} \circ \tau_{(13)}\circ 
(\dgal{-,-}_\rho \otimes \id_\cH) \circ (\id_\cH\otimes \dgal{-,-}_\rho) 
 \circ \tau_{(13)} \circ \tau_{(123)}^s  \\
 &= \tau_{(13)}\circ \dgal{-,-,-}_\rho \circ \tau_{(13)}.
\end{align*}
Thus, $\| -,-,- \|_\rho$ vanishes if and only if the double Jacobiator $\dgal{-,-,-}_\rho$ vanishes. 
For the second statement, it suffices to note that applying $\tau_{(13)}$ to \eqref{Eq:qP-Fox} with $x=c$, $y=b$ and $z=a$ we obtain \eqref{qPabc} in the case $I=\{1\}$ where $e_1=1$. 
\end{proof}

Next, let us relate these constructions to our setting.

\begin{prop} \label{Pr:Hopf2}
Let $\rho$ be an antisymmetric Fox pairing. Then the double bracket $\dgal{-,-}_\rho$ defined by \eqref{Eq:dbr-Fox}
is $S$-adapted.  
\end{prop}
\begin{proof}
We need to verify \eqref{Eq:BrComp} where $\phi=S$ is the antipode. We start by computing thanks to \eqref{Eq:dbr-Fox},
\begin{align*}
\dgal{S(a),S(b)}_\rho^\circ &= 
 \rho(S(a)_{[2]},S(b)_{[2]})_{[2]}  \otimes S(b)_{[1]} S(\rho(S(a)_{[2]},S(b)_{[2]})_{[1]}) S(a)_{[1]} \\
&=  \rho\big(S(a_{[1]}),S(b_{[1]})\big)_{[2]}  \otimes
S(b_{[2]}) S\big(\rho\big(S(a_{[1]}),S(b_{[1]})\big)_{[1]}\big) S(a_{[2]}) \\
&=S^{\otimes 2}\Big( S\big( \rho\big(S(a_{[1]}),S(b_{[1]})\big)_{[2]} \big) \otimes
a_{[2]} \rho\big(S(a_{[1]}),S(b_{[1]})\big)_{[1]} b_{[2]}\Big)\,.
\end{align*}
Thus, in view of \eqref{Eq:BrComp}, the last equality should be $S^{\otimes 2}$ applied to $\dgal{a,b}_\rho$.
We record the following equalities that are obtained as in the proof of \cite[Lem.~5.2]{MT14}:
\begin{equation} \label{Eq:pfAntip0}
  \rho(S(x),y) = - S(x_{[1]})\,\rho(x_{[2]},y), \quad
 \rho(x,S(y)) = - \rho(x,y_{[1]})\, S(y_{[2]}), \quad x,y\in \cH.
\end{equation}
In particular, we can write, 
\begin{equation} \label{Eq:pfAntip1}
 \begin{aligned}
 &\Delta\circ\rho(S(x),S(y))=
 \Delta(S(x_{[1]}) \,\rho(x_{[2]},y_{[1]})\, S(y_{[2]})) \\
 =&\, S(x_{[2]}) \, \rho(x_{[3]},y_{[1]})_{[1]}\, S(y_{[3]}) \otimes
 S(x_{[1]})\, \rho(x_{[3]},y_{[1]})_{[2]}\, S(y_{[2]})\,.
 \end{aligned}
\end{equation}
We then get
\begin{align*}
&S\big( \rho\big(S(a_{[1]}),S(b_{[1]})\big)_{[2]} \big) \otimes
a_{[2]} \rho\big(S(a_{[1]}),S(b_{[1]})\big)_{[1]} b_{[2]} \\
=&S\Big( S(a_{[1]})\, \rho(a_{[3]},b_{[1]})_{[2]}\, S(b_{[2]}) \Big) \otimes
a_{[4]} S(a_{[2]}) \, \rho(a_{[3]},b_{[1]})_{[1]}\, S(b_{[3]}) b_{[4]} \\
=&b_{[1]}\, S\big( \rho(a_{[2]},b_{[2]})_{[1]} \big) \, a_{[1]}  \otimes
a_{[3]} S(a_{[4]}) \, \rho(a_{[2]},b_{[2]})_{[2]}\, S(b_{[3]}) b_{[4]} \\
=&b_{[1]}\, S\big( \rho(a_{[2]},b_{[2]})_{[1]} \big) \, a_{[1]}  \otimes
  \rho(a_{[2]} \epsilon(a_{[3]}),b_{[2]} \epsilon(b_{[3]}))_{[2]} \\
=&b_{[1]}\, S\big( \rho(a_{[2]},b_{[2]})_{[1]} \big) \, a_{[1]}  \otimes
  \rho(a_{[2]} ,b_{[2]})_{[2]} ,
\end{align*}
where we used \eqref{Eq:pfAntip1} in the first equality, the cocommutativity of $\cH$ in the second equality, and then standard Hopf algebraic manipulations.
This is precisely $\dgal{a,b}_\rho$ by  \eqref{Eq:dbr-Fox}.
\end{proof}

\begin{cor} \label{Cor:Hopf1}
Let $\rho$ be an antisymmetric Fox pairing such that the tritensor map $\|-,-,-\|_\rho$ is identically vanishing.
Then, for any $N\geq 1$, the twisted representation space $\Rep^{\phi,\tau}(\cH,N)$ is equipped with a Poisson bracket uniquely determined by $\dgal{-,-}_\rho$ and $\phi=S$ using \eqref{Eq:BrPO} in the orthogonal case, or using the identities in Remark \ref{Rem:Sympl} (with $N=2n$) in the symplectic case.

\noindent If instead the tritensor map $\|-,-,-\|_\rho$ satisfies the quasi-Poisson Fox property \eqref{Eq:qP-Fox}, the same statement holds except that the induced bracket is \emph{quasi-}Poisson.
\end{cor}
\begin{proof}
Combining Propositions \ref{Pr:Hopf1} and \ref{Pr:Hopf2}, the double bracket $\dgal{-,-}_\rho$ is Poisson and $S$-adapted. 
Then the result follows from the discussion of Section \ref{Sec:DPoi} with $|I|=1$ (which boils down to \cite{OS}).

For the quasi-Poisson version of the result, one uses either Theorem \ref{Thm:BrqPO} ($\Orm_N$ case) or Theorem \ref{Thm:BrqPsp} ($\Sp_N$ case) to conclude.
\end{proof}

\medskip

Finally, we explain how the (quasi-)Poisson bracket of Corollary \ref{Cor:Hopf1} can equivalently be derived using the formalism of Massuyeau and Turaev \cite{MT18}.
We start with a small computation.

\begin{lem}
For the double bracket defined by  \eqref{Eq:dbr-Fox}, one has for any $a,b\in \cH$,
\begin{equation} \label{Eq:Antip4}
 \dgal{a,S(b)}_\rho =
 - S\big(\rho(a_{[2]},b_{[1]})_{[1]}\, \big)\, a_{[1]}  \otimes \rho(a_{[2]},S(b_{[1]}))_{[2]} S(b_{[2]})\,.
\end{equation}

\end{lem}
\begin{proof}
This is a simple computation:
\begin{align*}
\dgal{a,S(b)}_\rho&=
S(b)_{[1]} \, S(\rho(a_{[2]},S(b)_{[2]})_{[1]})\, a_{[1]}  \otimes
\rho(a_{[2]},S(b)_{[2]})_{[2]} \\
&=
 S\big(\rho(a_{[2]},S(b_{[1]}))_{[1]}\, b_{[2]}\big)\, a_{[1]} \otimes
\rho(a_{[2]},S(b_{[1]}))_{[2]} \\
&=  -
 S\big(\rho(a_{[2]},b_{[1]})_{[1]}\, S(b_{[2]})b_{[4]}\big)\, a_{[1]} \otimes
\rho(a_{[2]},S(b_{[1]}))_{[2]} S(b_{[3]}) \\
&=  -
 S\big(\rho(a_{[2]},b_{[1]})_{[1]}\big)\, a_{[1]} \otimes
\rho(a_{[2]},S(b_{[1]}))_{[2]} S(\epsilon(b_{[2]})b_{[3]}) \\
&=  -
S\big(\rho(a_{[2]},b_{[1]})_{[1]}\big)\, a_{[1]} \otimes
\rho(a_{[2]},S(b_{[1]}))_{[2]} S(b_{[2]})\,.
\end{align*}
Here we used \eqref{Eq:dbr-Fox} for the first equality,
$\Delta(S(b))=S(b_{[2]})\otimes S(b_{[1]})$ in the second equality,
\eqref{Eq:pfAntip0} in the third equality,
and the cocommutativity in the fourth equality.
\end{proof}

Given $a,b\in \cH$, one obtains from \cite[\S~8.3]{MT18} that an antisymmetric (ungraded) Fox pairing induces an antisymmetric biderivation (with the factor explained in Footnote~\ref{Ft:ConvMT}) satisfying
\begin{equation}  \label{Eq:pfAntip5}
\begin{aligned}
  \br{a_{ij} , b_{kl}} =& \frac12(b_{[1]} S(\rho(a_{[2]},b_{[2]})_{[1]})\, a_{[1]} )_{kj} \, (\rho(a_{[2]},b_{[2]})_{[2]})_{il} \\
 &- \frac12( S(\rho(a_{[2]},b_{[1]})_{[1]}) a_{[1]} )_{lj} \, (\rho(a_{[2]},b_{[1]})_{[2]} \, S(b_{[2]}))_{ik}\,.
\end{aligned}
\end{equation}
Meanwhile, our formula \eqref{Eq:BrPO} entails
\begin{equation}
\label{Eq:pfAntip6}
 \br{a_{ij} , b_{kl}} = \frac12(\dgal{a,b}_\rho)_{kj,il} + \frac12(\dgal{S(a),b}_\rho)_{ki,jl}
 = \frac12(\dgal{a,b}_\rho)_{kj,il} + \frac12(\dgal{a,S(b)}_\rho)_{lj,ik}
\end{equation}
after combining the relation \eqref{Eq:DefRelO} and the
$S$-compatibility of  $\dgal{-,-}_\rho$ to rewrite the second term.
In view of \eqref{Eq:dbr-Fox} and \eqref{Eq:Antip4}, this reproduces \eqref{Eq:pfAntip5}.

\begin{rem}
Adapting \eqref{Eq:pfAntip6} using \eqref{Eq:DefRelSp}--\eqref{Eq:BrPsymp} in the symplectic case ($N=2n$),
we obtain:
 \begin{equation}  \label{Eq:pfAntip7}
\begin{aligned}
  \br{a_{ij} , b_{kl}} =&
  \frac12(b_{[1]} S(\rho(a_{[2]},b_{[2]})_{[1]})\, a_{[1]} )_{kj} \, (\rho(a_{[2]},b_{[2]})_{[2]})_{il} \\
 &-\frac12 \sgn_{2n}(k,l)\, ( S(\rho(a_{[2]},b_{[1]})_{[1]}) a_{[1]} )_{\tau(l)j} \,
 (\rho(a_{[2]},b_{[1]})_{[2]} \,S(b_{[2]}))_{i\tau(k)}\,.
\end{aligned}
\end{equation}
Such a formula can also be derived\footnote{We thank G. Massuyeau for explaining this result to us.} by adapting \cite[\S8.3]{MT18} to the symplectic case.
\end{rem}

We should mention that Massuyeau and Turaev induced an operation on a representation algebra, denoted $\cH_B$,
which is defined in a slightly different way than $\CC[\Rep^{\phi,\tau}(\cH,N)]$.
We show in Proposition~\ref{Pr:RepHB} that these algebras can be identified.


\section{Modified Kontsevich system} \label{Sec:Ksys}

On the algebra of noncommutative Laurent polynomials $A:=\CC\langle u^{\pm 1},v^{\pm 1}\rangle$, consider the Kontsevich system
\begin{equation} \label{Eq:Konts}
 \frac{\dd u}{\dd t}=uv-uv^{-1}-v^{-1}, \quad
 \frac{\dd v}{\dd t}=vu^{-1}-vu+u^{-1}.
\end{equation}
Its integrability was studied by Wolf and Efimovskaya \cite{EW}.
Following \cite{Art15,Art17}, an Hamiltonian structure is obtained from the double quasi-Poisson bracket
of Example \ref{Ex:g1r0} written as ($v=x$ and $u=y$)
\begin{equation} \label{Eq:q-uv}
\begin{aligned}
 \dgal{v,v}&= \frac12(v^2\otimes 1 - 1 \otimes v^2), \quad
 \dgal{u,u}= -\frac12(u^2\otimes 1 - 1 \otimes u^2), \\
 \dgal{v,u}&= \frac12(uv\otimes 1 + 1\otimes vu - v \otimes u + u\otimes v),
\end{aligned}
\end{equation}
together with the Hamiltonian function $h:=H+w$, where
\begin{equation} \label{Eq:H-w}
 H=u+v+u^{-1}+v^{-1}, \qquad
 w= u^{-1}v^{-1}.
\end{equation}
Indeed, a tedious but straightforward computation yields
that the elements $(h^k)_{k\geq 1}$ pairwise commute under the Lie bracket $\br{-,-}_\sharp$ defined as in Remark
\ref{Rem:HoP}. (See \cite{MS} for more on integrable non-abelian ODEs.)
Going to the $N$-th representation space $\Rep(A,N)$, this induces that \eqref{Eq:Konts} seen as an ODE on $\Gl_N(\CC)\times \Gl_N(\CC)$ is Hamiltonian for $\tr(h)$ and, furthermore, the elements $(\tr(h^m))_{m\geq 1}$ are conserved quantities defining pairwise commuting flows.

The previous considerations can not be directly applied to twisted representation spaces, where one takes $\phi:A\to A$ to be the inverse mapping (when evaluated on monomials) as in Example \ref{Ex:g1r0}. Note that $\phi(H)=H$ but $\phi(w)=w^{-1}=vu$.
Given $a\in A$, the identity \eqref{Eq:BrPO-trLeft} implies that the flow of $\tr(a)$ on $\Rep^{\phi,\tau}(A,N)$ is given in types $\typB,\typD$ by (for any $b\in A$ and indices $k,l$)
\[
\br{\tr(h),b_{kl}}_{\phi,\Orm}=\frac12 (\mathrm{m} \circ (\dgal{h,b}+\dgal{\phi(h),b}))_{kl} .
\]
The same formula holds for $\br{-,-}_{\phi,\Sp}$.
Commutativity under $\br{-,-}_\sharp$ of the classes in $A/[A,A]$ of the elements $(h^m)_{m\geq 1}$ no longer guarantees that each $\tr(h^m) \in \CC[\Rep^{\phi,\tau}(A,N)]$ is a first integral of that new flow:
one now needs $\br{h,\phi(h)^m}_\sharp=0$ as well, and this equality does not always hold.
To rectify this issue, fix $\lambda\in \CC$ and introduce the $\phi$-invariant element (we use $H,w$ as in \eqref{Eq:H-w})
\begin{equation} \label{Eq:hphi}
 h_\lambda = \frac12 (\id_A+\phi)(H+2\lambda w) = H + \lambda (w + w^{-1})\,.
\end{equation}

\begin{thm} \label{Thm:mKonts}
For any fixed $\lambda\in \CC$, the (Hamiltonian) element $h_\lambda$  \eqref{Eq:hphi} defines the system
\begin{equation} \label{Eq:mKonts}
 \frac{\dd u}{\dd t}=uv-uv^{-1}-\lambda v^{-1}+\lambda uvu, \quad
 \frac{\dd v}{\dd t}=vu^{-1}-vu+\lambda u^{-1}-\lambda vuv,
\end{equation}
which satisfies:
\begin{enumerate}[label=\roman*)]
 \item $vuv^{-1}u^{-1}\in A$ is a first integral; \label{mK-item1}
 \item the class in $A/[A,A]$ of any $\phi$-invariant element $h_\lambda^m$ is a first integral,
and these pairwise commute under $\br{-,-}_\sharp$; \label{mK-item2}
\item The previous statements are also valid on the representation space $\Rep(A,N)$, $N\geq 1$, and on the twisted representation space $\Rep^{\phi,\tau}(A,N)$. That is, $\X(vuv^{-1}u^{-1})$ is a (matrix) first integral,
and the elements $(\tr(h_\lambda^m))_{m\geq 1}$ are first integrals that pairwise commute under the induced Poisson bracket on the (twisted) representation space.
\label{mK-item3}
\end{enumerate}
\end{thm}

\begin{proof}[Proof of Theorem \ref{Thm:mKonts}]
Let us first derive \eqref{Eq:mKonts}.
By gathering \eqref{Eq:q-uv} and \eqref{Eq:H-w}, we compute
\begin{subequations} \label{Eq:HW-u}
 \begin{align}
\dgal{H,u}&=\frac12 \left( u\otimes H - H \otimes u - uH\otimes 1 + 1\otimes Hu \right) + uv\otimes 1- u\otimes v^{-1}; \\
\dgal{w,u}&=\frac12\left(-w\otimes u - u\otimes w - uw\otimes 1 + 1 \otimes wu\right).
 \end{align}
\end{subequations}
We then get $\dgal{w^{-1},u}=-w^{-1}\ast \dgal{w,u}\ast w^{-1}$, and we obtain the first equality in \eqref{Eq:mKonts} by expanding
$\frac{\dd u}{\dd t}:=\mathrm{m}\circ \dgal{h_\lambda,u}$ linearly using \eqref{Eq:hphi}.
Similarly, we find $\frac{\dd v}{\dd t}:=\mathrm{m}\circ \dgal{h_\lambda,v}$ after deriving
\begin{subequations} \label{Eq:HW-v}
 \begin{align}
\dgal{H,v}&=\frac12 \left( -v\otimes H + H \otimes v + vH\otimes 1 - 1\otimes Hv \right) - vu\otimes 1 +  v\otimes u^{-1}; \\
\dgal{w,v}&=\frac12\left(w\otimes v  - v\otimes w + vw\otimes 1 + 1 \otimes wv\right).
 \end{align}
\end{subequations}

We can now deduce the different items.
For \ref{mK-item1}, since $\Phi=vuv^{-1}u^{-1}$ is a moment map, \eqref{Phim} entails
$\mathrm{m}\circ \dgal{a,\Phi}=0$ for any $a\in A$.
(Thus, $\Phi^{-1}=uvu^{-1}v^{-1}$ is a first integral as in the original case \cite{EW}.)

For \ref{mK-item2}, the result follows from Lemma \ref{Lem:mKonts} below.

For \ref{mK-item3}, let us only consider the twisted orthogonal case, as the symplectic case and the ``untwisted'' type $\typA$ cases are completely analogous.
By \eqref{Eq:BrPO-trLeft} and the $\phi$-invariance of $h_\lambda^m$, one has for any $b\in A$,
\[
 \br{\tr(h_\lambda^m),b_{kl}}_{\phi,\Orm}=(\mathrm{m} \circ \dgal{h_\lambda^m,b})_{kl} .
\]
By taking $m=1$ we obtain $\dd\X(b)/\dd t=\X(\dd b/\dd t)$, so that we can write \eqref{Eq:mKonts} for $\X(u),\X(v)$ instead of $u,v$.
We get similarly $\br{\tr(h_\lambda^m),\X(\Phi)}_{\phi,\Orm}=0_N$ by reproducing the argument used for \ref{mK-item1}.
Then, one has for $m,n\geq 0$,
\[
 \br{\tr(h_\lambda^m),\tr(h_\lambda^n)}_{\phi,\Orm}=\tr(\br{h_\lambda^m,h_\lambda^n}_\sharp)
\]
which vanishes by \ref{mK-item2}.
\end{proof}

\begin{lem} \label{Lem:mKonts}
The classes in $(A/[A,A],\br{-,-}_\sharp)$ of the $\phi$-invariant elements $(h_\lambda^k)_{k\geq 0}$ span an abelian Lie subalgebra.
\end{lem}
\begin{proof}
It suffices to check the vanishing for any $k,\ell\in \Z_{>0}$ of the following element in $A/[A,A]$:
\begin{equation} \label{Eq:hphi-hphi}
 \frac{1}{k\ell}\br{h_\lambda^k,h_\lambda^\ell}_\sharp =
 \dgal{h_\lambda,h_\lambda}' h_\lambda^{k-1} \dgal{h_\lambda,h_\lambda}'' h_\lambda^{\ell-1}\,.
\end{equation}
(This equality follows from \eqref{Eq:outder}-\eqref{Eq:inder} modulo commutators.)
For $H,w$ as in \eqref{Eq:H-w}, a routine computation based on \eqref{Eq:HW-u} and \eqref{Eq:HW-v} yields
\begin{align*}
\dgal{H,H}=&\,(uv-vu)\otimes 1 + 1 \otimes (v^{-1}u^{-1}-u^{-1}v^{-1}) + \frac12 \Big(1\otimes H^2 - H^2\otimes 1\Big) \\
&+ (v+u^{-1})H \otimes 1 - (v+u^{-1})\otimes H + H \otimes (v+u^{-1}) - 1\otimes H(v+u^{-1})\, ; \\
\dgal{H,w}=&\,\frac12 \Big(H\otimes w + w\otimes H - wH\otimes 1 - 1 \otimes Hw\Big) 
+1\otimes (u+  v^{-1})w - (v+ u^{-1})\otimes w\, ; \\
\dgal{w,w}=&\, \frac12 \Big(1\otimes w^2 - w^2\otimes 1\Big)\,.
\end{align*}
We also deduce from the last two equalities their counterparts for $w^{-1}$ in place of $w$ since
\[
 \dgal{w^{-1},a}=-w^{-1}\ast \dgal{w,a}\ast w^{-1}, \quad \dgal{a,w^{-1}}=-w^{-1} \dgal{a,w} w^{-1},
\]
for any $a\in A$. We can then write explicitly
\[
 \dgal{h_\lambda,h_\lambda} = \dgal{H+\lambda w+\lambda w^{-1},H+\lambda w+\lambda w^{-1}}\,.
\]
After obvious cancellations and using that pairs of terms of the form $b\otimes 1 - 1\otimes b$ ($b\in A$) cancel out when substituted in \eqref{Eq:hphi-hphi}, we obtain in $A/[A,A]$
\begin{align*}
\frac{1}{k\ell}\br{h_\lambda^k,h_\lambda^\ell}_\sharp =&
(uv-vu) h_\lambda^{k+\ell-2}+ (v^{-1}u^{-1}-u^{-1}v^{-1}) h_\lambda^{k+\ell-2} \\
&+\big((v+u^{-1})H-H(v+u^{-1})\big) h_\lambda^{k+\ell-2} \\
&+H h_\lambda^{k-1} (v+u^{-1}) h_\lambda^{\ell-1}
-(v+u^{-1}) h_\lambda^{k-1} H h_\lambda^{\ell-1} \\
&+\lambda\, w h_\lambda^{k-1} (v+u^{-1}) h_\lambda^{\ell-1}
-\lambda\, (v+u^{-1}) h_\lambda^{k-1} w h_\lambda^{\ell-1} \\
&+\lambda\, (u+v^{-1}) h_\lambda^{k-1} w^{-1} h_\lambda^{\ell-1}
-\lambda\, w^{-1} h_\lambda^{k-1} (u+v^{-1}) h_\lambda^{\ell-1} \\
&+\lambda^2\, w h_\lambda^{k-1} w^{-1} h_\lambda^{\ell-1}
-\lambda^2\, w^{-1} h_\lambda^{k-1} w h_\lambda^{\ell-1}
\end{align*}
The first two lines disappear because $[v+u^{-1},H]=-uv+vu-v^{-1}u^{-1}+u^{-1}v^{-1}$ thanks to \eqref{Eq:H-w}.
For the remaining terms, we use that $H+\lambda w=h_\lambda-\lambda w^{-1}$ by \eqref{Eq:hphi} and the definition of $H$ \eqref{Eq:H-w} to write modulo commutators,
\begin{align*}
\frac{1}{k\ell}\br{h_\lambda^k,h_\lambda^\ell}_\sharp =&
\,\left(H+\lambda w \right) h_\lambda^{k-1} (v+u^{-1}) h_\lambda^{\ell-1}
-(v+u^{-1}) h_\lambda^{k-1} \left(H+\lambda w \right) h_\lambda^{\ell-1} \\
&+\lambda\left(u+v^{-1}+\lambda w \right) h_\lambda^{k-1} w^{-1} h_\lambda^{\ell-1}
-\lambda w^{-1} h_\lambda^{k-1} \left(u+v^{-1}+\lambda w \right) h_\lambda^{\ell-1} \\
=&\,-\lambda w^{-1} h_\lambda^{k-1} (v+u^{-1}) h_\lambda^{\ell-1}
+\lambda (v+u^{-1}) h_\lambda^{k-1}w^{-1} h_\lambda^{\ell-1} \\
&+\lambda\left(u+v^{-1}+\lambda w \right) h_\lambda^{k-1} w^{-1} h_\lambda^{\ell-1}
-\lambda w^{-1} h_\lambda^{k-1} \left(u+v^{-1}+\lambda w \right) h_\lambda^{\ell-1} \\
=&-\lambda w^{-1} h_\lambda^{k-1} \left(H+\lambda w \right) h_\lambda^{\ell-1}
+\lambda\left(H+\lambda w \right) h_\lambda^{k-1} w^{-1} h_\lambda^{\ell-1} = 0\,.
\end{align*}
This proves the vanishing of \eqref{Eq:hphi-hphi}.
\end{proof}

\begin{rem} \label{Rem:IS}
The last item in Theorem \ref{Thm:mKonts} suggests that \eqref{Eq:mKonts} defines an integrable system when $u,v$ are both invertible, or both orthogonal, or both symplectic matrices.
Indeed, $\Rep(A,N)$ is the internally fused double of $\Gl_N$,
while $\Rep^{\phi,\tau}(A,N)$ is the one of $\Orm_N$ or $\Sp_N$ (due to Theorem \ref{Thm:qP-Double}).
In the $\Gl_N$-case, integrability can be established for generic $\lambda$ after quasi-Poisson reduction
to the $q$-Calogero-Moser space~\cite{Obl} (where $X,Y$ represent $v,u$) for $q\in \CC^\times$ not a root of unity,
\[
 \mathcal{C}_{N,q}:=\{(X,Y,V,W)\in \Gl_N^{\times 2} \times \CC^N \times (\CC^N)^\ast \mid
 XYX^{-1}Y^{-1}=q\Id_N+VW\}/\Gl_N\,.
\]
Note that $\mathcal{C}_{N,q}$ is an irreducible smooth variety of dimension $2N$, such that the matrix $X$ (or $Y$) is generically diagonalizable with pairiwse distinct eigenvalues.
Thus, $\tr(X+X^{-1}),\ldots,\tr((X+X^{-1})^N)$ are $N$ functionally idependent elements of $\CC[\mathcal{C}_{N,q}]$.
Then, the Poisson commuting elements $\{\tr(h_\lambda^m)\mid m\geq 1\}$ obtained from Theorem \ref{Thm:mKonts} give
the abelian Poisson subalgebra
\[
P_\lambda:=\CC[\tr\big(X+Y+X^{-1}+Y^{-1}+\lambda(XY+Y^{-1}X^{-1})\big)^m \mid m\geq 1]\subset \CC[\mathcal{C}_{N,q}]\,.
\]
after reduction. Write $P_\infty:=\CC[\tr\big(XY+(XY)^{-1}\big)^m \mid m\geq 1]$.
The automorphism $(X,Y,V,W)\mapsto (XY,X^{-1},V,W)$ of $\mathcal{C}_{N,q}$ and the previous discussion entail that $P_\infty$ has functional dimension $N$, hence it defines a Liouville integrable system on $\mathcal{C}_{N,q}$.
Thus, the same holds for $P_\lambda$ with generic $\lambda$.

More generally, performing quasi-Poisson reduction at a generic conjugacy class leads to a phase space of dimension strictly greater than $2N$. There, the abelian Poisson algebra $P_\lambda$ is no longer sufficiently large to define a Liouville integrable system and integrability in such cases needs to be investigated.
\end{rem}

\begin{rem} \label{Rem:Schedl}
It was suggested by T. Schedler that one can introduce a parameter in the last term of both equations in the original Kontsevich system \eqref{Eq:Konts}.
This corresponds to considering the (non-$\phi$-invariant) Hamiltonian $\tilde{h}_\lambda:=H+ \lambda w$.
The fact that the classes in $(A/[A,A],\br{-,-}_\sharp)$ of the elements $(\tilde{h}_\lambda^k)_{k\geq 0}$ span an abelian Lie subalgebra can be obtained similarly to the proof of Lemma \ref{Lem:mKonts}.
Then, we get a family of Poisson commuting functions on the internally fused double of $\Gl_N$ given by $\Rep(A,N)$ (this is \emph{not} true in the twisted case).
Arguing as in Remark \ref{Rem:IS}, one obtains that the subalgebra
\[
\tilde{P}_\lambda:=\CC[\tr\big(X+Y+X^{-1}+Y^{-1}+\lambda XY \big)^m \mid m\geq 1]\subset \CC[\mathcal{C}_{N,q}]\,,
\]
defines a Liouville integrable system on $\mathcal{C}_{N,q}$
for generic $\lambda$.
\end{rem}

\newpage


\appendix

\section{Some detailed calculations} \label{App:A}

\subsection{Cartan trivector in the symplectic case} \label{App:CartSP}

Recall that an element $M\in \spg(2n)$ has the form
\begin{equation*}
 M=\left(\begin{array}{cc} a&b\\c&-a^T \end{array}  \right), \quad
 a,b,c\in \gl(n),\,\, b=b^T,\,\,c=c^T\,.
\end{equation*}
A basis for $\spg(2n)$ is made of the $2n^2+n$ elements
\begin{equation}  \label{Eq:AppSP1}
\begin{aligned}
  F_{ij}^{(1)}&=E_{ij} - E_{n+j,n+i}, \qquad 1\leq i,j\leq n; \\
  F_{ij}^{(2)}&=E_{i,n+j} + E_{j,n+i}, \qquad 1\leq i\leq j\leq n; \\
  F_{ij}^{(3)}&=E_{n+i,j} + E_{n+j,i}, \qquad 1\leq i\leq j\leq n.
\end{aligned}
\end{equation}
Let us record the following identities relating \eqref{Eq:SpanSP} to that basis for $1\leq i,j\leq n$:
\begin{equation} \label{Eq:AppSPD1}
 F_{ij}=-F_{j+n,i+n}=F_{ij}^{(1)}, \quad
 F_{i,j+n}=F_{j,i+n}=F_{ij}^{(2)}, \quad
 F_{i+n,j}=F_{j+n,i}=F_{ij}^{(3)}.
\end{equation}
The dual basis of \eqref{Eq:AppSP1} under the trace pairing (i.e. $\langle \check{F}_{ij}^{(a)},  F_{kl}^{(b)}\rangle_{\spg(n)}=\delta_{ik}\delta_{jl}\delta_{ab}$) is given as follows:
\begin{equation}
\begin{aligned}  \label{Eq:AppSP2}
  \check{F}_{ij}^{(1)}&=\frac12 F_{ji}^{(1)}, \qquad 1\leq i,j\leq n; \\
  \check{F}_{ij}^{(2)}&=\frac12 F_{ij}^{(3)}, \quad 1\leq i< j\leq n;  \quad
  \check{F}_{ii}^{(2)}&=\frac14 F_{ii}^{(3)}, \quad 1\leq i\leq n;\\
  \check{F}_{ij}^{(3)}&=\frac12 F_{ij}^{(2)}, \quad 1\leq i< j\leq n;  \quad
  \check{F}_{ii}^{(3)}&=\frac14 F_{ii}^{(2)}, \quad 1\leq i\leq n.
\end{aligned}
\end{equation}
This trivially implies $[\check{F}_{ij}^{(2)},  \check{F}_{kl}^{(2)}]=0$, $[\check{F}_{ij}^{(3)},  \check{F}_{kl}^{(3)}]=0$.
We also compute\footnote{We express the Lie brackets of the dual basis $(\check{F}_{ij}^{(a)})$ in terms of the original basis $(F_{ij}^{(a)})$ so as to easily take the pairing with elements of the dual basis, cf. \eqref{Eq:Cartan3}.}
\begin{equation}  \label{Eq:AppSP3}
 [\check{F}_{kl}^{(1)},\check{F}_{uv}^{(1)}] =
 \frac14 \delta_{vk} F_{lu}^{(1)} - \frac14 \delta_{lu} F_{vk}^{(1)}, \qquad 1\leq k,l,u,v\leq n\,,
\end{equation}
and
\small
\begin{subequations}  \label{Eq:AppSP4}
 \begin{align}
[\check{F}_{kl}^{(2)},\check{F}_{uv}^{(3)}] &=
- \frac14 \delta_{lu} F_{vk}^{(1)} - \frac14 \delta_{ku} F_{vl}^{(1)}
- \frac14 \delta_{lv} F_{uk}^{(1)} - \frac14 \delta_{kv} F_{ul}^{(1)},
 && 1\leq k<l\leq n,\,1\leq u<v\leq n, \\
 [\check{F}_{kl}^{(2)},\check{F}_{uu}^{(3)}] &=
- \frac14 \delta_{lu} F_{uk}^{(1)} - \frac14 \delta_{ku} F_{ul}^{(1)},
 && 1\leq k<l\leq n,\,1\leq u\leq n, \\
 [\check{F}_{kk}^{(2)},\check{F}_{uv}^{(3)}] &=
- \frac14 \delta_{ku} F_{vk}^{(1)} - \frac14 \delta_{kv} F_{uk}^{(1)},
 && 1\leq k\leq n,\,1\leq u<v\leq n, \\
 [\check{F}_{kk}^{(2)},\check{F}_{uu}^{(3)}] &=
- \frac14 \delta_{ku} F_{uk}^{(1)} , && 1\leq k,u\leq n\,.
 \end{align}
\end{subequations}
\normalsize
We can deduce all remaining Lie brackets using the ad-invariance of the pairing. One can then write using \eqref{Eq:Cartan3}
(all indices are in $\{1,\ldots,n\}$)
\begin{equation}   \label{Eq:AppSP8}
 \begin{aligned}
 \psi^{\spg(2n)} &=\frac{1}{12} \sum_{i,j,k,l,u,v}
\langle \check{F}_{ij}^{(1)},  [\check{F}_{kl}^{(1)},\check{F}_{uv}^{(1)}]\rangle_{\spg(n)} \,
F_{ij}^{(1)}\wedge F_{kl}^{(1)} \wedge F_{uv}^{(1)} \\
 &\quad +\frac{1}{2} \sum_{i\leq j}\sum_{k\leq l}\sum_{u\leq v}
\langle \check{F}_{ij}^{(1)},  [\check{F}_{kl}^{(2)},\check{F}_{uv}^{(3)}]\rangle_{\spg(n)} \,
F_{ij}^{(1)}\wedge F_{kl}^{(2)} \wedge F_{uv}^{(3)}  \,.
 \end{aligned}
\end{equation}

\begin{prop}
The Cartan trivector $\psi^{\spg(2n)}$ can be written as \eqref{Eq:CartSp} with the spanning set \eqref{Eq:SpanSP}.
\end{prop}
\begin{proof}
 We start by developing $\psi^{\spg(2n)}$ \eqref{Eq:AppSP8}.
We denote by $\psi_1$ the first line appearing in \eqref{Eq:AppSP8}, and by
$\psi_2$ (resp. $\psi_3$; $\psi_4$; or $\psi_5$) the second line appearing in \eqref{Eq:AppSP8}
for $k<l$ and $u<v$ (resp. $k<l$ and $u=v$; $k=l$ and $u<v$; $k=l$ and $u=v$).
 On the one hand, \eqref{Eq:AppSP3} yields
\begin{equation} \label{Eq:AppSPC1}
 \psi_1=\frac{1}{24}\sum_{i,j,k=1}^n
\,F_{ij}^{(1)}\wedge F_{ki}^{(1)} \wedge F_{jk}^{(1)}\,.
\end{equation}
On the other hand, \eqref{Eq:AppSP4} yields
\begin{equation}
 \begin{aligned} \label{Eq:AppSPC2}
\psi_2=& -\frac{1}{8} \sum_{j<k<i} F_{ij}^{(1)}\wedge F_{jk}^{(2)} \wedge F_{ki}^{(3)}
-\frac{1}{8} \sum_{i,j=1}^n\sum_{k<\min(i,j)} F_{ij}^{(1)}\wedge F_{kj}^{(2)} \wedge F_{ki}^{(3)} \\
& -\frac{1}{8} \sum_{i,j=1}^n\sum_{k>\max(i,j)} F_{ij}^{(1)}\wedge F_{jk}^{(2)} \wedge F_{ik}^{(3)}
-\frac{1}{8} \sum_{i<k<j} F_{ij}^{(1)}\wedge F_{kj}^{(2)} \wedge F_{ik}^{(3)}\,;
\end{aligned}
\end{equation}
then
\begin{align}
 \psi_3&= -\frac{1}{8} \sum_{j<i} F_{ij}^{(1)}\wedge F_{ji}^{(2)} \wedge F_{ii}^{(3)}
 -\frac{1}{8} \sum_{i<j} F_{ij}^{(1)}\wedge F_{ij}^{(2)} \wedge F_{ii}^{(3)}   \label{Eq:AppSPC3} \\
 \psi_4&= -\frac{1}{8} \sum_{j<i} F_{ij}^{(1)}\wedge F_{jj}^{(2)} \wedge F_{ji}^{(3)}
 -\frac{1}{8} \sum_{i<j} F_{ij}^{(1)}\wedge F_{jj}^{(2)} \wedge F_{ij}^{(3)}  \label{Eq:AppSPC4}\\
 \psi_5&= -\frac{1}{8} \sum_{i=1}^n F_{ii}^{(1)}\wedge F_{ii}^{(2)} \wedge F_{ii}^{(3)} \,. \label{Eq:AppSPC5}
\end{align}
By construction, $\psi^{\spg(2n)}=\psi_1+\psi_2+\psi_3+\psi_4+\psi_5$.
Using the spanning set \eqref{Eq:SpanSP}, let us now develop the RHS of \eqref{Eq:CartSp}, which we label $\phi$.
After introducing
\begin{equation*}
\phi_{ijk}:=F_{ij} \otimes F_{ki} \otimes F_{jk} - F_{ij} \otimes F_{jk} \otimes F_{ki}\,,
\end{equation*}
we can consider the splitting $\phi=\phi_\circ +\phi_{-}+\phi_+$ where
\small
\begin{align*}
 \phi_\circ&=\frac{1}{16}\left(\sum_{1\leq i,j,k\leq n} + \sum_{n+1\leq i,j,k\leq 2n} \right)  \phi_{ijk}, \\
 \phi_-&=\frac{1}{16}\left(\sum_{1\leq i\leq n}\sum_{n+1\leq j,k\leq 2n} + \sum_{1\leq j\leq n}\sum_{n+1\leq i,k\leq 2n} + \sum_{1\leq k\leq n}\sum_{n+1\leq i,j\leq 2n} \right)  \phi_{ijk}, \\
 \phi_+&=\frac{1}{16}\left(\sum_{n+1\leq i\leq 2n} \sum_{1\leq j,k\leq n}+ \sum_{n+1\leq j\leq 2n} \sum_{1\leq i,k\leq n} + \sum_{n+1\leq k\leq 2n} \sum_{1\leq i,j \leq n}\right)  \phi_{ijk}.
\end{align*}
\normalsize
We can write using \eqref{Eq:AppSPD1}
\begin{align*}
\phi_\circ=&\frac{1}{16}\sum_{1\leq i,j,k\leq n} \,
(F_{ij} \otimes F_{ki} \otimes F_{jk} - F_{ij} \otimes F_{jk} \otimes F_{ki}
- F_{ji} \otimes F_{ik} \otimes F_{kj} + F_{ji} \otimes F_{kj} \otimes F_{ik}) \\
=&\frac{1}{24}\sum_{i,j,k=1}^n
\,F_{ij}\wedge F_{ki}\wedge F_{jk} = \psi_1\,; \\
\phi_-=&\frac{1}{16}\sum_{1\leq i, j,k\leq n}
\big(
F_{i,j+n} \otimes F_{k+n,i} \otimes F_{j+n,k+n} - F_{i,j+n} \otimes F_{j+n,k+n} \otimes F_{k+n,i} \\
&\quad + F_{i+n,j} \otimes F_{k+n,i+n} \otimes F_{j,k+n} - F_{i+n,j} \otimes F_{j,k+n} \otimes F_{k+n,i+n} \\
&\quad + F_{i+n,j+n} \otimes F_{k,i+n} \otimes F_{j+n,k} - F_{i+n,j+n} \otimes F_{j+n,k} \otimes F_{k,i+n}
\big) \\
=&\frac{1}{16}\sum_{1\leq i, j,k\leq n}
F_{ij}\wedge F_{k+n,i}\wedge F_{j,k+n}\,.
\end{align*}
Analogous manipulations for $\phi_+$ give that expression, hence $\phi_+=\phi_-$.
We obtain in terms of \eqref{Eq:AppSP1},
\small
\begin{align*}
 &\phi_-+\phi_+=\frac18\sum_{i, j,k}
F_{ij}^{(1)}\wedge (\delta_{(k< i)} F^{(3)}_{ki}+\delta_{ik} F^{(3)}_{ii}+\delta_{(i<k)} F^{(3)}_{ik})\wedge (\delta_{(j< k)} F^{(2)}_{jk}+ \delta_{jk} F^{(2)}_{jj} +\delta_{(k<j)} F^{(2)}_{kj}) \\
&=\frac18\sum_{j<k<i} F_{ij}^{(1)}\wedge F^{(3)}_{ki} \wedge F^{(2)}_{jk}
+ \frac18\sum_{j<i} F_{ij}^{(1)}\wedge F^{(3)}_{ji} \wedge F^{(2)}_{jj}
+\frac18\sum_{i,j}\sum_{k<\min(i,j)} F_{ij}^{(1)}\wedge F^{(3)}_{ki} \wedge F^{(2)}_{kj} \\ 
&\quad +\frac18\sum_{j<i} F_{ij}^{(1)}\wedge F^{(3)}_{ii} \wedge F^{(2)}_{ji}
+\frac18\sum_{i} F_{ii}^{(1)}\wedge F^{(3)}_{ii} \wedge F^{(2)}_{ii}
+\frac18\sum_{i<j} F_{ij}^{(1)}\wedge F^{(3)}_{ii} \wedge F^{(2)}_{ij}  \\ 
&\quad +\frac18\sum_{i,j}\sum_{k>\max(i,j)} F_{ij}^{(1)}\wedge F^{(3)}_{ik} \wedge F^{(2)}_{jk}
+\frac18\sum_{i<j} F_{ij}^{(1)}\wedge F^{(3)}_{ij} \wedge F^{(2)}_{jj}
+\frac18\sum_{i<k<j} F_{ij}^{(1)}\wedge F^{(3)}_{ik} \wedge F^{(2)}_{kj}\,. 
\end{align*}
\normalsize
These terms are, in order of occurrence :
first term of \eqref{Eq:AppSPC2}, first term of \eqref{Eq:AppSPC4}, second term of \eqref{Eq:AppSPC2},  
first term of \eqref{Eq:AppSPC3}, \eqref{Eq:AppSPC5}, second term of \eqref{Eq:AppSPC3},  
third term of \eqref{Eq:AppSPC2}, second term of \eqref{Eq:AppSPC4}, and fourth term of \eqref{Eq:AppSPC2}.  
All in all, we obtained
$\phi_{-}+\phi_+=\psi_2+\psi_3+\psi_4+\psi_5$. Summing this equality with $\phi_\circ =\psi_1$
yields the claimed equality $\phi = \psi^{\spg(2n)}$.
\end{proof}

\subsection{Fusion and bracket-compatibility} \label{App:Fus}

We gather all the missing cases from the proof of Proposition \ref{Pr:CompFus} for the sake of conciseness.
We use the notation from \ref{ss:Fus}.
We want to check
\begin{equation} \label{Eq:FusCompBis}
 \phi_f^{\otimes 2}(\dgal{a,b})=\dgal{\phi_f(a),\phi_f(b)}^\circ, \qquad
 \phi_f^{\otimes 2}(\dgal{a,b}_{\fus}) = \dgal{\phi_f(a),\phi_f(b)}_{\fus}^\circ,
\end{equation}
for all kinds of generators $a,b\in A^f$. Recall that the extended anti-involution acts as \eqref{Eq:phiExt}.
In each case, we indicate the equations of \cite[Lem.~2.19]{F2} used to compute the double bracket $\dgal{-,-}_{\fus}$.

\medskip

\underline{$a=\epsilon t\epsilon$, $b=\epsilon \tilde{t}\epsilon$ of first kind.}
We use \cite[(2.14a)]{F2}:
\begin{align*}
\phi_f^{\otimes 2}(\dgal{t,\tilde{t}})
&= \phi^{\otimes 2}(\dgal{t,\tilde{t}})
=(\dgal{\phi(t),\phi(\tilde{t})})^\circ
= \dgal{\phi_f(t),\phi_f(\tilde{t})}^\circ \,, \\
\phi_f^{\otimes 2}( \dgal{t , \tilde{t}}_{\fus}) &= 0 = \dgal{\phi_f( t) , \phi_f(\tilde{t})}_{\fus}^\circ\,.
\end{align*}

\underline{$a=\epsilon t\epsilon$ of first kind, $b=e_{12}u \epsilon$ of second kind.}
We use \cite[(2.14b)]{F2} and \cite[(2.14c)]{F2}:
\begin{align*}
\phi_f^{\otimes 2}(\dgal{t,e_{12}u})
&= (\phi^{\otimes 2}(\dgal{t,u}))\ast e_{21}
=(\dgal{\phi(t),\phi(u)} e_{21})^\circ
= \dgal{\phi_f(t),\phi_f(e_{12}u)}^\circ ;
\\
\phi_f^{\otimes 2}( \dgal{t , e_{12} u}_{\fus})
&=\frac12 \left( e_1 \otimes \phi(u) e_{21} \phi(t) - \phi(t) e_1 \otimes  \phi(u) e_{21}\right) \\
&=\frac12 \left(\phi(u) e_{21} \phi(t) \otimes e_1 - \phi(u) e_{21} \otimes \phi(t) e_1 \right)^\circ
=\dgal{\phi_f( t) , \phi_f(e_{12} u)}_{\fus}^\circ\,.
\end{align*}

\underline{$a=\epsilon t\epsilon$ of first kind, $b= \epsilon v e_{21}$ of third kind.}
We use \cite[(2.14b)]{F2} and \cite[(2.14c)]{F2}:
\begin{align*}
\phi_f^{\otimes 2}(\dgal{t,v e_{21}})
&= e_{12}\ast(\phi^{\otimes 2}(\dgal{t,v}))
=(e_{12}\dgal{\phi(t),\phi(v)})^\circ
= \dgal{\phi_f(t),\phi_f(v e_{21})}^\circ ;
\\
\phi_f^{\otimes 2}( \dgal{t , v e_{21}}_{\fus})
&=\frac12 \left( \phi(t) e_{12} \phi(v) \otimes e_1 - e_{12} \phi(v) \otimes  e_1 \phi(t) \right) \\
&=\frac12 \left(e_1 \otimes \phi(t) e_{12} \phi(v) - e_1 \phi(t) \otimes e_{12} \phi(v) \right)^\circ
=\dgal{\phi_f(t) , \phi_f(v e_{21})}_{\fus}^\circ
\end{align*}

\underline{$a=\epsilon t\epsilon$ of first kind, $b=e_{12}w e_{21}$ of fourth kind.} This was covered in the main proof.

\underline{$a=e_{12}u \epsilon$, $b=e_{12}\tilde u \epsilon$ of second kind.}
We use \cite[(2.15b)]{F2} and  \cite[(2.16c)]{F2}:
\begin{align*}
\phi_f^{\otimes 2}(\dgal{e_{12}u,e_{12}\tilde u})
&= (\phi^{\otimes 2}(\dgal{u,\tilde u}))e_{21}\ast e_{21}
=(\dgal{\phi(u),\phi(\tilde u)} e_{21}\ast e_{21})^\circ
= \dgal{\phi_f(e_{12}u),\phi_f(e_{12}\tilde u)}^\circ ;
\\
\phi_f^{\otimes 2}( \dgal{e_{12} u,e_{12}\tilde u}_{\fus})
&=\frac12 \left( e_1 \otimes \phi(\tilde u) e_{21} \phi(u)e_{21} - \phi(u) e_{21} \phi(\tilde u)e_{21} \otimes e_1\right) \\
&=\frac12 \left(\phi(\tilde u) e_{21} \phi(u)e_{21} \otimes e_1 - e_1 \otimes \phi(u) e_{21} \phi(\tilde u)e_{21} \right)^\circ
=\dgal{\phi_f(e_{12} u) , \phi_f(e_{12} \tilde u)}_{\fus}^\circ .
\end{align*}

\underline{$a=e_{12}u \epsilon$ of second kind, $b= \epsilon v e_{21}$ of third kind.}
We use \cite[(2.15c)]{F2} and  \cite[(2.16b)]{F2}:
\begin{align*}
\phi_f^{\otimes 2}(\dgal{e_{12}u,v e_{21}})
&= e_{12}\ast(\phi^{\otimes 2}(\dgal{u,v}))e_{21}
=(e_{12}\dgal{\phi(u),\phi(v)} \ast e_{21})^\circ
= \dgal{\phi_f(e_{12}u),\phi_f(v e_{21})}^\circ ;
\\
\phi_f^{\otimes 2}( \dgal{e_{12}u,v e_{21}}_{\fus})
&=\frac12 \left(  \phi(u)e_{21} \otimes e_{12}\phi(v) e_{1} - e_{12}\phi(v)  \otimes e_1 \phi(u)e_{21}\right) \\
&=\frac12 \left( e_{12}\phi(v) e_{1}  \otimes \phi(u)e_{21} - e_1\phi(u)e_{21} \otimes e_{12}\phi(v) \right)^\circ
=\dgal{\phi_f(e_{12} u) , \phi_f(v e_{21})}_{\fus}^\circ .
\end{align*}

\underline{$a=e_{12}u \epsilon$ of second kind, $b= e_{12} w e_{21}$ of fourth kind.}
We use \cite[(2.15d)]{F2} and  \cite[(2.16d)]{F2}:
\begin{align*}
&\phi_f^{\otimes 2}(\dgal{e_{12}u,e_{12} w e_{21}})
= e_{12}\ast (\phi^{\otimes 2}(\dgal{u,w}))e_{21}\ast e_{21} \\
&\qquad \qquad =(e_{12}\dgal{\phi(u),\phi(w)} e_{21} \ast e_{21})^\circ
= \dgal{\phi_f(e_{12}u),\phi_f(e_{12}w e_{21})}^\circ ;
\\
&\phi_f^{\otimes 2}( \dgal{e_{12}u,e_{12} w e_{21}}_{\fus})
=\frac12 \left(e_1 \otimes e_{12}  \phi(w)e_{21} \phi(u) e_{21} - e_{12}\phi(w)e_{21}  \otimes e_1 \phi(u)e_{21}\right) \\
&=\frac12 \left(  e_{12}  \phi(w)e_{21} \phi(u) e_{21}  \otimes e_{1} - e_1\phi(u)e_{21} \otimes e_{12}\phi(w)e_{21} \right)^\circ
=\dgal{\phi_f(e_{12} u) , \phi_f(e_{12} w e_{21})}_{\fus}^\circ .
\end{align*}

\underline{$a=\epsilon v e_{21}$, $b= \epsilon \tilde v e_{21}$ of third kind.}
We use \cite[(2.16c)]{F2} and  \cite[(2.15b)]{F2}:
\begin{align*}
\phi_f^{\otimes 2}(\dgal{v e_{21},\tilde v e_{21}})
&= e_{12}\ast e_{12}(\phi^{\otimes 2}(\dgal{v,\tilde v}))
=(e_{12}\ast e_{12} \dgal{\phi(v),\phi(\tilde v)})^\circ
= \dgal{\phi_f(v e_{21}),\phi_f(\tilde v e_{21})}^\circ ;
\\
\phi_f^{\otimes 2}( \dgal{v e_{21},\tilde v e_{21}}_{\fus})
&=\frac12 \left(  e_{12}\phi(v)e_{12}\phi(\tilde v) \otimes  e_{1} - e_{1} \otimes e_{12}\phi(\tilde v)e_{12}\phi(v)\right) \\
&=\frac12 \left(  e_{1}  \otimes e_{12}\phi(v)e_{12}\phi(\tilde v)  - e_{12}\phi(\tilde v)e_{12}\phi(v)\otimes e_1 \right)^\circ
=\dgal{\phi_f(v e_{21}) , \phi_f(\tilde v e_{21})}_{\fus}^\circ .
\end{align*}

\underline{$a=\epsilon v e_{21}$ of third kind, $b= e_{12} w e_{21}$ of fourth kind.}
We use \cite[(2.16d)]{F2} and  \cite[(2.15d)]{F2}:
\begin{align*}
&\phi_f^{\otimes 2}(\dgal{v e_{21},e_{12} w e_{21}})
= e_{12}\ast e_{12}(\phi^{\otimes 2}(\dgal{v,w}))\ast e_{21} \\
&\qquad \qquad =(e_{12}\ast e_{12} \dgal{\phi(v),\phi(w)} e_{21})^\circ
= \dgal{\phi_f(v e_{21}),\phi_f(e_{12}w e_{21})}^\circ ;
\\
&\phi_f^{\otimes 2}( \dgal{v e_{21},e_{12} w e_{21}}_{\fus})
=\frac12 \left(  e_{12}\phi(v)e_{12}\phi(w)e_{21} \otimes  e_{1} - e_{12}\phi(v) e_{1} \otimes e_{12}\phi(w)e_{21}\right) \\
&=\frac12 \left(  e_{1}  \otimes e_{12}\phi(v)e_{12}\phi(w)e_{21}  - e_{12}\phi(w)e_{21} \otimes e_{12}\phi(v)e_1 \right)^\circ
=\dgal{\phi_f(v e_{21}) , \phi_f(e_{12} w e_{21})}_{\fus}^\circ .
\end{align*}

\underline{$a=e_{12} w e_{21}$, $b= e_{12} \tilde w e_{21}$ of fourth kind.}
We use  \cite[(2.17d)]{F2}:
\begin{align*}
\phi_f^{\otimes 2}(\dgal{e_{12}w e_{21},e_{12}\tilde w e_{21}})
=& e_{12}\ast e_{12}(\phi^{\otimes 2}(\dgal{w,\tilde w}))e_{21}\ast e_{21} \\
=&(e_{12}\ast e_{12} \dgal{\phi(w),\phi(\tilde w)} e_{21}\ast e_{21})^\circ
= \dgal{\phi_f(e_{12}w e_{21}),\phi_f(e_{12}\tilde w e_{21})}^\circ ;
\\
\phi_f^{\otimes 2}( \dgal{e_{12} w e_{21},e_{12}\tilde w e_{21}}_{\fus}) =&\, 0
=\dgal{\phi_f(e_{12} w e_{21}) , \phi_f(e_{12}\tilde w e_{21})}_{\fus}^\circ .
\end{align*}

There is no other case to check due to the  cyclic antisymmetry of the double bracket.

\subsection{Comparison of representation algebras}
\label{App:RepHB}

Consider the Hopf algebra $B=\CC[\Orm_N]$ or $\CC[\Sp_{N}]$ (for $N$ even).
As a commutative algebra, it admits the presentation
\begin{equation} \label{Eq:DefB}
  \begin{aligned}
    B&=\CC[t_{ij} \mid i,j=1,\ldots,N]/I_B, \\
I_B&:= \Big\langle\sum_{k=1}^N \sgn_N(\tau(k),j)\, t_{i\tau(k)}t_{\tau(j) k} - \delta_{ij} 1_B \Big| i,j=1,\ldots,N \Big\rangle.
  \end{aligned}
\end{equation}
Here, we use the notations of \ref{sss:Not-symp} for $\Sp_{N}$, while $\tau=\id$ and $\sgn_N \equiv 1$ for $\Orm_N$.
The coalgebra structure and the antipode are given by
\[
 \Delta_B(t_{ij})=\sum_{1\leq k\leq N} t_{ik}\otimes t_{kj}, \quad \epsilon_B(t_{ij})=\delta_{ij}, \quad
 s_B(t_{ij})= \sgn_N(i,j)\, t_{\tau(j),\tau(i)}\,. 
\]
Fix a cocommutative Hopf algebra $\cH$ as in Section~\ref{Sec:MT}.
The representation algebra $\cH_B$ of Massuyeau and Turaev \cite{MT18} is the unital commutative algebra generated by symbols $\{x_{b} \mid x\in \cH,\, b\in B\}$ subject to the relations
(for $x,y\in \cH$, $b,c\in B$, $\lambda \in \CC$)
\begin{subequations}
 \begin{align}
  &(\lambda x)_b=x_{\lambda b}=\lambda\, x_b,& \quad &(x+y)_b=x_b+y_b,& \quad x_{b+c}=x_b+x_c, \label{Eq:MT-a} \\
  &(xy)_b = x_{b_{[1]}} y_{b_{[2]}},& \qquad &(1_{\cH})_b =\epsilon_B(b)\, 1,& \label{Eq:MT-b} \\
  &x_{bc} = (x_{[1]})_b (x_{[2]})_c,& \qquad &x_{1_B} = \epsilon(x)\, 1,& \label{Eq:MT-c} \\
 &S(x)_b = x_{s_B(b)}.& \label{Eq:MT-d}
 \end{align}
\end{subequations}
Consider the algebra $\CC[\Rep^{S,\tau}(\cH,N)]$, cf. \ref{ss:Rep}, where $\tau$ depends on the type of $B$.
\begin{prop} \label{Pr:RepHB}
 The algebra $\CC[\Rep^{S,\tau}(\cH,N)]$  is isomorphic to $\cH_B$.
\end{prop}
\begin{proof}
If we write $x_{ij}:=x_{t_{ij}}$ for the generators $t_{ij}\in B$,
the first two relations in \eqref{Eq:MT-a} together with \eqref{Eq:MT-b}  are precisely the relations defining $\CC[\Rep(\cH,N)]$,
while \eqref{Eq:MT-d} is just the corresponding generator of the ideal \eqref{TwIdeal} defining $\CC[\Rep^{S,\tau}(\cH,N)]$.
Thus, we get a morphism of algebras
\begin{equation*}
\Psi_1 : \CC[\Rep^{S,\tau}(\cH,N)] \longrightarrow \cH_B, \quad x_{ij} \mapsto x_{t_{ij}} \,.
\end{equation*}
In the other direction, we define a morphism of algebras
\begin{align*}
 \Psi_2 : &\, \cH_B \longrightarrow \CC[\Rep^{S,\tau}(\cH,N)], \\
 &x_{1_B}\mapsto \epsilon(x)\, 1, \qquad  x_{t_{ij}} \mapsto x_{ij}, \\
 &x_b=\sum \lambda_{i_1,\ldots, i_{\ell+1}}\,  (x_{[1]})_{i_1 i_2} (x_{[2]})_{i_2 i_3} \cdots  (x_{[\ell]})_{i_\ell i_{\ell+1}} \quad \text{if }
 b= \sum \lambda_{i_1,\ldots ,i_{\ell+1}}\, t_{i_1 i_2} t_{i_2 i_3} \cdots t_{i_\ell i_{\ell+1}}\,,
\end{align*}
where in the last case we use an arbitrary expression for writing $b$ in terms of the generators of $B$.
The map $\Psi_2$ is well-defined provided that $\Psi_2(x_b)=0$ whenever $b\in I_B$.
This holds because, if we denote $\mathtt{b}_{ij}$ the generator of $I_B$ in \eqref{Eq:DefB} corresponding to the pair of indices $(i,j)$,
\begin{align*}
\Psi_2(x_{\mathtt{b}_{ij}})&=
\sum_{k} \sgn_N(\tau(k),j)\, \Psi_2((x_{[1]})_{t_{i\tau(k)}})\Psi_2((x_{[2]})_{t_{\tau(j) k}}) - \delta_{ij} \Psi_2(x_{1_B}) \\
&=
\sum_{k} \sgn_N(\tau(j),k)\, (x_{[1]})_{i\tau(k)} (x_{[2]})_{\tau(j) k} - \delta_{ij}\, \epsilon(x) \\
&=
\sum_{k}  (x_{[1]})_{i\tau(k)} S(x_{[2]})_{\tau(k) j} - \delta_{ij} \epsilon(x) \\
&=(x_{[1]} S(x_{[2]}))_{ij} - \delta_{ij} \epsilon(x) = 0\,.
\end{align*}
where we used \eqref{Eq:sgnProp} in the second equality and \eqref{Eq:DefRelSp} (with $\phi=S$) in the third.
To conclude, it suffices to verify that $\Psi_1$ and $\Psi_2$ are inverses of each other.
\end{proof}


\end{document}